%% file: main.tex
\documentclass[11pt]{article}
\usepackage[utf8]{inputenc}
\usepackage{amsmath}
\usepackage{amsthm}
\usepackage{amssymb}
\usepackage{url}
\usepackage{authblk}
\usepackage{constants}
\usepackage{graphicx}
\usepackage[backend=bibtex,maxnames=4,url=false,
  style=numeric,isbn=false,giveninits=true,eprint=false,
  sorting=none,doi=false,date=year]{biblatex}
\usepackage{float}
\usepackage{hyperref}
\usepackage{subcaption}
\usepackage{commath}
\usepackage{microtype}
\usepackage{cleveref}
\crefname{prop}{Proposition}{Propositions}
\usepackage{todonotes}
\setuptodonotes{inline}

\bibliography{references,IP}

\input{definitions.tex}

\newcommand{\PSF}{\text{PSF}}

\hypersetup{
   colorlinks=true,%
   citecolor=black,%
   filecolor=black,%
   linkcolor=black,%
   urlcolor=blue
}

\theoremstyle{plain}
\newtheorem{thm}{\protect\theoremname}
  \theoremstyle{plain}
  \newtheorem{lem}[thm]{\protect\lemmaname}
  \theoremstyle{plain}

  \theoremstyle{plain}
  \newtheorem{defn}[thm]{\protect\definitionname}
  \theoremstyle{plain}
  \newtheorem{prop}[thm]{\protect\propositionname}
  \theoremstyle{definition}
  
  \newtheorem{remark}{Remark}
  \theoremstyle{plain}
  
  \newtheorem{assumption}[thm]{\protect\assumptionname}

  \providecommand{\definitionname}{Definition}
  \providecommand{\examplename}{Example}
  \providecommand{\lemmaname}{Lemma}
  \providecommand{\corollaryname}{Corollary}
  \providecommand{\propositionname}{Proposition}
  \providecommand{\conditionsname}{Conditions}
  \providecommand{\theoremname}{Theorem}
  \providecommand{\assumptionname}{Assumption}

\counterwithin{equation}{section}

\renewbibmacro{in:}{%
  \ifentrytype{article}
    {}
    {\bibstring{in}%
     \printunit{\intitlepunct}}}
\DeclareFieldFormat{pages}{#1}

\makeatletter
\DeclareCiteCommand{\fullcite}
  {\defcounter{maxnames}{1}%
    \usebibmacro{prenote}}
  {\usedriver
     {\DeclareNameAlias{sortname}{default}}
     {\thefield{entrytype}}}
  {\multicitedelim}
  {\usebibmacro{postnote}}
\makeatother

\begin{document}

\title{Image reconstruction in light-sheet microscopy:
  spatially varying deconvolution and mixed noise}

\author[1,3,4]{Bogdan Toader}
\author[2]{J\'{e}r\^{o}me Boulanger}
\author[3]{Yury Korolev}
\author[1,5]{Martin O. Lenz}
\author[2]{James Manton}
\author[3]{Carola-Bibiane Sch\"{o}nlieb}
\author[1,4,5]{Leila Mure\c{s}an}
\affil[1]{Cambridge Advanced Imaging Centre, University of
  Cambridge, UK}
\affil[2]{MRC Laboratory of Molecular Biology, UK}
\affil[3]{Department of Applied Mathematics and
    Theoretical Physics, University of Cambridge, UK}
\affil[4]{Department of Physiology, Development and Neuroscience, University of Cambridge, UK}
\affil[5]{Sainsbury Laboratory, University of Cambridge, UK}

    
\date{}

\maketitle

\begin{abstract}
We study the problem of deconvolution for light-sheet
microscopy, where the data is corrupted by spatially
varying blur and a combination of Poisson and Gaussian
noise. The spatial variation of the point spread function
(PSF) of a light-sheet microscope is determined by the
interaction between the excitation sheet and the detection
objective PSF. First, we introduce a model of the image
formation process that incorporates this interaction,
therefore capturing the main characteristics of this imaging
modality. Then, we formulate a variational model that
accounts for the combination of Poisson and Gaussian noise
through a data fidelity term consisting of the infimal
convolution of the single noise fidelities, first introduced
in~\fullcite{calatroni2017infimal}. We establish convergence
rates in a Bregman distance under a source condition for the
infimal convolution fidelity and a discrepancy principle for
choosing the value of the regularisation parameter. The
inverse problem is solved by applying the primal-dual hybrid
gradient (PDHG) algorithm in a novel way. Finally, numerical
experiments performed on both simulated and real data show
superior reconstruction results in comparison with other
methods.
\end{abstract}

\tableofcontents


\input{1_introduction}

\input{2_forward_model}

\input{3_inverse_problem}

\input{4_solving}

\input{5_results}

\input{6_conclusion}

\section{Acknowledgements}
BT and LM gratefully acknowledge the funding by 
Isaac Newton Trust/Wellcome Trust ISSF/University of Cambridge
Joint Research Grants Scheme and EPSRC EP/R025398/1.
MOL and LM also thank the Gatsby Charitable Foundation for financial support. 
YK acknowledges financial support of the EPSRC (Fellowship EP/V003615/1), the Cantab Capital Institute for the Mathematics of Information at the University of Cambridge and the National Physical Laboratory. 
CBS acknowledges support from the Philip Leverhulme Prize, the Royal Society Wolfson Fellowship, the EPSRC grants EP/S026045/1 and EP/T003553/1, EP/N014588/1, EP/T017961/1, the Wellcome Innovator Award RG98755, the Leverhulme Trust project Unveiling the invisible, the European Union Horizon 2020 research and innovation programme under the Marie Sk{\l}odowska-Curie grant agreement No. 777826 NoMADS, the Cantab Capital Institute for the Mathematics of Information and the Alan Turing Institute.

Imaging was performed at the
Microcopy Facility of the Sainsbury Laboratory Cambridge University. 
We thank Dr. Alessandra Bonfanti and Dr. Sarah Robinson for
providing the Marchantia sample and Prof. Sebastian
Schornack and Dr. Giulia Arsuffi
(Sainsbury Laboratory Cambridge University) 
for provision of the line of Marchantia used.

We also acknowledge the support of NVIDIA Corporation with the donation of two Quadro P6000, a Tesla K40c and a Titan Xp GPU used for this research.

\printbibliography


\end{document}

%% file: definitions.tex
\newcommand{\defeq}{:=}

\newcommand{\bigO}{O}

\renewcommand{\leq}{\leqslant}
\renewcommand{\geq}{\geqslant}
\renewcommand{\phi}{\varphi}

\let\sp\relax
\newcommand{\sp}[1]{\langle #1 \rangle}

\newcommand{\subdiff}{\partial}

\DeclareMathOperator*{\argmin}{arg\,min}

\newcommand{\reg}{\mathcal J}

\newcommand{\dJ}{\subdiff \reg}

\newcommand{\Jminsol}{u^\dagger_{\TV}}

\DeclareMathOperator{\prox}{prox}

\DeclareMathOperator*{\supp}{supp}
\DeclareMathOperator{\vecspan}{span}

\let\ker\relax
\newcommand{\ker}[1]{\mathcal N(#1)}

\let\div\relax
\DeclareMathOperator{\div}{div}

\let\L\relax
\DeclareMathOperator{\L}{\mathrm{L}}
\DeclareMathOperator{\BV}{BV}

\DeclareMathOperator{\TV}{TV}

\newcommand{\R}{\mathbb R}



%% file: 1_introduction.tex
\section{Introduction}

Light-sheet microscopy is a fluorescence microscopy 
technique that enables volumetric
imaging of biological samples at high frame rate 
with better sectioning and lower photo-toxicity in
comparison to other fluorescent techniques.
This is achieved by illuminating a thin slice of the sample
using a sheet of light and detecting the emitted
fluorescence from this plane with another objective
perpendicular to the plane of the sheet.
A schematic representation of a light-sheet microscope is
shown in Figure \ref{fig:microscope}.
By contrast, widefield microscopy illuminates the whole
sample using a single objective and achieves only very
limited sectioning. Confocal microscopy allows 
improved sectioning by utilising a pinhole to discard 
out-of-focus light, at the cost of higher photo-toxicity and
reduced frame rate.
By only selectively illuminating the slice of the sample
being imaged, less photo-toxicity damage is induced in
light-sheet microscopy and, therefore, imaging of living 
samples over a longer period of time is possible.
The combination of lower photo-toxicity, better sectioning
capabilities and faster image acquisition led to
light-sheet microscopy being recognised as
``Method of the Year'' by Nature Methods 
in 2014~\cite{nmeth2015}. 

\begin{figure}[h]
  \centering
  \includegraphics[width=0.9\textwidth]{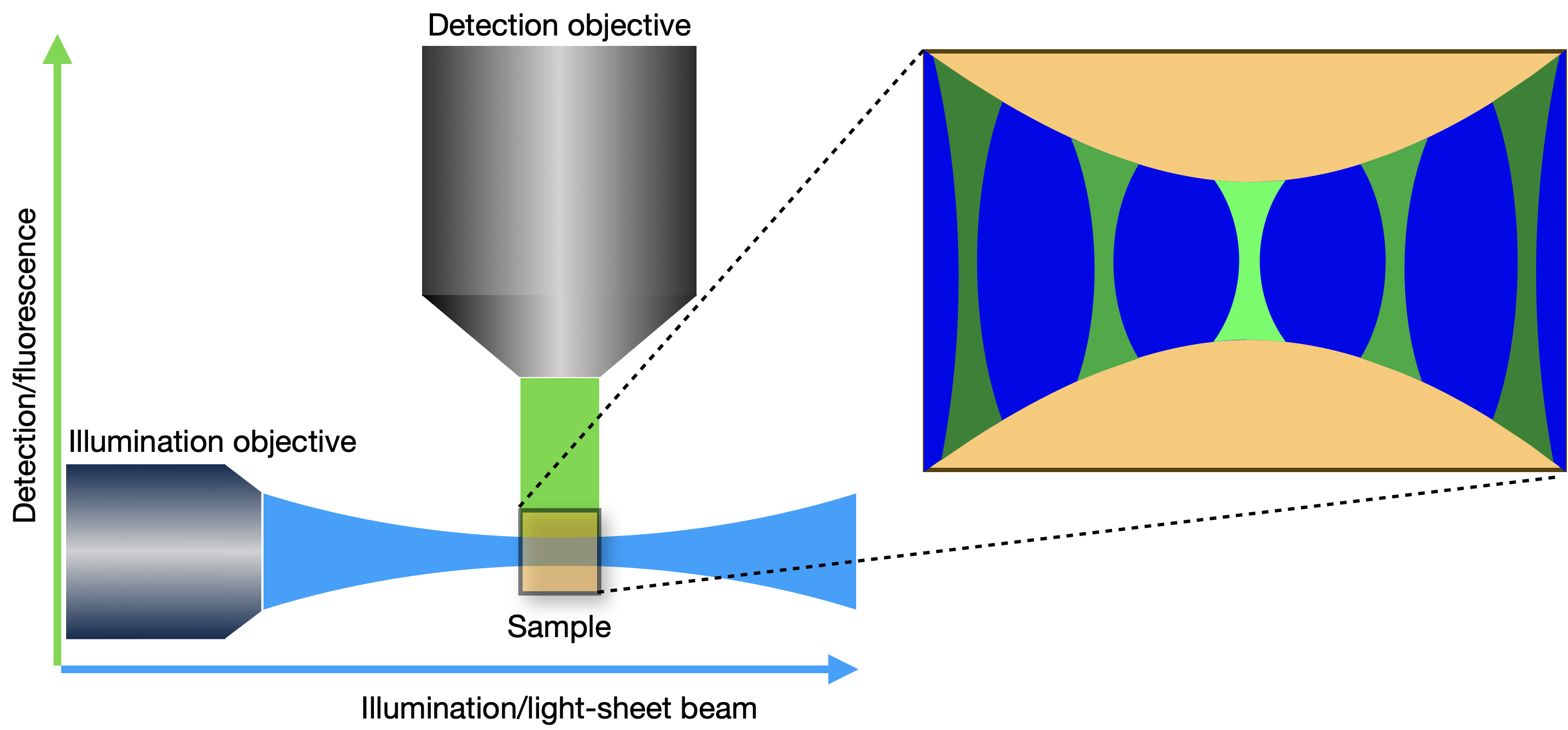}
  \caption{Schematic of a light-sheet microscope, showing
  the illumination and the detection directions.
  The interaction of the light-sheet with 
  the detection PSF leads to a spatially varying 
  overall PSF and decreasing of the pixel intensities
  away from the centre in the horizontal direction.}
  \label{fig:microscope}
\end{figure}

The focus of the present manuscript is on deconvolution 
techniques for light-sheet microscopy data. In this context,
deconvolution refers to the computational method of
reversing the effect of blurring in the image acquisition
process due to the point spread function (PSF) of the 
microscope \cite{McNally1999,Starck2002,Sarder2006}. 
Specifically, the PSF of an imaging system 
represents its response to a point object. Knowledge of the 
PSF, which can be modelled mathematically and calibrated
using bead data (samples containing small spheres of known
dimensions),
is used in the formulation of a forward
model of the image formation, which can then be inverted,
for example using optimisation methods, to reconstruct
the original, deblurred object \cite{Chambolle2016}.

In the case of light-sheet microscopy, simply knowing 
or estimating the PSF of the detection objective is not
sufficient, since the overall response of the system
to a point source is also influenced by the excitation 
light-sheet used to illuminate the slice. The overall PSF
could be approximated by the detection PSF in the region
where the illumination sheet is focused. However, the 
detection PSF becomes more distorted and loses intensity
away from the focus of the excitation light-sheet,
which is illustrated in Figure \ref{fig:microscope}. 
Therefore, the problem we propose to address can be seen as a specific
case of spatially-varying deconvolution \cite{Loic2015,Debarnot2019}. 

Two examples of acquired data are shown in Figure 
\ref{fig:data example}. We can see in both cases the 
effect of the spatially varying light-sheet: 
the image is sharper in the 
centre and blurry on the sides, with the amount of blur
growing with the horizontal distance from the centre.
In addition, the fluorescence intensity of imaged beads 
in Figure~\ref{fig:data example}(a) is
unevenly distributed despite imaging a homogeneous sample of
beads, with the centre of the image
being brighter than the left and right sides.
The aim of our work is to correct these effects.

\begin{figure}[h]
  \centering 
  \begin{minipage}[b]{0.46\linewidth}
    \centering
    \includegraphics[width=\textwidth]{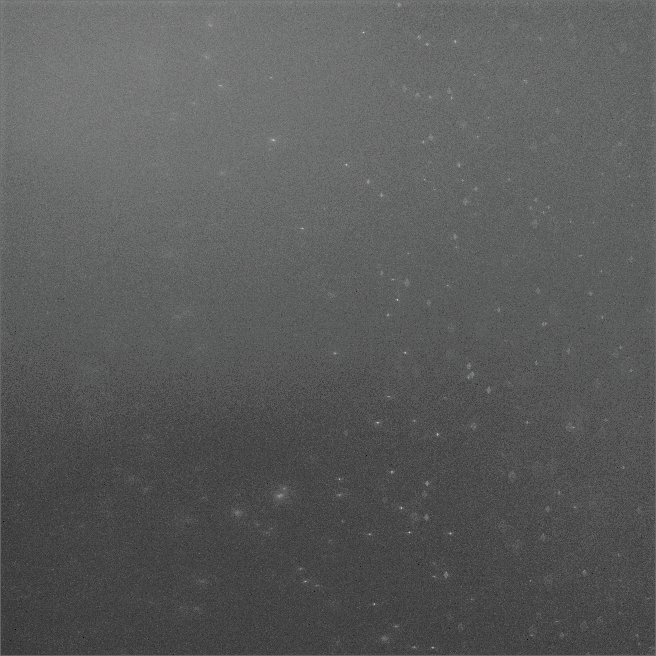}
    (a) $0.5 \mu m$ multi-colour Tetraspeck microspheres
    (slice)
  \end{minipage}
  \begin{minipage}[b]{0.46\linewidth}
    \centering
    \includegraphics[width=\textwidth]{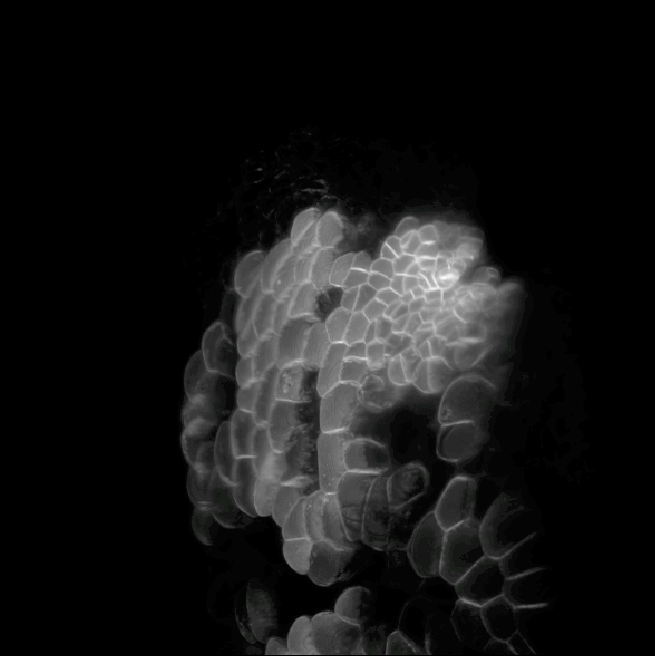}
    (b) Membrane labelled Marchantia thallus
    (maximum intensity projection)
  \end{minipage}

  \caption{Examples of light-sheet microscopy data of
    dimensions $665.6 \mu m \times 665.6 \mu m$: beads in
    (a) and Marchantia thallus in (b).
  The effect of the light-sheet is visible along the
  horizontal direction, as the image is sharp in the centre,
  where the sheet is focused, while the quality of the image
  decreases away from the centre. Another source of blur 
  observed, especially in the bead image (left) is given
  by optical aberrations due to the sample imaging medium 
  (dynamic range is stretched for better visibility). 
  The Marchantia image has been acquired using samples 
  from Dr. Alessandra Bonfanti and Dr. Sarah Robinson 
  using the genetic line provided by
  Prof. Sebastian Schornack and Dr. Giulia Arsuffi
  at the Sainsbury Laboratory Cambridge University.}
  \label{fig:data example}
\end{figure}

\subsection{Contribution}

We propose a method for deconvolution of 3D light-sheet  
microscopy data that takes into account the spatially 
varying nature of the PSF and is scalable to the dimensions 
typical to biological samples imaged using light-sheet microscopy -- 
4.86GB per 3D 16-bit stack of $2048 \times 2048 \times 580$ voxels.

Our approach is based on a new model for image formation 
that describes the interaction between the light-sheet 
and the detection PSF which replicates the physics
of the microscope. Then, we formulate an inverse problem
where the forward operator is given by model of the image formation
process and which takes into account the degradation of the data by 
both Gaussian and Poisson noise as an infimal convolution between an $\L^2$ term and a Kullback-Leibler divergence term, following~\cite{calatroni2017infimal}. The proposed variational problem is solved by
applying the Primal Dual Hybrid Gradient (PDHG) algorithm
in a novel way. Finally, we exploit the noise model to automatically tune the balance between the data fidelity and regularisation resorting to a discrepancy principle. We obtain convergence rates in a Bregman distance for the infimal convolution fidelity from~\cite{calatroni2017infimal} under a standard source condition.

In our numerical experiments, we first show how this method performs on simulated data, 
where the ground truth is known, then we apply our method
to two examples of data from experiments: an image of fluorescent 
beads and a sample of Marchantia. In both cases, we see that the
deconvolved images show improved contrast, while outperforming deconvolution using only the constant detection PSF.

\subsection{Related work}

Before describing in more detail our 
approach to the deconvolution problem,
we give a brief overview of the literature
on spatially varying deconvolution in the context
of microscopy and how our work relates to it.

Purely data-driven approaches estimate a spatially
varying PSF in a low dimensional space 
(for scalability reasons) using bead images 
\cite{Nagy1998,Loic2015,Debarnot2019}.
This is usually not application specific and can be included in a more 
general blind deconvolution framework.
Similarly, the work in \cite{Hadj2014} involves writing
the spatially varying PSF as a convex combination 
of spatially invariant PSFs. The algorithm alternates
between estimating the image and estimating the PSF.
In a similar vein, the authors of \cite{Hirsch2010} 
approach the problem of blind 
deconvolution by defining the
convolution operator using efficient 
matrix-vector multiplication operations. This decomposition
is similar to the discrete formulation of 
our image formation model.
These methods optimise over the (unknown) operator
in addition to the unknown image.
Related to these results is \cite{OConnor2017},
where the authors consider the
models from \cite{Nagy1998} and \cite{Hirsch2010}
under the assumption that the blurring operator is
known and given as a sum of weighted spatially invariant
operators. They exploit this structure of the 
operator and use a Douglas-Rachford based splitting 
to solve the optimisation problem efficiently.
While more general than our approach, 
we consider that using the knowledge
of the image formation process in the forward model
is advantageous for the reconstruction of 
light-sheet microscopy data.

A number of groups consider the problem of reconstruction
from multiple views in the context of light-sheet microscopy.
In \cite{Temerinac-Ott2012}, the problem of multi-view
reconstruction under a spatially varying blurring operator
for 3D light-sheet data is considered. They divide the
image into small blocks where they perform deconvolution
using spatially-invariant PSFs estimated from beads 
(and interpolated PSFs in regions where there are
no beads). 
In \cite{Preibisch2014}, the authors extend the
Richardson-Lucy algorithm to the multi-view reconstruction
problem in a Bayesian setting. While it allows for different 
PSFs for each view (estimated using beads), 
this work does not consider spatial variations of the PSF.
While using data from multiple views improves the quality
of the reconstruction, these approaches are agnostic to
the physics of the microscope.

Taking an approach similar in spirit 
to ours, the authors of \cite{Becker2019} model the
effective PSF of a light-sheet microscope, which is then
plugged into a regularised version of the 
Richardson-Lucy algorithm for deconvolution. However, 
while they model the detection PSF and the light-sheet 
separately, they assume the effective PSF
of the microscope is spatially invariant and the 
point-wise product of the two PSFs. In contrast,
we do not take this simplifying step in our modelling,
as we consider that the relationship between the two PSFs
plays in important role in the resulting blur of the
image.

The work of Guo et al. in \cite{Guo2020} uses a modified
Richardson-Lucy algorithm implemented on GPU to improve 
the speed of convergence, further improved by the use of 
a deep neural network, which is a promising approach.

Moreover, in \cite{zhang2020} the authors
introduce an image formation model similar to the one
described in the present manuscript. 
However, the regions of the resulting PSF where the
light-sheet is out of focus are discarded, hence approximating
the overall PSF with a constant PSF and then performing 
deconvolution using the ADMM algorithm.
In Cueva et al. \cite{Cueva2020},
a mathematical model which takes into account image
fusion with two sided illumination
is derived from fist principles. However, it is 
restricted to 2D and they do not apply the method
to real data.

Lastly, regarding the mixed Gaussian-Poisson noise fidelity,
our method follows the infimal convolution variational
approach described in~\cite{calatroni2017infimal},
with the additional light-sheet blurring operator. 
The same inverse problem, without the blurring operator, 
is solved in~\cite{zhang2021} albeit using an ADMM algorithm for 
the minimisation. 

\subsection{Paper structure}

The paper is organised as follows. In Section~\ref{sec:fm}, 
we introduce a mathematical model of the image formation 
process in a light-sheet microscope. This model describes 
how the sample is blurred by the excitation illumination together 
with the detection objective PSF. Optical aberrations of 
the system are modelled using Zernike polynomials in the 
detection PSF, which we discuss in Section~\ref{sec:psf}.
In Section~3, we define the mathematical setting for the
deconvolution problem and we state an inverse problem using
a data fidelity as an infimal convolution of the individual 
Gaussian and Poisson data fidelities. We discuss
convergence rates and a discrepancy principle for choosing
the regularisation parameter in Section~\ref{sec:ifm}.
In Section~\ref{sec:solving min problem}, we describe 
how PDHG is applied to this inverse problem, with 
details of the implementation of the proximal
operator and the convex conjugate of the joint
Kullback-Leibler divergence. Finally, we validate our method
with numerical experiments both with simulated and real
data in Section~\ref{sec:numerics}, before concluding and
giving a few directions for future work in
Section~\ref{sec:conclusion}.

%% file: 2_forward_model.tex
\section{Forward model}
\label{sec:fm}

The first contribution of the current work 
is a model of the image formation 
process in light-sheet microscopy. By modelling
the excitation light-sheet and the detection PSF separately
and their interaction in a way that replicates the 
physics of the microscope, we are able to accurately
simulate the spatially varying PSF of the 
imaging system. We then incorporate this knowledge 
as the forward model in an inverse problem, which
we solve to remove the noise and blur in 
light-sheet microscopy data.
In this section, we describe the image formation 
process and the PSF model.

\subsection{Image formation model}
\label{sec:ifm}

A light-sheet propagated along the $x$ direction
is focused by the excitation objective at an
axial position $z=z_0$ and the local light-sheet intensity $l$ is 
modelled by the incoherent point spread function (PSF) of
the excitation objective. The sample with local density of 
fluorophores $u$ emits photons proportionally to the local
intensity $l$ of the light-sheet. These photons are then 
collected by a detection objective, whose action on the 
illuminated sample is modelled as a convolution with its 
PSF $h$. 
Finally, the sensor conjugated with the image plane $z_0$ 
collects photons and converts them to digital values for 
storage. Consequently, the recorded image is corrupted
by a combination of Gaussian and Poisson noise.
We can see here again how the local variation of 
the light-sheet will result into a spatially varying
blur and spatially-varying illumination intensity in the 
captured image.
This process is then repeated for each $z_0$ to obtain
the measured data $f$.

More specifically, we model $u$, $f$, $l$ and $h$
as functions defined on $\Omega \subset \mathbb{R}^3$, 
a rectangular domain of dimensions 
$\Omega_x \times \Omega_y \times \Omega_z$ (in $\mu m$) with
$\Omega = [-\frac{\Omega_x}{2}, \frac{\Omega_x}{2}] \times [-\frac{\Omega_y}{2},\frac{\Omega_y}{2}] 
\times [-\frac{\Omega_z}{2},\frac{\Omega_z}{2}]$.
For the sample $u$, the light-sheet $l$ and the 
detection objective PSF $h$, the measured data $f$ 
is given by:
\begin{equation}
  f(x,y,z) = \iiint l(s,t,w) u(s,t,w - z) h(x-s,y-t,w) \dif s \dif t \dif w.
  \label{eq:lightsheet model}
\end{equation}
The detection PSF $h$ is given by
\begin{equation}
  h(x,y,z) = \left|
    \iint g_{\sigma} * p_{Z}(\kappa_x,\kappa_y) e^{
      2 i \pi z 
      \sqrt{(n/\lambda_h)^2 - \kappa_x^2 - \kappa_y^2}
    }
    e^{
      2 i \pi (\kappa_x x + \kappa_y y)
    }
    \dif \kappa_x \dif \kappa_y
  \right|^2
  \label{eq:detection PSF}
\end{equation}
and the light-sheet $l$ is the y-averaged light-sheet 
beam PSF $l_{beam}$:
\begin{equation}
  l_{beam}(x,y,z) = \left|
    \iint p_0(\kappa_z,\kappa_y) e^{
      2 i \pi x 
      \sqrt{(n/\lambda_l)^2 - \kappa_z^2 - \kappa_y^2}
    }
    e^{
      2 i \pi (\kappa_z z + \kappa_y y)
    }
    \dif \kappa_z \dif \kappa_y
  \right|^2,
  \label{eq:beam PSF}
\end{equation}
where $n$ is the refractive index, $\lambda_h,\lambda_l$ are the wave 
lengths corresponding to the detection objective and light-sheet beam
respectively,
and $g_{\sigma}$ represents Gaussian blur.
Lastly, $p_Z(\kappa_x,\kappa_y)$ and $p_0(\kappa_z,\kappa_y)$
are the pupil functions for the detection PSF and the light-sheet beam 
respectively, both given by:
\begin{equation}
  p_{\phi}(x,y) = \begin{cases}
    e^{2i\pi \phi} 
    \quad &\text{for} \quad \sqrt{x^2+y^2} \leq NA/\lambda_i,
    \\
    0
    \quad &\text{otherwise},
  \end{cases}
\end{equation}
for their respective $\lambda_i = \lambda_h$ or
$\lambda_i = \lambda_l$,
where the phase for the light-sheet pupil $p_0$ is
equal to zero and the phase for the detection PSF 
pupil $p_Z$ is an approximation of the optical 
aberrations written as an expansion in a Zernike polynomial
basis, which we will explain in more detail
in Section \ref{sec:psf},
and with different numerical apertures NA. 
In general, the NA of the excitation sheet is much lower 
than the NA of the detection lens.
We note that the overall process is not translation 
invariant and cannot be modelled by a convolution operator.

Note that both the detection PSF $h$ and the light-sheet 
PSF have a similar formulation derived from:
\begin{equation}
  \PSF(x,y,z) = \left|
    \iint p(\kappa_x,\kappa_y) e^{
      2 i \pi z 
      \sqrt{(n/\lambda_i)^2 - \kappa_x^2 - \kappa_y^2}
    }
    e^{
      2 i \pi (\kappa_x x + \kappa_y y)
    }
    \dif \kappa_x \dif \kappa_y
  \right|^2,
  \label{eq:psf}
\end{equation}
which includes the pupil function for modelling aberrations
and a defocus term before taking the Fourier transform
(see, for example \cite{Stokseth1969,Soulez2013}).

In practice, the image formation process modelled 
by \eqref{eq:lightsheet model} is discretised at the
point of recording by the camera sensor in the $xy$ plane
and by the step size of the light-sheet in the $z$ 
direction. 
If the camera has a resolution of $N_x \times N_y$ pixels
and the light-sheet illuminates the sample at $N_z$ 
distinct steps, 
the model \eqref{eq:lightsheet model} becomes:
\begin{equation}
  \tilde{f}_{i,j,k} = \frac{1}{\tilde{C}}
    \sum_{i'=1}^{N_x} \sum_{j'=1}^{N_y} \sum_{k'=1}^{N_z}
    \tilde{l}_{i',j',k'} \tilde{u}_{i',j',k' - k} \tilde{h}_{i-i',j-j',k'},
  \label{eq:lightsheet model discrete}
\end{equation}
for all $i = 1,\ldots,N_x, j = 1,\ldots,N_y,k = 1,\ldots,N_z$,
and a normalisation constant $\tilde{C}$,
where $\tilde{u},\tilde{f},\tilde{l},\tilde{h} \in \mathbb{R}^{N_x \times N_y \times N_z}$
are the discretised versions of $u,f,l,h$ respectively. 
Similarly, the sampling performed by the camera sensor
leads to a discretisation of the Fourier space and the use
of the discrete Fourier transform in the
PSF and light-sheet models \eqref{eq:detection PSF}
and \eqref{eq:beam PSF}.
Lastly, in our implementation we normalise
$\tilde{h}$ so that 
$\sum_{i=1}^{N_x}\sum_{j=1}^{N_y}\sum_{k=1}^{N_z} \tilde{h}_{i,j,k} = 1$ 
and choose the normalisation constant $\tilde{C}$
so that the norm of the
resulting operator is equal to one.

\subsection{Derivation of the model}

Let $l, u, h$ be defined as in Section~\ref{sec:ifm},
with $h$ symmetric around the origin
and $l$ translation invariant in the $y$ direction, 
centred and symmetric in the $xz$ plane around the origin.
For a fixed $z_0 \in [-\frac{\Omega_z}{2},\frac{\Omega_z}{2}]$, we
take the following steps, which replicate the inner
workings of a light-sheet microscope:
\begin{enumerate}
  \item  
    Image the sample at $z = z_0$:
    centre the sample $u$ at $z_0$ and multiply 
    the result with the light-sheet $l$:
    \begin{equation}
      F(x,y,z;z_0) = l(x,y,z) \cdot u(x,y,z - z_0),
    \end{equation}
  \item Convolve with the objective PSF $h$:
    \begin{align}
      C(x,y,z;z_0) &= F(x,y,z;z_0) * h(x,y,z) \\
      &=
      \iiint F(s,t,w;z_0) h(x-s,y-t,z-w) \dif s \dif t \dif w
    \end{align}
  \item Slice at $z = 0$:
    \begin{equation}
      f(x,y,z_0) = \left[
        C(x,y,z;z_0)
      \right]_{z=0},
    \end{equation}
\end{enumerate}
which leads to:
\begin{equation}
  f(x,y,z_0) = \iiint l(s,t,w)u(s,t,w-z_0) h(x-s,y-t,-w) 
    \dif s \dif t \dif w.
\end{equation}
This is the same as model \eqref{eq:lightsheet model}, where we 
substitute $z$ for $z_0$ and note that $h$ is symmetric 
in the third variable around the origin.

For a discretisation of the domain using a 3D grid 
with $N_x \times N_y$ pixels and $N_z$ light-sheet steps,
the forward model can be computed by following the three
steps above for each $k=1,\ldots,N_z$, where we perform the convolutions
using the fast Fourier transform (FFT), resulting in a number of 
$\mathcal{O}(N_x N_y N_z^2 \log(N_x N_y N_z))$ operations.

Alternatively, we can re-write the last integral above as:
\begin{equation}
  f(x,y,z_0) = \int K(x,y,w) * h(x,y,w) \dif w,
  \label{eq:model alt}
\end{equation}
where 
\begin{equation}
  K(x,y,w) = l(x,y,w) u(x,y,w-z_0),
\end{equation}
and the convolution in \eqref{eq:model alt}
is a 2D convolution in $(x,y)$:
\begin{equation}
  K(x,y,w) * h(x,y,w) = 
  \iint 
    K(s,t,w) h(x-s,y-t,w) 
  \dif s \dif t.
\end{equation}
In terms of numbers of FFTs performed on a 
discretised $N_x \times N_y \times N_z$ grid,
this alternative formulation requires 
$\mathcal{O}(N_x N_y N_z^2 \log(N_x N_y))$
operations.


\subsection{Point spread function model}
\label{sec:psf}

While both the light-sheet profile and the detection PSF are 
based on the same model of a defocused 
system \eqref{eq:psf} introduced in \cite{Stokseth1969},
note that our definition of $h$ in \eqref{eq:detection PSF}
includes an additional convolution operation with 
a Gaussian $g_{\sigma}$ and a pupil function $p_Z$
with a non-zero phase. Let us turn to why this is the case.

It is well known that optical aberrations hamper results based 
on deconvolution with theoretical PSFs. In light-sheet microscopy, 
the effect of aberrations is more visible away from the centre, 
as shown for example in the bead image in Figure \ref{fig:data example},
or in the more detailed example beads in Figure \ref{fig:psfs data}.
It is, therefore, required that we model the (spatially
invariant) aberrations of the detection lens.


\begin{figure}[h]
  \begin{minipage}[t]{0.32\linewidth}
    \centering
    \includegraphics[width=\textwidth]{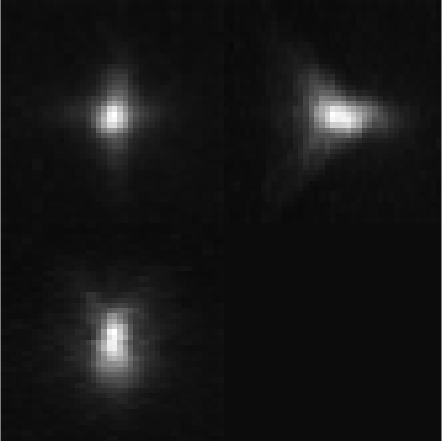}
    (a) Bead, no aberrations (maximum intensity projections)
  \end{minipage}
  \begin{minipage}[t]{0.32\linewidth}
    \centering
    \includegraphics[width=\textwidth]{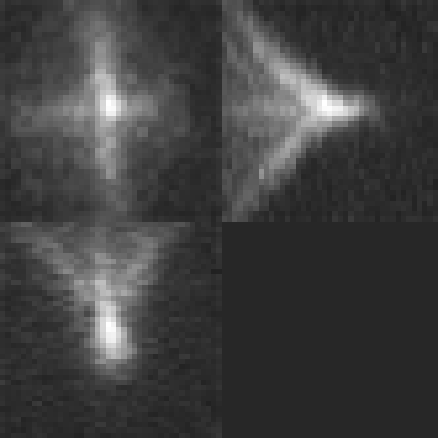}
    (b) Bead, with aberrations (maximum intensity
    projections)
  \end{minipage}
  \begin{minipage}[t]{0.32\linewidth}
    \centering
    \includegraphics[width=\textwidth]{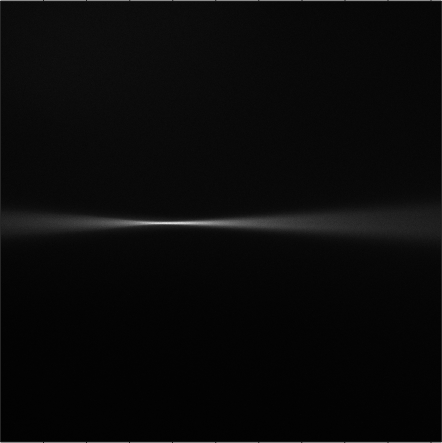}
    (c) Light sheet profile (slice showing an $x-z$ plane) 
  \end{minipage}

  \caption{Examples of beads and light-sheet profile. 
  The bead in (a) is cropped
  from the centre of Figure~\ref{fig:data example}(a) 
  and the bead in (b) is cropped from the right-hand side 
  of Figure~\ref{fig:data example}(a).
  The maximum intensity projections are taken
  in the $x-y$ plane (top left), the $z-y$ plane (top right)
  and the $x-z$ plane (bottom left).}
  \label{fig:psfs data}
\end{figure}

\begin{figure}[h]
  \centering
  \includegraphics[width=0.46\textwidth]{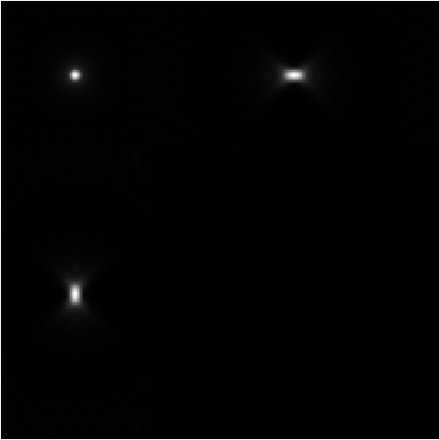}

  \caption{Objective PSF used in our model, 
    with no aberrations (maximum intensity projections taken
    in the same way as in Figure~\ref{fig:psfs data}).}
  \label{fig:simulated psf h}
\end{figure}

The general PSF model \eqref{eq:psf}, with the phase
of the pupil function equal to zero, does not take
optical aberrations into account and therefore it is not 
an accurate representation of the objective PSF $h$.
For example, a PSF calculated using \eqref{eq:psf} 
with zero phase of the pupil and the parameters of the
detection objective, shown
in Figure~\ref{fig:simulated psf h}, does not resemble
the actual bead images in the data 
in Figure~\ref{fig:psfs data}.

There has been extensive work on the problem of phase reconstruction
in the literature \cite{Dunn2011,Hanser2004,Paxman1992,Petrov2017}, 
but here we take a more straightforward approach 
using Zernike polynomials to include aberrations in 
the PSF~\cite{Wyant1992}, as follows.
Let $h_z$ be the objective
PSF calculated using \eqref{eq:psf} with Zernike polynomials
in the phase of the pupil function:
\begin{equation}
  h_z(x,y,z; c) = \left|
    \iint p_Z(\kappa_x,\kappa_y; c) e^{
      2 i \pi z 
      \sqrt{(n/\lambda_h)^2 - \kappa_x^2 - \kappa_y^2}
    }
    e^{
      2 i \pi (\kappa_x x + \kappa_y y)
    }
    \dif \kappa_x \dif \kappa_y
  \right|^2,
  \label{eq:psf zernike}
\end{equation}
where $p_z(\kappa_x,\kappa_y; c)$ is the pupil function with Zernike polynomials
in the phase:
\begin{equation}
  p_z(\kappa_x,\kappa_y; c) = \begin{cases}
    e^{2i\pi \sum_{j=1}^{15} c_j Z_j(\kappa_x,\kappa_y)}
      &\text{for} \quad
      \rho = \sqrt{\kappa_x^2 + \kappa_y^2} \leq NA/\lambda_h,\\
      0, \quad &\text{otherwise,}
  \end{cases}
  \label{eq:pupil zernike}
\end{equation}
and $ c = [c_1, \ldots,c_{15}]^T $ are coefficients 
corresponding to the polynomials.
The Zernike polynomials and the corresponding coefficients
that we use are given
in Table \ref{tab:zernike} and 
shown in Figure \ref{fig:zernike wiki}.

Moreover, let $h_{zb}$ be the blurred PSF obtained by
convolving $h_z$ with a Gaussian $g_{\sigma}$ 
with width $\sigma$:
\begin{equation}
  h_{zb}(x,y,z; c, \sigma) 
  = h_z(x,y,z; c) * g_{\sigma}.
  \label{eq:psf zernike blurred}
\end{equation}
This allows us to obtain a better approximation of the
objective PSF.
The parameters $c$ and $\sigma$ are calculated by solving
the least-squares problem
\begin{equation}
  \min_{c,\sigma} \| h_{zb}(c,\sigma) * b - h_{bead}  \|_2^2 
  \quad \text{subject to} \quad
  c \in [-3,3]^{15}, \sigma > 0,
\end{equation}
where $h_{bead}$ is the bead image 
from Figure \ref{fig:psfs data}b 
and $b$ is equal to one inside the sphere of the radius
equal to the radius of the bead (a parameter that is provided)
and zero outside the sphere. This takes into account the
non-negligible size of the beads used to generate the data.

In the implementation of
the fitting procedure, we normalise both the bead image $h_{bead}$
and the simulated PSF $h_{zb}$ by their maximum values
before calculating their error, and
we include two additional parameters, scaling and shift,
to ensure a better fit of the intensity values
(not shown here for simplicity of the presentation).
The resulting PSF is the detection PSF model
\eqref{eq:detection PSF} and is shown 
in Figure \ref{fig:psf h zernike}.

\begin{table}[H]
  \centering
  \begin{tabular}{| l | l l |}
    \hline
      $Z_j$ & Polynomial & $c_j$\\
    \hline 
      $Z_1$ & $\rho\cos\theta$ & -0.7763\\
      $Z_2$ & $\rho\sin\theta$ & -0.0460  \\
      $Z_3$ & $2\rho^2-1$ & -2.3608  \\
      $Z_4$ & $\rho^2\cos2\theta$ & -1.3001 \\
      $Z_5$ & $\rho^2\sin2\theta$ & 0.2024 \\
      $Z_6$ & $(3\rho^2-2)\rho\cos\theta $ & -0.3999 \\
      $Z_7$ & $(3\rho^2-2)\rho\sin\theta $ & 0.0348  \\
      $Z_8$ & $6\rho^4-6\rho^2+1 $ & -1.2112 \\
      $Z_9$ & $\rho^3\cos3\theta $ & -0.1521  \\
      $Z_{10}$ & $\rho^3\sin3\theta $ & -0.0466 \\
      $Z_{11}$ & $(4\rho^2-3)\rho^2\cos2\theta $ & -0.0930  \\
      $Z_{12}$ & $(4\rho^2-3)\rho^2\sin2\theta $ & 0.0427  \\
      $Z_{13}$ & $(10\rho^4-12\rho^2+3)\rho\cos\theta $ & -0.0117  \\
      $Z_{14}$ & $(10\rho^4-12\rho^2+3)\rho\sin\theta $ & -0.0581 \\
      $Z_{15}$ & $20\rho^6-30\rho^4+12\rho^2-1 $ & -0.0633  \\
    \hline
  \end{tabular}
  
  \caption{The first 15 Zernike Polynomials (in polar coordinates)
    and their coefficients used in $h_z$.}
  \label{tab:zernike}
\end{table}

\begin{figure}[H]
  \centering
  \begin{tabular}{c c c c}
    \includegraphics[width=0.1\textwidth]{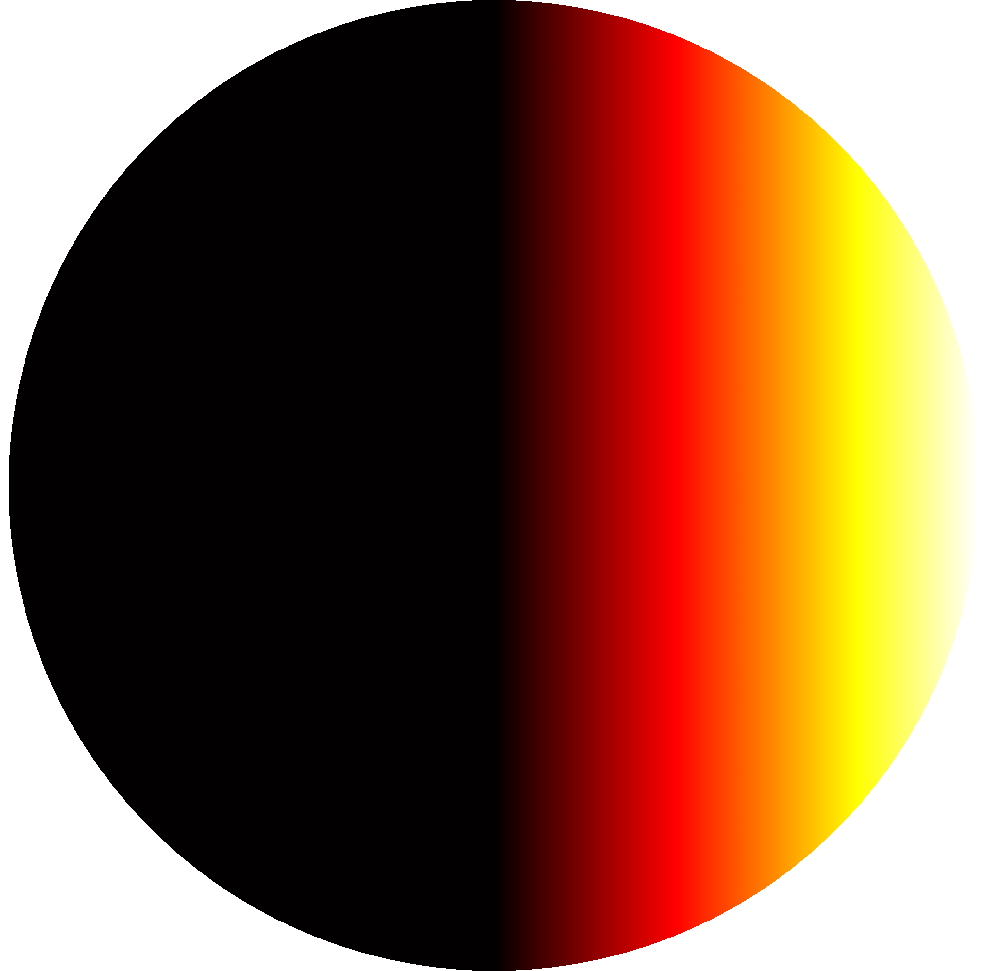}
    &
    \includegraphics[width=0.1\textwidth]{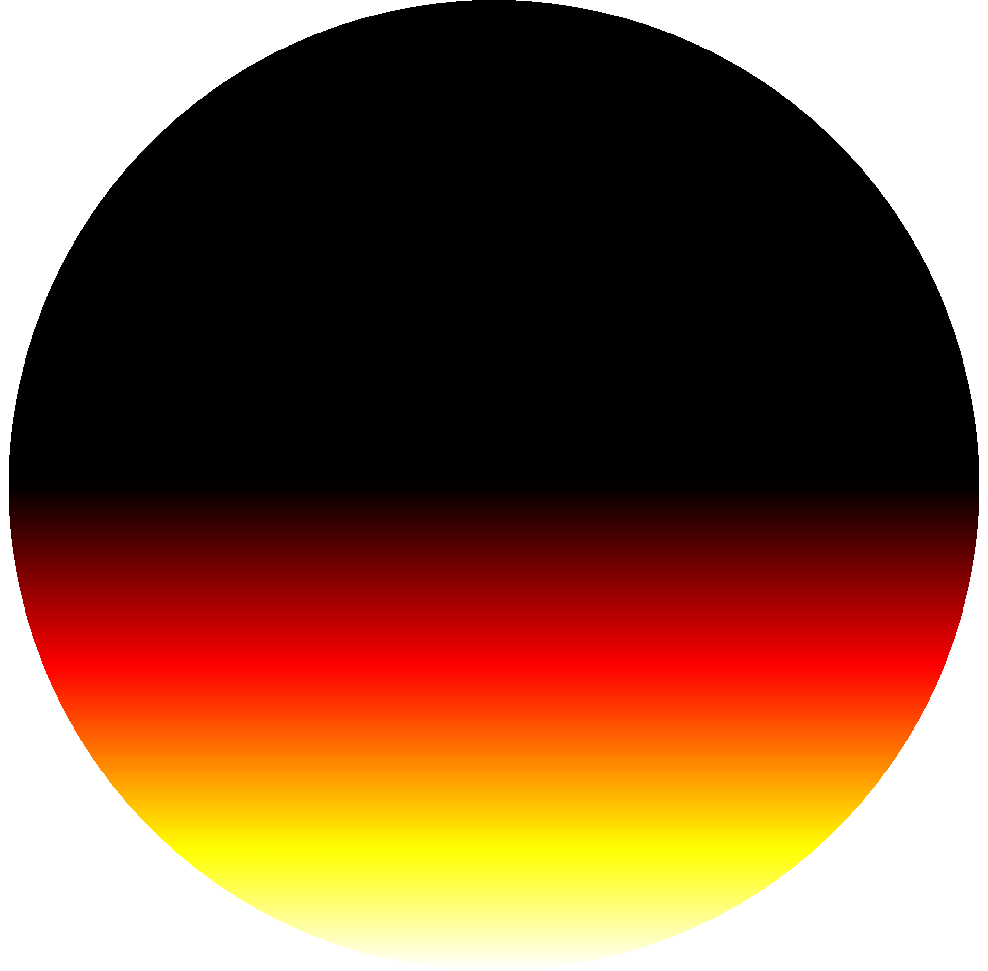}
    &
    \includegraphics[width=0.1\textwidth]{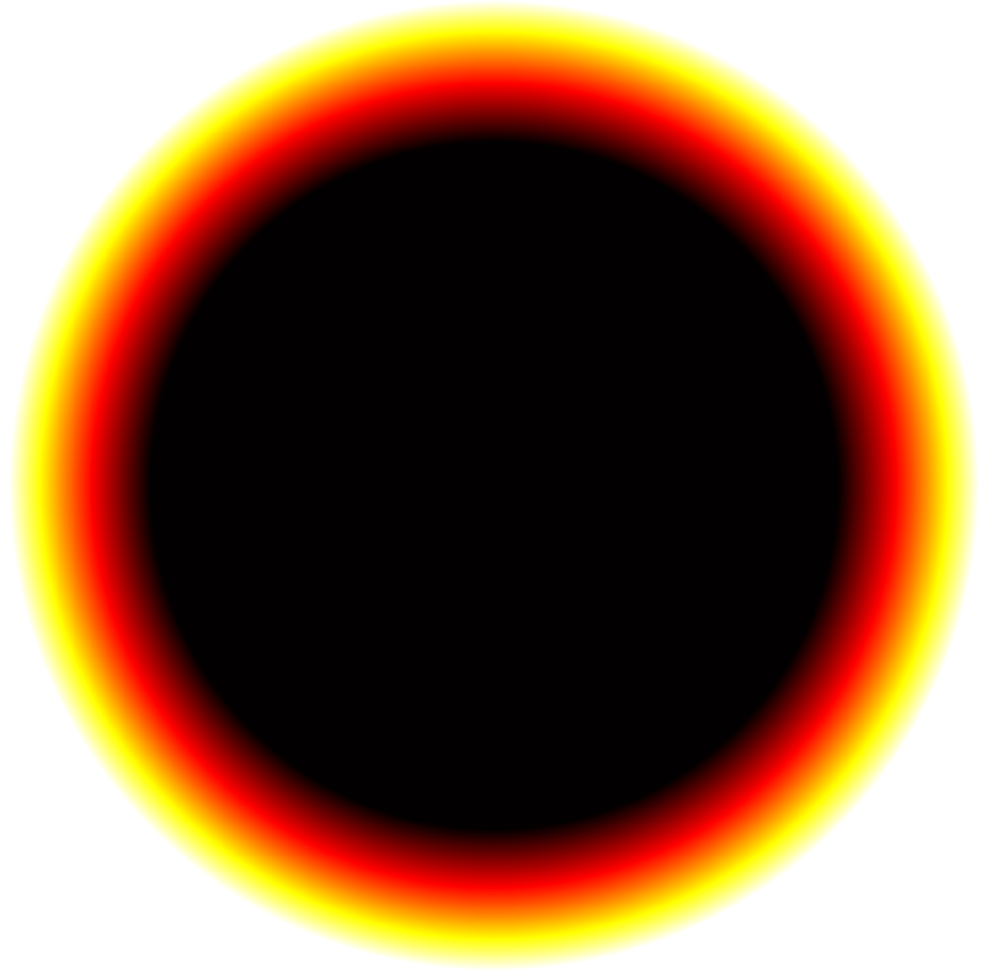}
    &
    \includegraphics[width=0.1\textwidth]{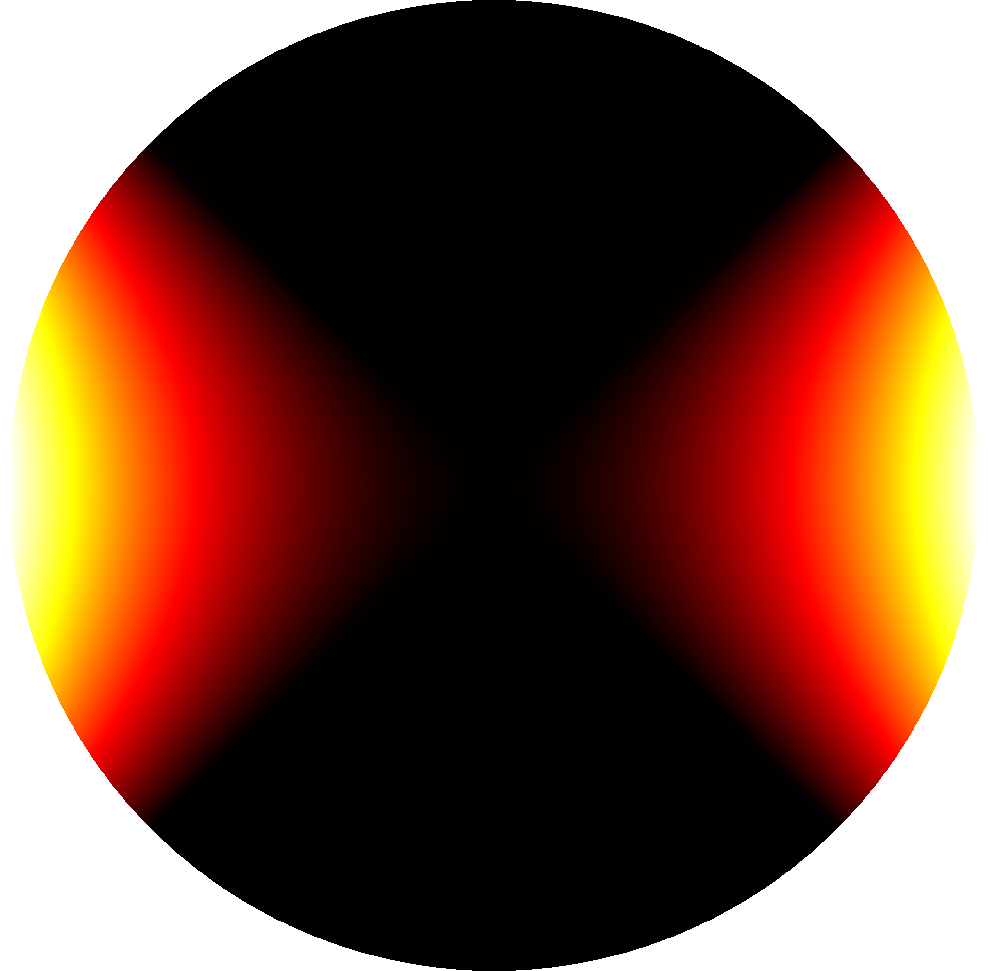}
    \\
    $Z_1$ & $Z_2$ & $Z_3$ & $Z_4$ 
    \\
    \includegraphics[width=0.1\textwidth]{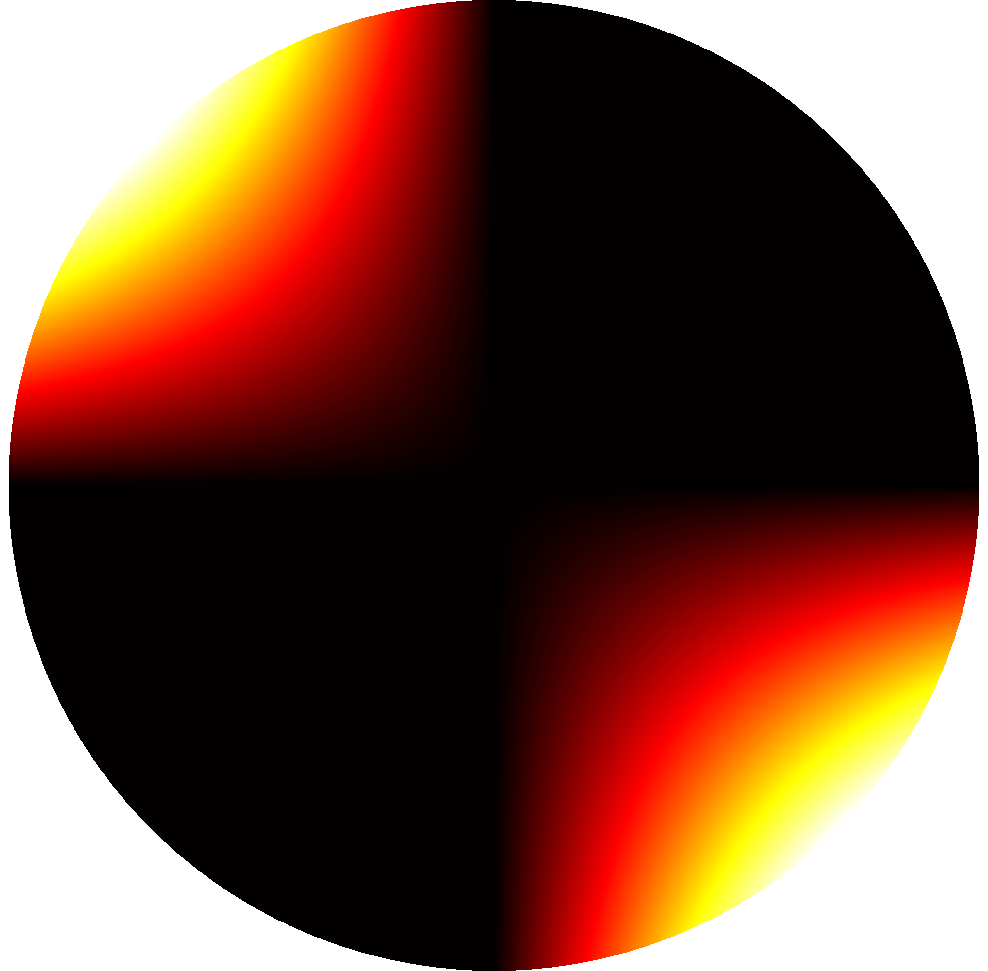}
    &
    \includegraphics[width=0.1\textwidth]{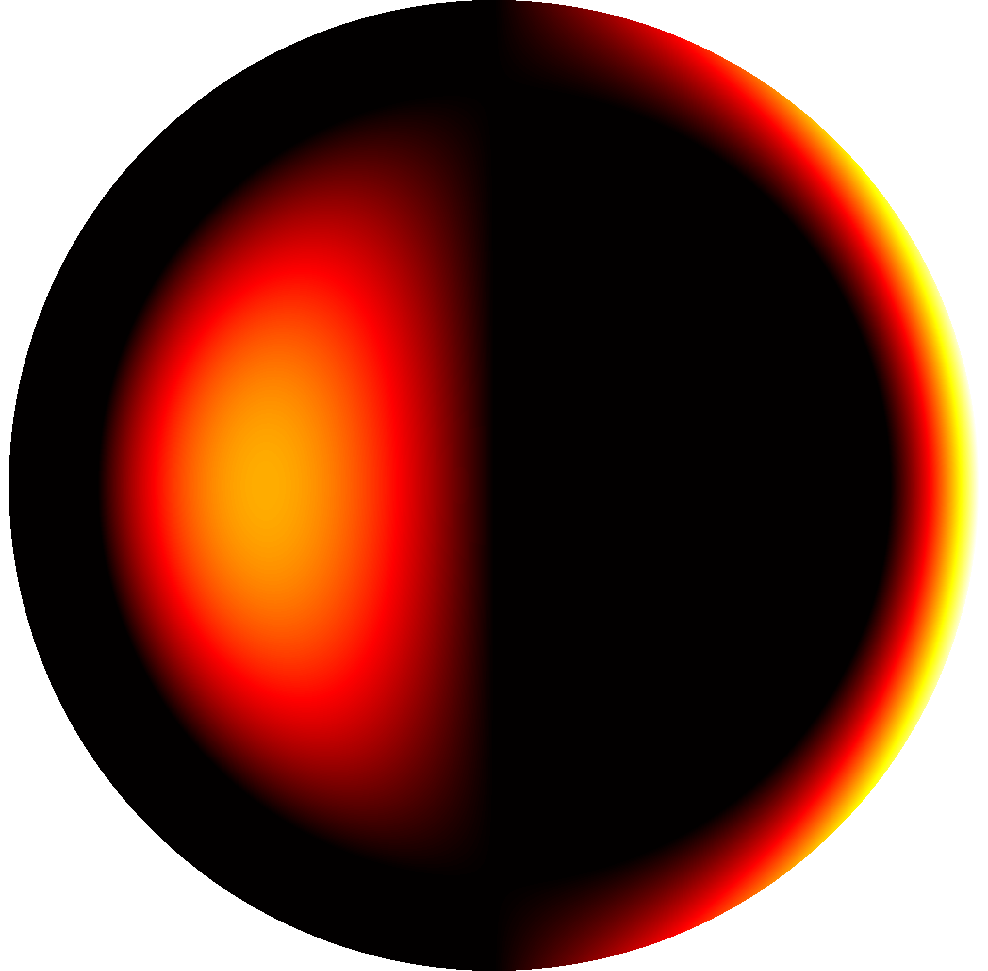}
    &
    \includegraphics[width=0.1\textwidth]{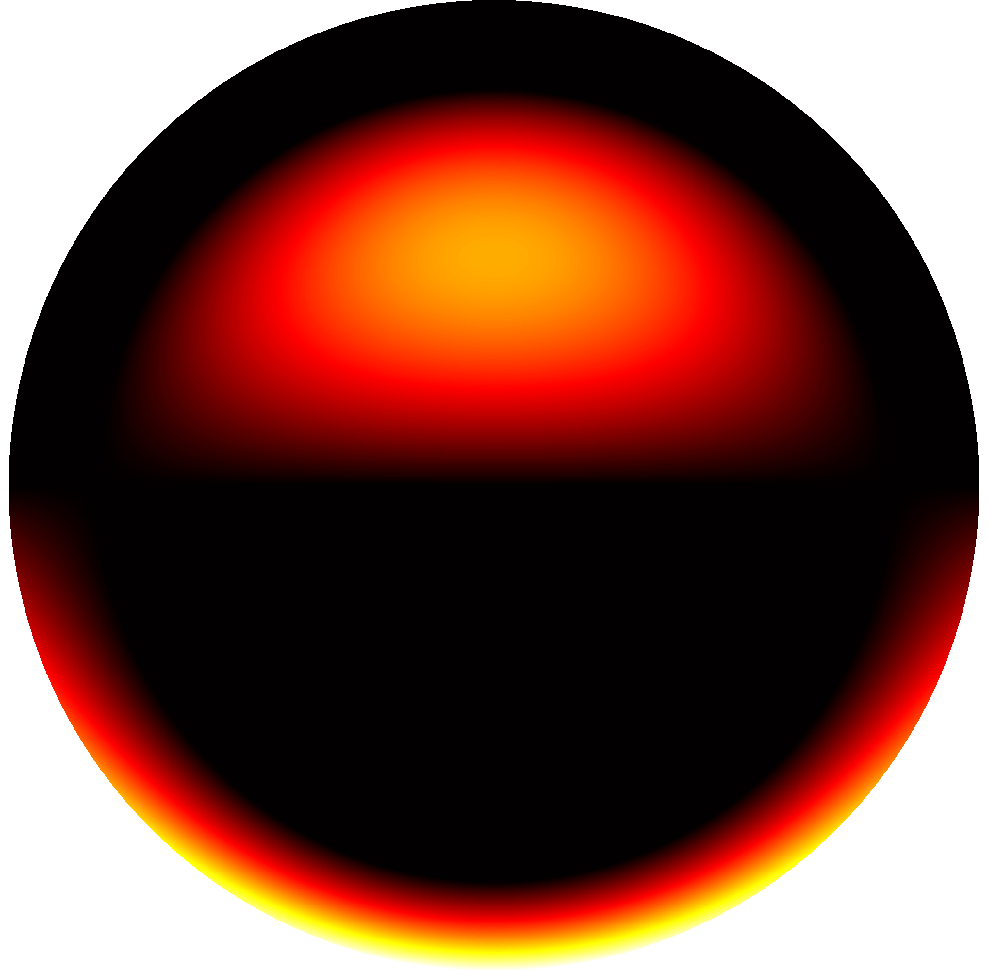}
    &
    \includegraphics[width=0.1\textwidth]{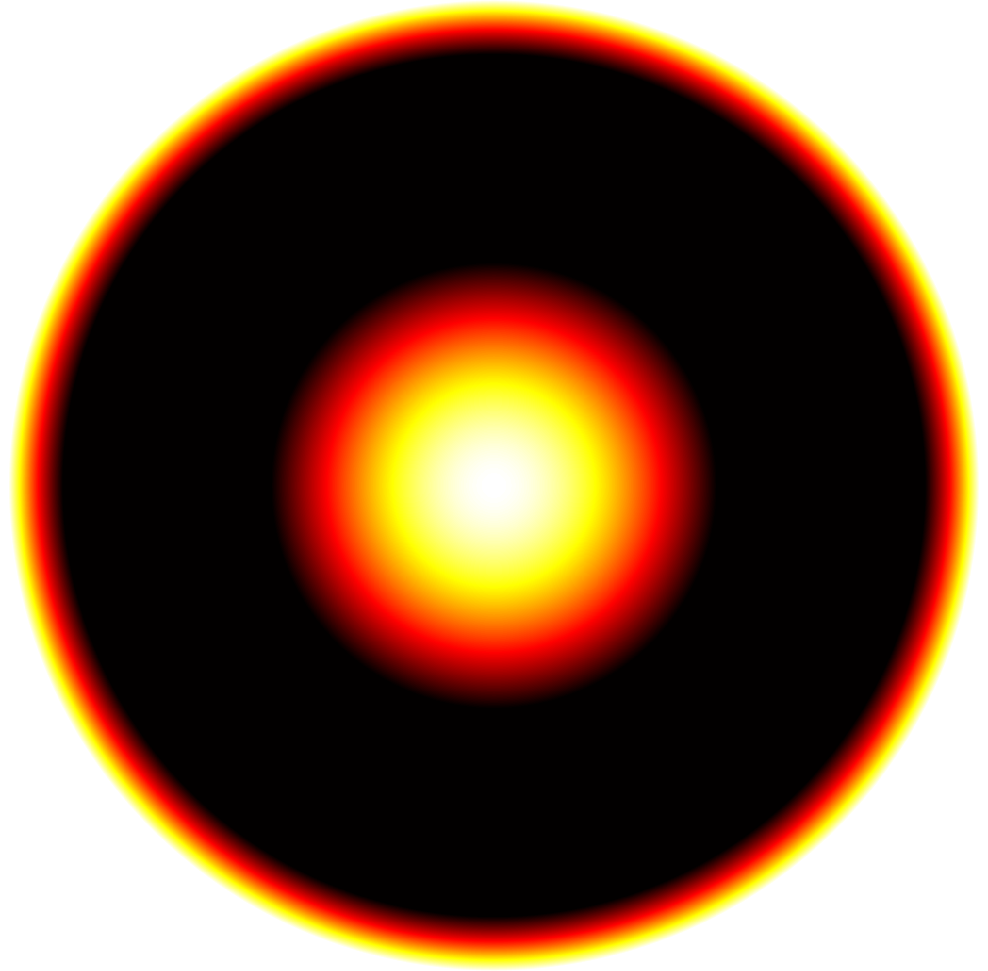}
    \\
    $Z_5$ & $Z_6$ & $Z_7$ & $Z_8$ 
    \\
    \includegraphics[width=0.1\textwidth]{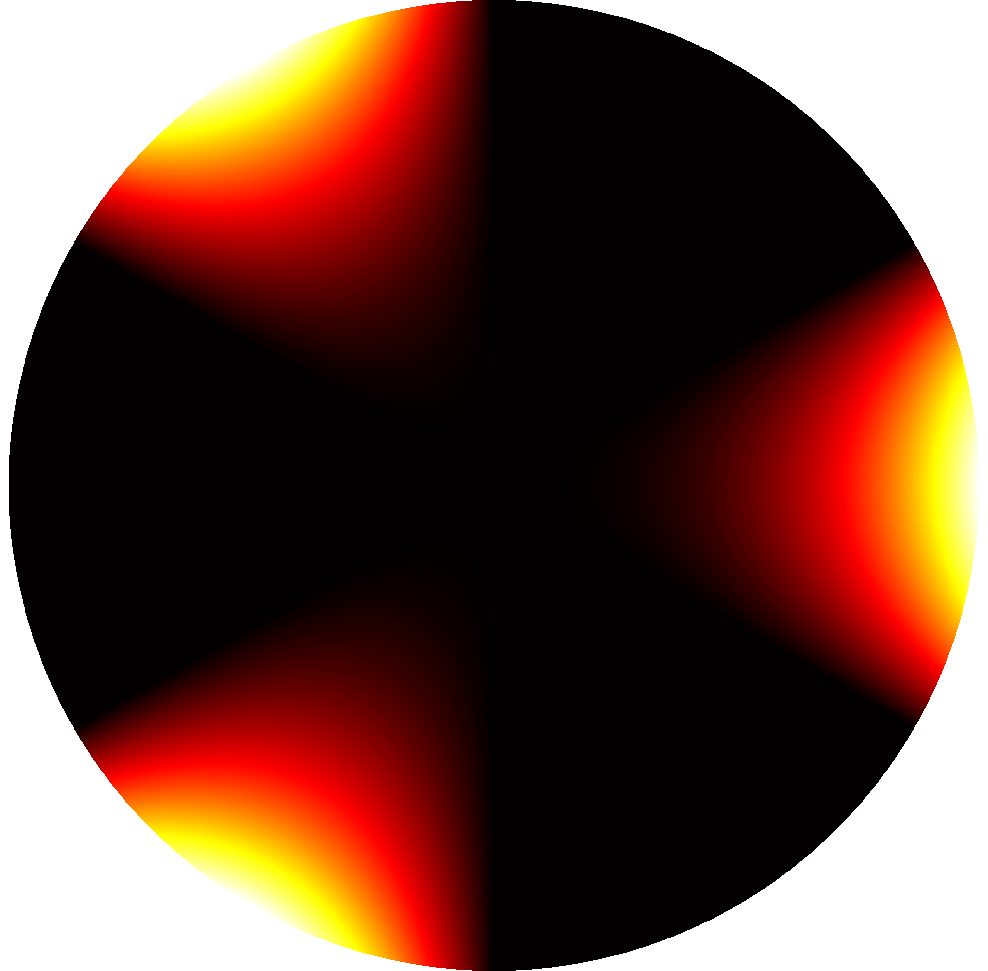}
    &
    \includegraphics[width=0.1\textwidth]{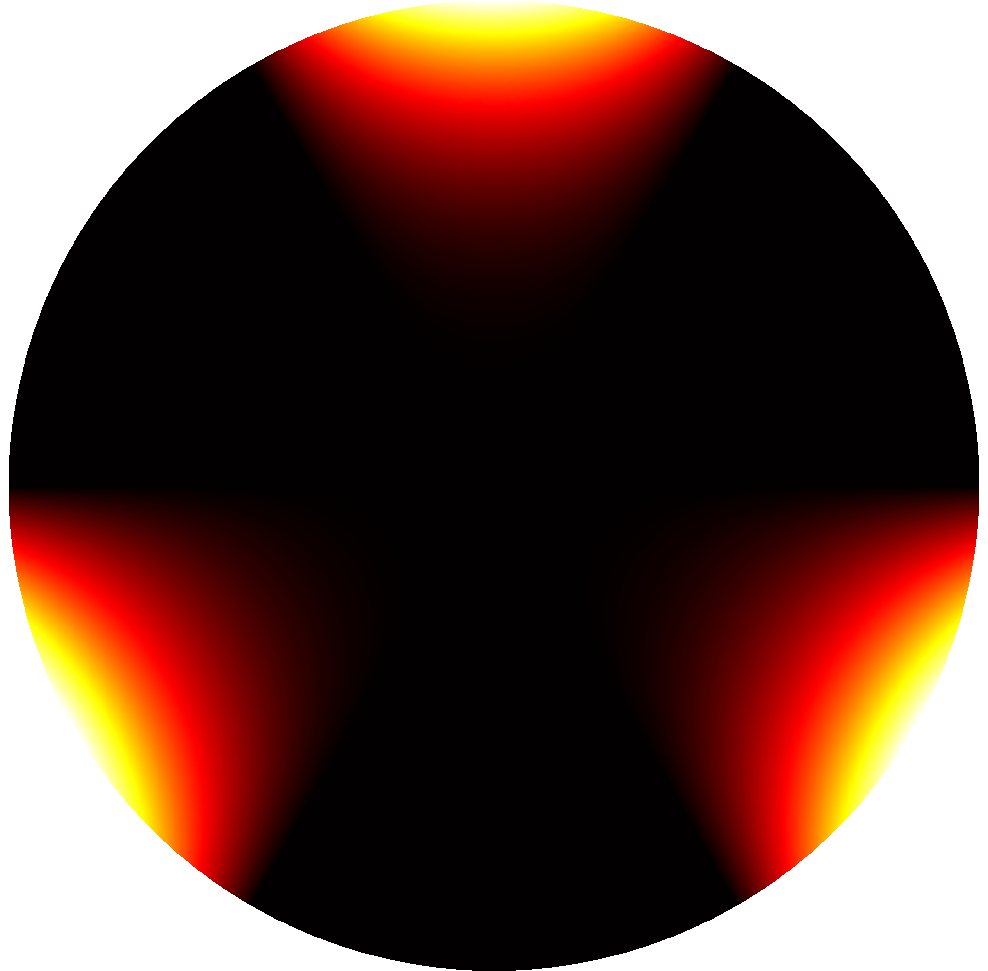}
    &
    \includegraphics[width=0.1\textwidth]{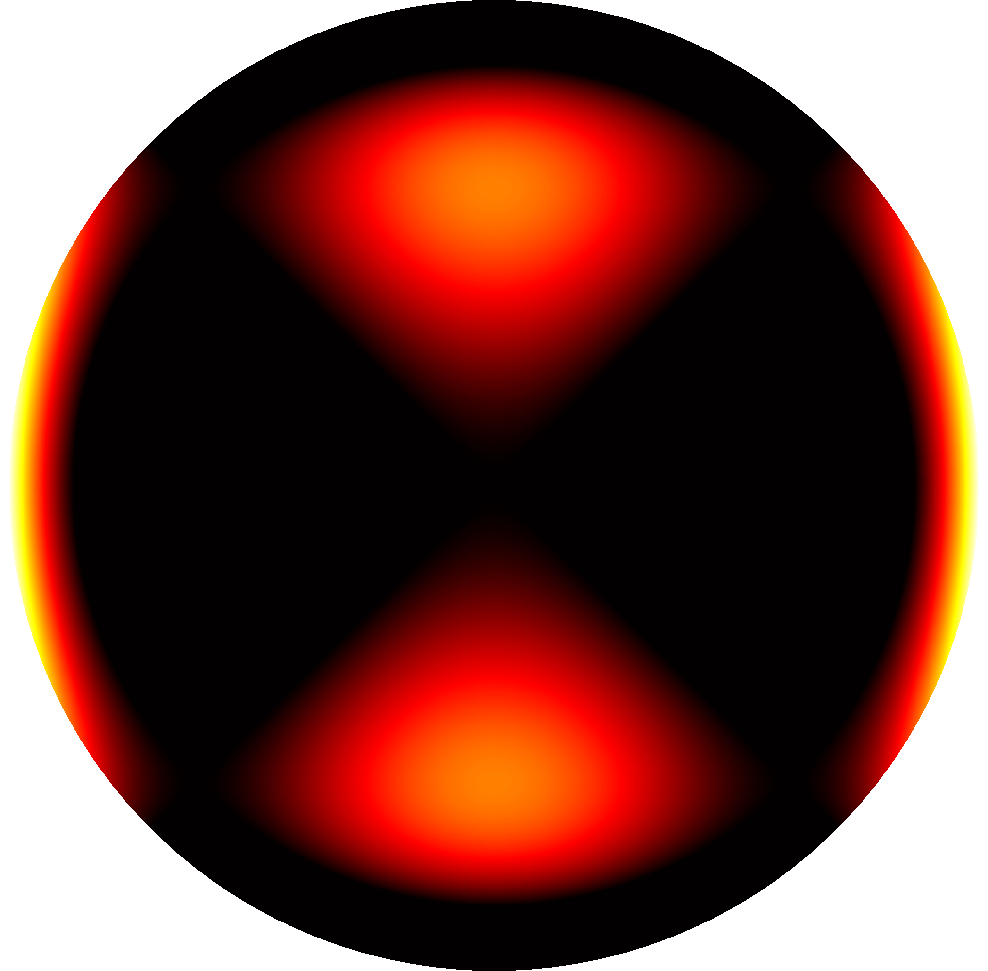}
    &
    \includegraphics[width=0.1\textwidth]{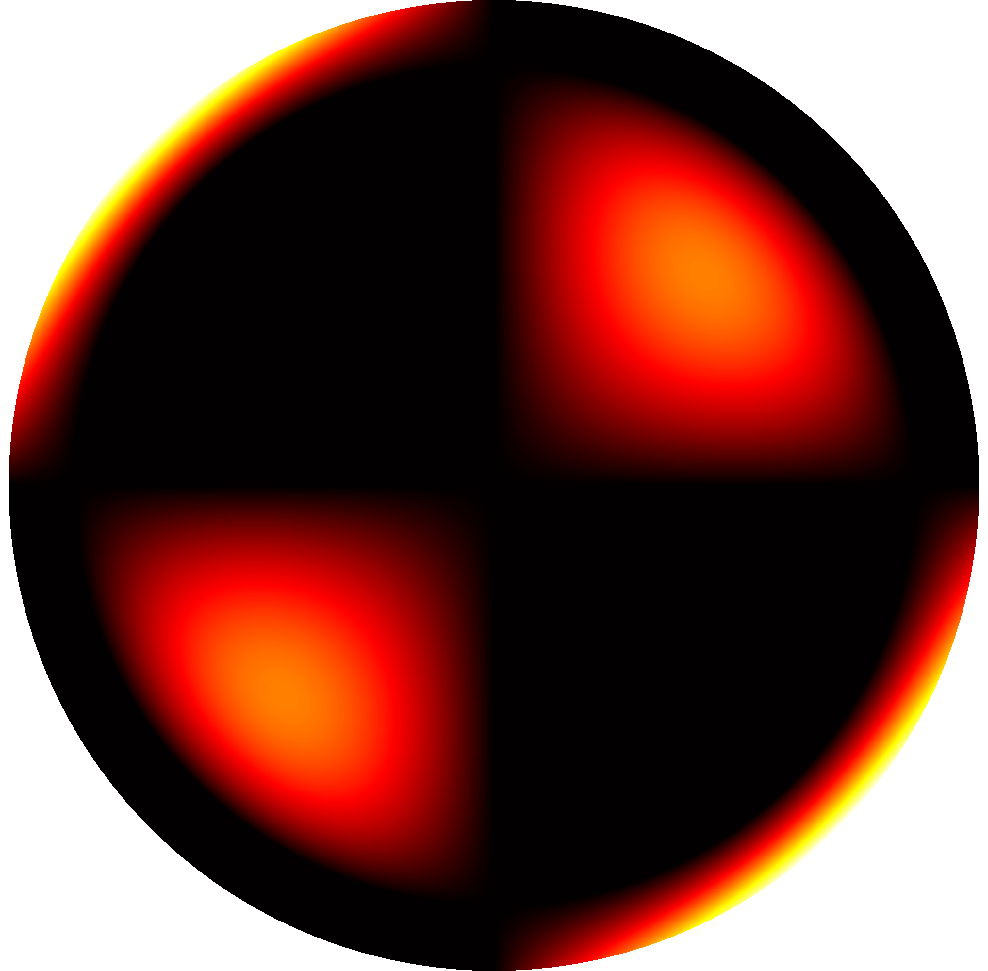}
    \\
    $Z_9$ & $Z_{10}$ & $Z_{11}$ & $Z_{12}$ 
    \\
    \includegraphics[width=0.1\textwidth]{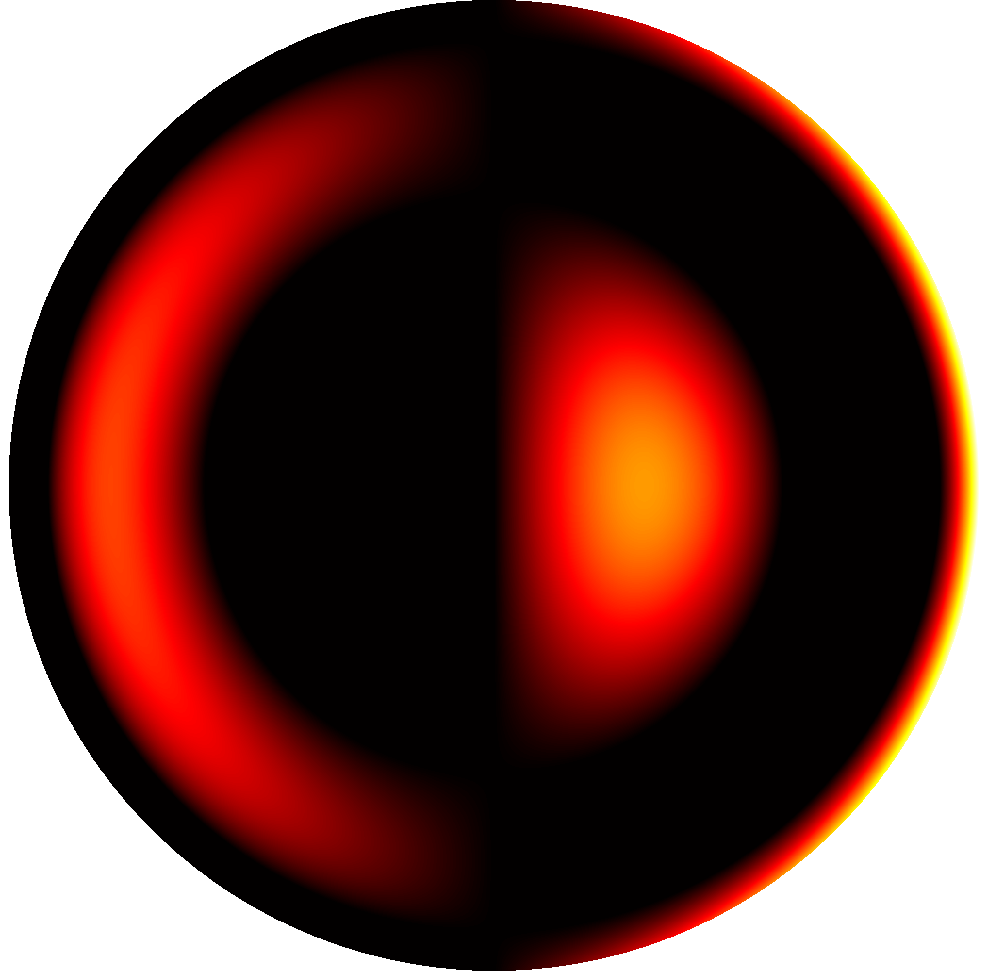}
    &
    \includegraphics[width=0.1\textwidth]{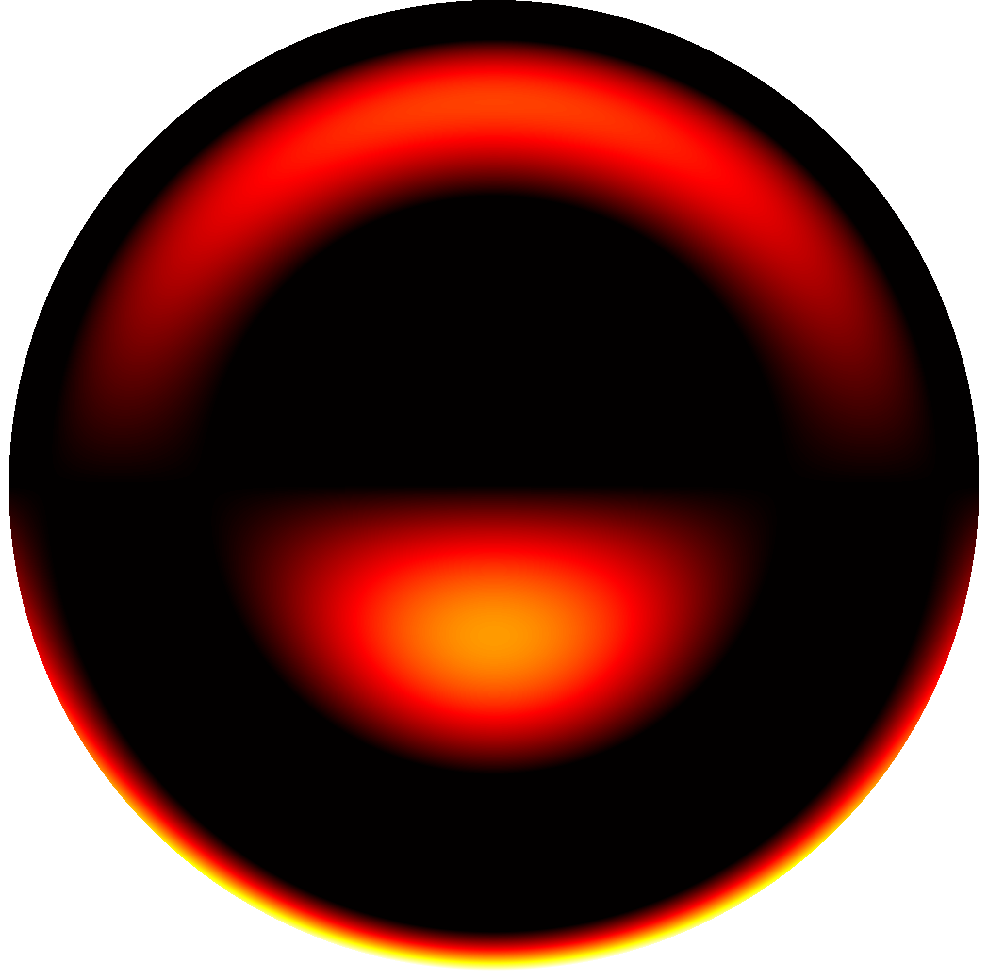}
    &
    \includegraphics[width=0.1\textwidth]{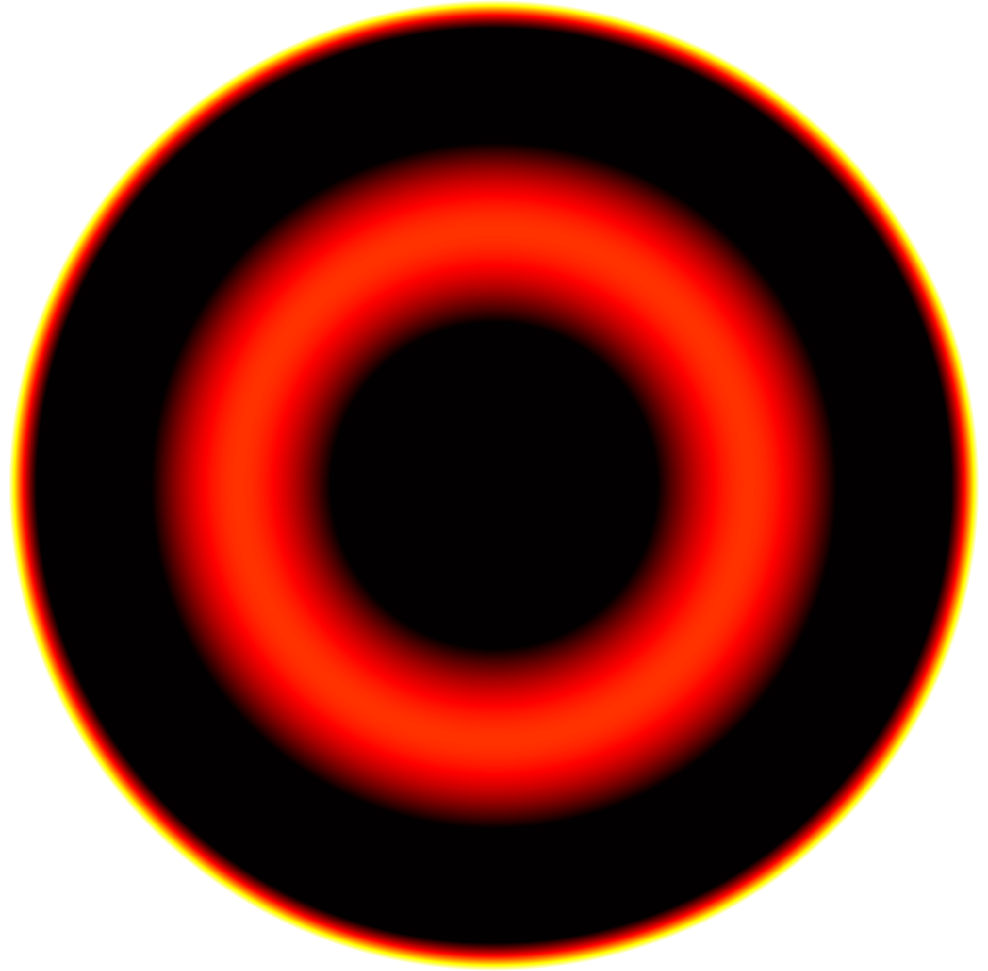}
    &
    \\
    $Z_{13}$ & $Z_{14}$ & $Z_{15}$ & 
    \\
  \end{tabular}
  \caption{The Zernike polynomials used in the PSF $h_z$,
    with image range $[-1,1]$. }
  \label{fig:zernike wiki}
\end{figure}

\begin{figure}[H]
  \begin{minipage}[t]{0.32\textwidth}
    \centering
    \includegraphics[width=\textwidth]{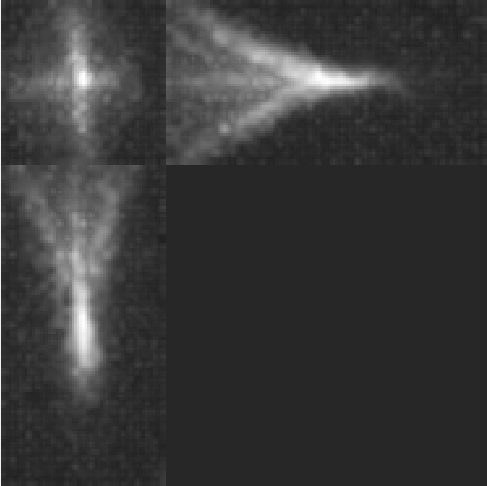}
    (a) Bead image data (maximum intensity projections)
  \end{minipage}
  \begin{minipage}[t]{0.32\textwidth}
    \centering
    \includegraphics[width=\textwidth]{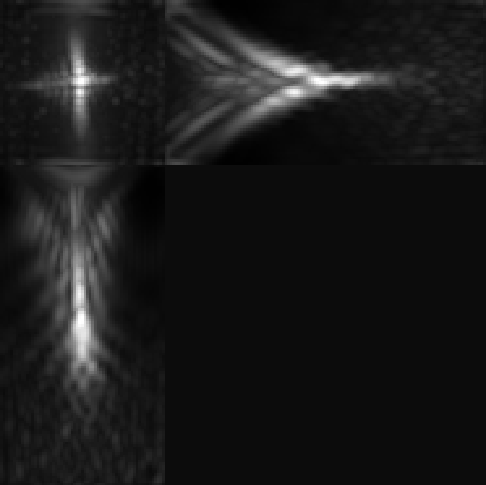}
    (b) Fitted PSF $h_z$, no blur 
    (maximum intensity projections)
  \end{minipage}
  \begin{minipage}[t]{0.32\textwidth}
    \centering
    \includegraphics[width=\textwidth]{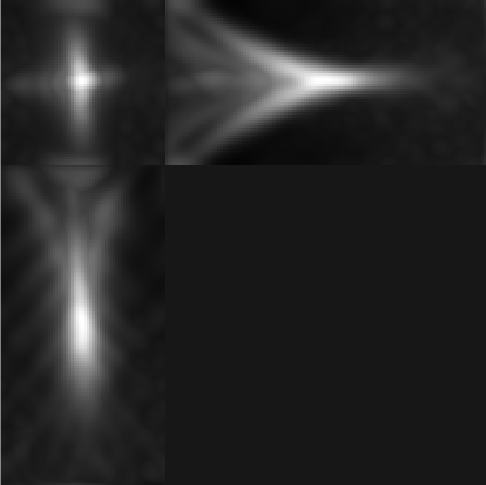}
    (c) Fitted PSF $h$ with Gaussian blur 
    (maximum intensity projections)
  \end{minipage}

  \caption{Fitted PSF using Zernike polynomials. In panel
  (c), we can see the benefits of using the Gaussian
  blur $g_{\sigma}$ in obtaining an accurate approximation
  of the bead in (a).}
  \label{fig:psf h zernike}
\end{figure}

%% file: 3_inverse_problem.tex
\section{Inverse problem}
\subsection{Problem statement}\label{sec:IP}

In this section we formally state the inverse problem of deblurring a light-sheet microscopy image. 
Let $\Omega \subset \R^3$ 
be a bounded Lipschitz domain and let $L \colon \L^p(\Omega) \to \L^2(\Omega)$ be the forward operator defined by~\eqref{eq:lightsheet model}. Here $1 < p < 3/2$ is chosen such that the embedding of the $\BV$ space is compact~\cite{Burger_Osher_TV_Zoo}. Clearly, $L$ is linear.

We consider the following inverse problem 
\begin{equation}\label{eq:IP}
  Lu = \bar{f},
\end{equation}
where $\bar f \in \L^2(\Omega)$ is the exact (noise-free) data. 
As outlined in Section~\ref{sec:ifm}, the measurements in light microscopy are corrupted by a combination of Poisson and Gaussian noise. More precisely, the measurement is given by $f = v + w$, where $v \sim Pois(\bar{f})$ is a Poisson distributed random variable with mean $\bar f$ and $w$ represents additive zero-mean Gaussian noise. We do not model Gaussian noise statistically and instead, in the spirit of (deterministic) variational regularisation, assume that $w \in \L^2(\Omega)$ is a fixed perturbation with $\norm{w}_{\L^2(\Omega)} \leq \sigma_G$ for some known $\sigma_G > 0$. Poisson noise is typically modelled using the Kullback-Leibler divergence as the data fidelity term~\cite{hohage2013poisson,hohage2016poisson}.

Let us give a brief justification of
the inverse problem formulation described
in this section~\cite{lanza_image_2014,calatroni2017infimal},
from a Bayesian perspective.
First, by using the Poisson and Gaussian probability density
functions, we have that 
$p(v|u) = \frac{(Lu)^v e^{-(Lu)}}{v!}$ and
$p(f|v) = \frac{1}{\sqrt{2\pi}\sigma_G}e^{-\frac{1}{2}\left(\frac{f-v}{\sigma_G}\right)^2}$,
and from Bayes' theorem and conditional probability:
\begin{align}
  p(u,v|f) = \frac{p(f|v)p(v|u)p(u)}{p(f)},
  \label{eq:bayes}
\end{align}
where we used that $p(f|u,v) = p(f|v)$.
Moreover, we assume that the prior is a Gibbs distribution 
$p(u) = e^{-\alpha \reg(u)}$ for a convex functional
$\reg(u)$, which we will introduce later.
To obtain a maximum \textit{a posteriori} estimation
of $u$ and $v$ (i.e. maximise the posterior distribution
$p(u,v|f)$), we take the minimum of the negative log 
of~\eqref{eq:bayes} and, after discarding the
denominator $p(f)$ and using the Stirling
approximation for the factorial $\log v! = v \log v - v$, 
we obtain the minimisation problem:
\begin{equation}
    \argmin_{u,v} 
    \alpha \reg(u) +  
    \frac{1}{2\sigma_G^2} \|f-v\|^2 +
    v \log \frac{v}{Lu} + Lu - v, 
    \label{eq:bayes min}
\end{equation}
where the first term is the regularisation term and the
remaining terms form the data fidelity term.

We will now describe the formal mathematical setting for 
\eqref{eq:bayes min} in the context of variational 
regularisation. This will allow us to show well-posedness
of the model, establish convergence rates of the solution
with respect to the noise in the measurements
and to introduce a discrepancy principle for choosing the value of the
regularisation parameter $\alpha$.

First, note that in \eqref{eq:bayes min},
we can perform the minimisation
over $v$ only on the data fidelity part of the objective,
which can be written as an infimal convolution of the
two separate Gaussian and Poisson fidelities. 
Therefore, we define the following data fidelity
term, as proposed in~\cite{calatroni2017infimal}:
\begin{equation}\label{eq:IC-fid}
    \Phi(\bar f,f) \defeq \inf_{v \in \L^2_+(\Omega)} \left\{ \frac{1}{2}\norm{f-v}^2_{\L^2} + D_{KL}(v, \bar f) \right\}, \quad f \in \L^2(\Omega), \, \bar f \in \L^1_+(\Omega),
\end{equation}
where $\L^{1,2}_+(\Omega)$ denotes the positive cone in $\L^{1,2}(\Omega)$ (that is, functions $f \in \L^{1,2}(\Omega)$ such that $f \geq 0$ a.e.) and $D_{KL}$ denotes the Kullback-Leibler divergence which we define as follows
\begin{align}
  D_{KL}(v,\bar f) :&= 
  \begin{cases}
  \int_{\Omega}\left(
    \bar f(x) - v(x) + v(x) \log \frac{v(x)}{\bar f(x)}
  \right) \dif x, \quad & v, \bar f \geq 0 \, \wedge \, \int_\Omega v \dif x = \int_\Omega \bar f \dif x = 1, \\
  +\infty \quad &\text{otherwise,}
  \end{cases}   \label{eq:dkl continuous} \\
  &= \begin{cases}
  \int_{\Omega}
     v(x) (\log v(x) - \log \bar f(x))
   \dif x, \qquad\quad\; & v, \bar f \geq 0 \, \wedge \, \int_\Omega v \dif x = \int_\Omega \bar f \dif x = 1, \\
  +\infty \qquad\quad\; &\text{otherwise.}
  \end{cases}
  \nonumber
\end{align}
We note that ${\abs{\int_\Omega v(x) \log v(x) \dif x} < \infty}$ for $v \in \L^2$, since $\L^2$ is continuously embedded into the Orlicz space $\L\log\L$ of functions of finite entropy~\cite{clason:2021-OT,bennett-sharpley:1988} 
\begin{equation}\label{eq:LlogL}
    \L\log\L(\Omega) \defeq \{f \in \L^1(\Omega) \colon \int_\Omega \abs{f(x)} (\log\abs{f(x)})_+ \dif x < \infty\},
\end{equation}
where $(\cdot)_+ = \max\{\cdot,0\}$ denotes the positive part.

A proof of the following result can be found in~\cite{calatroni2017infimal}, but we provide it here for readers' convenience.
\begin{prop}[Exactness of the infimal convolution] \label{prop:infconv-exact}
  For any $\bar f \in \L^1_+$ such that $\int_\Omega \bar f \dif x = 1$, there exists a unique solution $v^* = v^*(\bar f)$ of~\eqref{eq:IC-fid}, that is, the infimal convolution is exact. Moreover, the functional $\Phi(\bar f, \cdot) \colon \L^2 \to \R_+\cup\{+\infty\}$ is proper, convex and lower semicontinuous.
\end{prop}
\begin{proof}
  Fix $\bar f \in \L^1_+$ such that $\int_\Omega \bar f \dif x = 1$. Then~\eqref{eq:IC-fid} is the infimal convolution of the following two functionals on $\L^2$ 
  \begin{equation*}
    \phi(g) \defeq \chi_{\L^2_+}(g) + D_{KL}(g, \bar f), \quad \psi(g) \defeq \frac{1}{2}\norm{g}^2_{\L^2}, \quad g \in \L^2(\Omega),
  \end{equation*}
  where $\chi$ denotes the characteristic function.  The function $\phi$ is proper, convex, non-negative and lower semicontinuous, while $\psi$ is proper, convex, lower semicontinuous and coercive. Therefore, by~\cite[Prop. 12.14]{Bauschke_Combettes:2011}, the infimal convolution is exact and is itself a proper, convex and lower semicontinuous function. Uniqueness follows from strict convexity of $\psi$.
\end{proof}

Now we turn our attention to the lower semicontinuity of the functional $\Phi(\cdot,f)$ in its first argument.
\begin{prop}[Lower semicontinuity]\label{prop:lsc}
 For any $f \in \L^2_+(\Omega)$ such that $\int_\Omega f \dif x = 1$ the functional $\Phi(\cdot,f) \colon L^1(\Omega) \to \R_+\cup\{+\infty\}$ is lower semicontinuous.
\end{prop}
\begin{proof}
  We have
  \begin{align*}
      \Phi(g,f) &= \inf_{v \in \L^2_+(\Omega)} \left\{ \frac{1}{2}\norm{f-v}^2_{\L^2} + D_{KL}(v, g) \right\} \\
      &= \frac{1}{2}\norm{f-v^*(g)}^2_{\L^2} + D_{KL}(v^*(g), g)\\
      &= \frac{1}{2}\norm{f-v^*(g)}^2_{\L^2} + \int_{\Omega}
     v(x) (\log v(x) - \log \bar f(x))
   \dif x  + \chi_{\mathcal C}(g), \quad g \in \L^1(\Omega),
  \end{align*}
  where $v^*(g)$ is as defined in \cref{prop:infconv-exact} and $\mathcal C := \{g \in \L^1_+(\Omega) \colon \int_\Omega g \dif x = 1\}$. The characteristic function is lower semicontinuous because $\mathcal C$ is closed in $\L^1$ and the rest is lower semicontinuous by~\cite[Thm. 4.1]{calatroni2017infimal}.
\end{proof}

The following fact is easily established.
\begin{prop}\label{prop:L-cont}
  The operator $L \colon \L^p(\Omega) \to \L^1(\Omega)$ defined in~\eqref{eq:lightsheet model} 
  is continuous for any $ p \geq 1$. Moreover, if $l$ and
  $h$ are non-negative and have overlapping support:
  \begin{equation*}
    \supp(l) \cap \supp(h) \ne \varnothing,
  \end{equation*}
  then $\mathbf{1} \notin \ker{L}$, where $\mathbf{1}$ is the constant one function
  and $\ker{L}$ is the null space of $L$.
\end{prop}
\begin{proof}
  By \eqref{eq:lightsheet model}, we have
  \begin{equation*}
    Lu(x,y,z) = \int_{\Omega} l(s,t,w)h(x-s,y-t,w)u(s,t,w-z) \dif \mu_{stw},
  \end{equation*}
  where $\dif \mu_{stw} := \dif s \dif t \dif w$. Noting that 
  the light-sheet PSF $l$ and detection PSF $h$ are bounded 
  from above by some $C_1, C_2 > 0$, we have that:
  \begin{align*}
    \|Lu\|_{\L^1} &=  \int_{\Omega} \left| 
        \int_{\Omega} l(s,t,w)h(x-s,y-t,w)u(s,t,w-z) \dif\mu_{stw} 
    \right| \dif \mu_{xyz} \\
    &=
    \int_{\Omega} \left| 
        \int_{\Omega} l(s,t,w'+z)h(x-s,y-t,w'+z)u(s,t,w') \dif\mu_{stw'} 
    \right| \dif \mu_{xyz}   
    \quad (\text{by } w' = w-z) \\
    &\leq
    C_1 C_2 \int_{\Omega} \left| 
        \int_{\Omega} u(s,t,w') \dif\mu_{stw'} 
    \right| \dif \mu_{xyz} \\
    &= C_1 C_2 \abs{\Omega} \left| 
        \int_{\Omega} u(s,t,w') \dif\mu_{stw'} 
    \right| \\
    &\leq C(p) \norm{u}_{\L^p},
  \end{align*}
  where in the last inequality we applied H\"{o}lder's
  inequality and $C(p)$ is a constant that depends on $p$ 
  (as well as $C_{1,2}$ and $\Omega$). Hence, we obtain 
  the first claim.
  
  For the second claim, 
  we observe that
  \begin{equation*}
    L\mathbf 1(x,y,z) = 
    \int_{\Omega} l(s,t,w)h(x-s,y-t,w) \dif\mu_{stw} \geq 0.
  \end{equation*}
  Consider
  \begin{align*}
    \int_\Omega L\mathbf 1(x,y,z) \dif\mu_{xyz} &= 
    \int_\Omega \int_{\Omega} l(s,t,w)h(x-s,y-t,w) 
    \dif\mu_{stw} \dif\mu_{xyz} 
  \end{align*}
  and let $B_{l,h} \subset \supp{l} \cap \supp{h}$. Then,
  since both $l$ and $h$ are non-negative on $\Omega$, from
  the last equality above we have that:
  \begin{align*}
    \int_\Omega L\mathbf 1(x,y,z) \dif\mu_{xyz} &\geq
    \int_{B_{l,h}} l(x,y,z)h(x,y,z) \dif\mu_{xyz} > 0,
  \end{align*}
  which proves the second claim.
\end{proof}

\begin{remark}
Our setting with the measured data $f \in \L^2(\Omega)$ 
 differs slightly from~\cite{calatroni2017infimal}, where $f \in \L^\infty(\Omega)$ was assumed. 
\end{remark}

We will consider the following variational regularisation problem
\begin{equation}\label{eq:var-reg1}
    \min_{u \in \L^p_+(\Omega)} \Phi(f,Lu) + \alpha \reg(u),
\end{equation}
where $\Phi$ is the infimal-convolution fidelity as defined in~\eqref{eq:IC-fid}, $\reg \colon \L^p \to \R_+\cup\{+\infty\}$ is a regularisation functional,  $\alpha \in \R_+$ is a regularisation parameter and $1 < p < 3/2$. Without loss of generality, we assume that $\int_\Omega \bar f \dif x = 1$.

As the regulariser $\reg$ we choose the total variation~\cite{ROF}
\begin{equation*}
  \reg(u) = \TV(u) \defeq 
  \sup_{\substack{
    \xi \in C_0^{\infty}(\Omega,\mathbb{R}^3) \\ \|\xi\|_{\infty}\leq1
  }}
  \int_\Omega u \div(\xi) \dif x.
\end{equation*}

By the Rellich–Kondrachov theorem, the space 
\begin{equation*}
\BV(\Omega)\defeq\{u \in \L^1(\Omega) \colon \TV(u) < \infty\}, \quad \norm{u}_{\BV} \defeq \norm{u}_{\L^1} + \TV(u),    
\end{equation*}
is compactly embedded into $\L^p(\Omega)$ for $1 \leq p < 3/2$ and continuously embedded into $\L^{3/2}(\Omega)$ since $\Omega \subset \R^3$. Therefore, we consider $\TV \colon \L^p \to \R_+\cup\{+\infty\}$ 
\begin{equation*}
    \TV(u) \defeq
    \begin{cases}
      \sup\limits_{\substack{
    \xi \in C_0^{\infty}(\Omega,\mathbb{R}^3) \\ \|\xi\|_{\infty}\leq1
  }}
  \int_\Omega u \div(\xi) \dif x, \quad & u \in \BV(\Omega), \\
      \infty, \quad & u \in \L^p(\Omega) \setminus { \BV(\Omega)}.
    \end{cases}
\end{equation*}

We will denote by $\Jminsol$ the $\TV$-minimising solution of~\eqref{eq:IP}, i.e. a solution that satisfies
\begin{equation*}
    L \Jminsol = \bar f \text{\quad and \quad $\TV(\Jminsol) \leq \TV(u)$ for all $u$ s.t. $Lu=\bar f$.}
\end{equation*}
The existence of such solution is obtained by standard arguments~\cite{Benning_Burger_modern:2018}. We will make the reasonable assumption that the $\TV$-minimising solution is positive, i.e. $\Jminsol \geq 0$ a.e. Due to the positivity of the kernels involved in~\eqref{eq:lightsheet model}, it is clear that $\Jminsol \geq 0$ implies $L \Jminsol = \bar f \geq 0$.

Since by Proposition~\ref{prop:infconv-exact} the infimal convolution~\eqref{eq:IC-fid} is exact, we can equivalently rewrite~\eqref{eq:var-reg1} as follows
\begin{equation}\label{eq:var-reg2}
    \min_{\substack{u \in \L^p_+(\Omega) \\ v \in \L^2_+(\Omega)}} 
    \frac{1}{2}\norm{f-v}^2_{\L^2} + D_{KL}(v, Lu) + \alpha \reg(u).
\end{equation}

Existence of minimisers in~\eqref{eq:var-reg1} and~\eqref{eq:var-reg2} is obtained by standard arguments~\cite[Thm. 4.1]{calatroni2017infimal}.
\begin{prop}\label{prop:existence}
Each of the opimisation problems~\eqref{eq:var-reg1} and~\eqref{eq:var-reg2} admits a unique minimiser.
\end{prop}

We will also need the following coercivity result. 
\begin{prop}\label{prop:coercivity}
 The functional $\Phi(f,\cdot) \colon \L^1(\Omega) \to \R_+\cup\{+\infty\}$ is strongly coercive with exponent $2$, i.e. there exists a constant $C>0$ such that
 \begin{equation*}
     \Phi(f,g) \geq C \norm{g-f}_{\L^1}^2, \quad g \in \L^1(\Omega).
 \end{equation*}
\end{prop}
\begin{proof}
Using Pinsker's inequality for the Kullback-Leibler divergence, we get
\begin{align*}
    \Phi(f,g) &= \inf_{v \in \L^2_+} \frac12 \norm{v-f}_{\L^2}^2 + D_{KL}(v,g) \\
    &\geq \inf_{v \in \L^2_+} \frac12 \norm{v-f}_{\L^2}^2 +  \norm{g-v}_{\L^1}^2 \\
    &\geq 2 C\inf_{v \in \L^2_+} \norm{v-f}_{\L^1}^2 +  \norm{g-v}_{\L^1}^2
\end{align*}
for some $C>0$. Note that Pinsker's inequality assumes that $f,g \geq 0$ and $\inf_\Omega f \dif x = \int_\Omega g \dif x = 1$, which we ensure by definition in~\eqref{eq:dkl continuous}.

Now, using the inequality $\frac12(a+b)^2 \leq a^2+b^2$ that holds for all $a,b \in \R$ and the triangle inequality, we obtain the claim
\begin{align*}
    \Phi(f,g) &\geq C\inf_{v \in \L^2_+} \left( \norm{v-f}_{\L^1} +  \norm{g-v}_{\L^1} \right)^2 \\
    &\geq C \inf_{v \in \L^2_+} \norm{v-f+g-v}_{\L^1}^2 \\
    &= C \norm{g-f}_{\L^1}^2.
\end{align*}
\end{proof}



\subsection{Convergence rates}

Our aim in this section is to establish convergence rates of minimisers of~\eqref{eq:var-reg1} as the amount of noise in the data decreases. But first we need to specify what we mean by the amount of noise in our setting.

We argue as follows. Since the noise in the measurement is generated sequentially, i.e.  photo-electrons are first counted by the sensor leading to a Poisson noise and later they are collected by the electronic circuit generating an additive Gaussian noise,
for any exact data $\bar f$ there exists $\bar z \sim Pois(\bar f)$ such that $D_{KL}(\bar z, \bar f) \leq \gamma$, where $\gamma>0$ depends on the exposure time $t$ and vanishes as $t \to \infty$~\cite{hohage2013poisson}. Further, there exists $w \in \L^2(\Omega)$ with $\norm{w}_{\L^2} \leq \sigma_G$ such that $f = \bar z + w$. Since $\bar z \geq 0$ is feasible in~\eqref{eq:IC-fid}, we get the following upper bound on the fidelity term~\eqref{eq:IC-fid} evaluated at the measurement $f$ and the exact data $\bar f$
\begin{equation}\label{eq:delta-def}
    \Phi(\bar f, f) \leq \frac12 \norm{f-\bar z}^2_{\L^2} + D_{KL}(\bar z,\bar f) = \frac12 \norm{w}^2_{\L^2} + D_{KL}(\bar z,\bar f) \leq \frac{\sigma_G^2}{2} + \gamma.
\end{equation}

The standard tool for establishing convergence rates are Bregman distances associated with the regulariser $\reg$. We briefly recall the necessary definitions.

\begin{defn} Let $X$ be a Banach space and $\reg \colon X \to \R_+\cup\{+\infty\}$ a proper convex functional.  The generalised Bregman distance between $x,y \in X$ corresponding to the subgradient $q \in \dJ(y)$ is defined as follows
	\begin{equation*}
	D_\reg^q(x,y) \defeq \reg(x) - \reg(y) - \sp{q,x-y}.
	\end{equation*}
	Here $\dJ(v)$ denotes the subdifferential of $\reg$ at $y \in X$. If, in addition, $p \in \dJ(x)$, the symmetric Bregman distance between $x,y \in X$ corresponding to the subgradients $p,q$ is defined as follows
	\begin{equation*}
	D_\reg^{p,q}(x,y) \defeq D_\reg^q(x,y) + D_\reg^p(y,x) = \sp{p-q,x-y}.
	\end{equation*}
\end{defn}

To obtain convergence rates, an additional assumption on the regularity of the $\TV$-minimising solution, called the~\emph{source condition}, needs to be made. We use the following variant~\cite{Burger_Osher:2004}.
\begin{assumption}[Source condition]\label{ass:sc}
	There exists an element $\mu^\dagger \in L^\infty(\Omega)$ such that
	\begin{equation*}
	q^\dagger := L^*\mu^\dagger \in \dJ(\Jminsol).
	\end{equation*}
\end{assumption}

\subsubsection{Parameter choice rules}

Let us summarise what we know about the fidelity function $\Phi$ as defined in~\eqref{eq:IC-fid}, the regularisation functional $\TV$ and the forward operator $L$:
\begin{itemize}
    \item $\Phi(f,\cdot)$ is proper, convex and coercive (\cref{prop:coercivity}) in $\L^1(\Omega)$;
    \item $\Phi(\cdot,\cdot)$ is jointly convex~\cite{resmerita:2007-KL} and lower semicontinuous (\cref{prop:infconv-exact,prop:lsc});
    \item $\Phi(f,g) = 0$ if and only if $f=g$;
    \item $\TV \colon \L^1(\Omega) \to \R\cup\{+\infty\}$ is proper, convex and lower semicontinuous~\cite{Burger_Osher_TV_Zoo} and its null space is given by $\ker{\TV} = \vecspan\{\mathbf 1\}$, where $\mathbf 1$ denotes the constant one function;
    \item $\TV$ is coercive on the complement of its null space in $\L^1(\Omega)$~\cite{Burger_Osher_TV_Zoo};
    \item $L \colon \L^p(\Omega) \to \L^1(\Omega)$ is continuous and $\ker{\TV}\cap\ker{L} = \{0\}$ (\cref{prop:L-cont}).
\end{itemize}

Using these facts and slightly modifying the proofs from~\cite{bungert:2020-arbitrary}, we obtain the following
\begin{thm}[Convergence rates under a priori parameter choice rules]
Let assumptions made in Section~\ref{sec:IP} hold and let the source condition (\cref{ass:sc}) be satisfied at the $\TV$-minimising solution $\Jminsol$. Let $u_{\sigma_G,\gamma}$ be a solution of~\eqref{eq:var-reg1} and let $\alpha$ be chosen such that
\begin{equation*}
    \alpha(\sigma_G,\gamma) = \bigO(\sigma_G+\sqrt{\gamma}).
\end{equation*}
Then
\begin{equation*}
    D^{q^\dagger}_{\TV} (u_{\sigma_G,\gamma},\Jminsol) = \bigO(\sigma_G+\sqrt{\gamma}),
\end{equation*}
where $q^\dagger = L^*\mu^\dagger$ is the subgradient from~\cref{ass:sc} and $\sigma_G,\gamma>0$ are as defined in~\eqref{eq:delta-def}.
\end{thm}
\begin{proof}
The proof is similar to~\cite[Thm. 3.9]{bungert:2020-arbitrary}.
\end{proof}

In a similar manner, we can obtain convergence rates for an a posteriori parameter choice rule known as the discrepancy principle~\cite{Morozov:1966,engl:1996,Sixou:2018}. Let $f$ be the noisy data and 
$\delta>0$ the amount of noise such that $\Phi(\bar f, f) \leq \delta$, where $\Phi$ is as defined in \eqref{eq:IC-fid}. In our case, $\delta = \frac{\sigma_G^2}{2} + \gamma$ by \eqref{eq:delta-def}. The discrepancy principle amounts to selecting $\alpha=\alpha(f,\delta)$ such that
\begin{equation}\label{eq:discr_pr1}
    \alpha = \sup\{\alpha>0 \colon \Phi(Lu^\alpha,f) \leq \tau\delta \},
\end{equation}
where $u^\alpha$ is the regularised solution corresponding the the regularisation parameter $\alpha$  and $\tau>1$ is a parameter.

Again, slightly modifying the proofs from~\cite{bungert:2020-arbitrary}, we obtain the following
\begin{thm}[Convergence rates under the discrepancy principle]
Let assumptions made in Section~\ref{sec:IP} hold and let the source condition (\cref{ass:sc}) be satisfied at the $\TV$-minimising solution $\Jminsol$. Let $u_{\sigma_G,\gamma}$ be a solution of~\eqref{eq:var-reg1} with $\alpha$ chosen according to the discrepancy principle \eqref{eq:discr_pr1}. 
Then
\begin{equation*}
    D^{q^\dagger}_{\TV} (u_{\sigma_G,\gamma},\Jminsol) = \bigO(\sigma_G+\sqrt{\gamma}),
\end{equation*}
where $q^\dagger = L^*\mu^\dagger$ is the subgradient from~\cref{ass:sc} and $\sigma_G,\gamma>0$ are as defined in~\eqref{eq:delta-def}.
\end{thm}
\begin{proof}
The proof is similar to~\cite[Thm. 4.10]{bungert:2020-arbitrary}.
\end{proof}

%% file: 4_solving.tex
\section{Solving the minimisation problem}
\label{sec:solving min problem}

\subsection{PDHG for infimal convolution model}

In practice, due to the joint convexity of the 
Kullback-Leibler divergence, we solve the minimisation 
problem \eqref{eq:var-reg2}, where we treat the 
reconstructed sample $u$ and the Gaussian denoised 
image $v$ jointly and, in addition, we impose lower and upper bound 
constraints on $u$ and $v$ by including the corresponding
characteristic functions in the objective:
\begin{equation}
  \min_{u,v} \left\{
    \alpha \TV(u) + 
    \frac{1}{2\sigma_G^2} \|f-v\|_2^2 + D_{KL}(v,Lu) 
    + \chi_{[l_1,l_2]^{2N}}([u, v]^T)
  \right\}.
  \label{eq:min problem impl}
\end{equation}
Note that the objective function in \eqref{eq:min problem impl}
is a sum of convex functions 
(the Kullback-Leibler 
divergence $D_{KL}$ is jointly convex~\cite{lindblad:1973}), and therefore
is itself convex.
We then write the problem \eqref{eq:min problem impl} as:
\begin{equation}
  \min_{w} \left\{ G(w) + \sum_{i=1}^m H_i(L_i w) \right\},
  \label{eq:min problem general}
\end{equation}
where we solve for $w = \begin{bmatrix}u \\ v\end{bmatrix}$, $m=3$ 
and:
\begin{align}
  &G(w) = \chi_{[l_1,l_2]^{2N}}\left(
  \begin{bmatrix}
    u\\v
  \end{bmatrix}\right),\\
  &H_1(\cdot) = \frac{1}{2\sigma_G^2} \left\|
    \cdot - f
  \right\|_2,
  && 
  L_1 = \begin{bmatrix} 0 & 1 \end{bmatrix},\\
  &H_2(w) = D_{KL}(v, u),
  && 
  L_2 = \begin{bmatrix}
    L & 0 \\ 0 & 1 
  \end{bmatrix}, \\
  &H_3(\cdot) = \alpha \left\|  
    \cdot
  \right\|_1,
  &&
  L_3 = \begin{bmatrix}
    \nabla_x & 0 \\ \nabla_y & 0 \\ \nabla_z & 0
  \end{bmatrix},
\end{align}
where $L$ is the forward operator corresponding
to the image formation model from Section~\ref{sec:ifm}.

Rather than solving the problem~\eqref{eq:min problem general}
directly, a common approach is to reformulate it 
as a saddle point problem using the Fenchel conjugate
$G^*(y) = \sup_z \langle z,y \rangle - G(z)$. For proper, convex 
and lower semicontinuous function $G$, we have
that $G^{**} = G$, so \eqref{eq:min problem general}
can be written as the saddle point problem
\begin{equation}
  \min_w \sup_{y_1,\ldots,y_m}
  \left\{
    G(w) + 
    \sum_{i=1}^m \langle y_i, L_i x \rangle - H_i^*(y_i)
  \right\},
  \label{eq:saddle point problem}
\end{equation}
and by swapping the $\min$ and the $\sup$ and applying
the definition of the convex conjugate $G^*$, one obtains
the dual of \eqref{eq:min problem general}:
\begin{equation}
  \max_{y_1,\ldots,y_m} 
  \left\{
    -G^*( -\sum_{i=1}^m L_i^* y_i )
    - \sum_{i=1}^m H_i^*(y_i) 
  \right\}.
  \label{eq:dual}
\end{equation}
The saddle point problem \eqref{eq:saddle point problem} 
is commonly solved using the primal-dual hybrid gradient 
(PDHG) algorithm \cite{esser2010general,chambolle2011,Chambolle2016}, 
and by doing so, both the primal problem \eqref{eq:min problem general}
and the dual \eqref{eq:dual} are solved.
We apply the variant of PDHG from \cite{condat2013primal},
which accounts for the sum of composite terms
terms $H_i \circ L_i$.
Given an initial guess for $(w_0,y_{1,0},\ldots,y_{m,0})$
and the parameters $\sigma, \tau > 0$, 
and $\rho \in [\epsilon,2-\epsilon]$ for some $\epsilon > 0$,
each iteration $k \geq 0$ consists of the following steps:
\begin{align}
  &1.\quad 
    \tilde{w}_{k+1} := \prox_{\tau G}
    (w_k - \tau \sum_{i=1}^m L_i^* y_{i,k}),
    \nonumber \\
  &2.\quad 
    w_{k+1} := \rho_k \tilde{w}_{k+1} + (1-\rho_k) w_k,
    \nonumber \\
  &3.\quad
    \forall i = 1,\ldots,m: \quad
    \tilde{y}_{i,k+1} := \prox_{\sigma H_i^*} \left(
      y_{i,k} + \sigma L_i(2 \tilde{w}_{k+1} - w_k)
    \right),
    \nonumber \\
  &4.\quad
    \forall i = 1,\ldots,m: \quad
    y_{i,k+1} := \rho \tilde{y}_{i,k+1} + (1-\rho)y_{i,k}.
    \label{eq:algorithm}
\end{align}
where for a proper, lower semi-continuous, convex function $G$,
$\prox_{\tau G}$ is its proximal operator, defined as:
\begin{equation}
  \prox_{\tau G}(y) := \argmin_x 
  \left\{  
    \frac{1}{2\tau} \|x - y\|_2^2 + G(x) 
  \right\}.
\end{equation}

The iterates $(w_k)_{k \in \mathbb{N}}$
and $(y_{i,k})_{k \in \mathbb{N}}$ ($i=1,\ldots,m$)
are shown to converge 
if the parameters $\sigma$ and $\tau$ are chosen such that
$ \sigma \tau \| \sum_{i=1}^m L_i^* L_i \| \leq 1$
(see \cite{condat2013primal}, Theorem 5.3).
In step 3 in \eqref{eq:algorithm}, we use Moreau's 
identity to obtain $\prox_{\sigma H^*_i}$ 
from $\prox_{H_i/\sigma}$:
\begin{equation}
  \prox_{\sigma H_i^*}(y) + 
  \sigma \prox_{H_i/{\sigma}} (y/{\sigma})
  = y.
\end{equation}

As a stopping criterion, one can use the primal-dual gap
i.e. the difference between the primal objective cost at the
current iterate and the dual objective cost at the current
(dual) iterate:
\begin{equation}
  D_{pd}(w,y_1,\ldots,y_m) = 
  G(w) + \sum_{i=1}^m H_i(L_i w)
  + G^*(-\sum_{i=1}^m L_i^*y_i)
  + \sum_{i=1}^m H_i^*(y_i)
  \label{eq:pd gap}
\end{equation}
Due to strong duality, optimality is reached when the 
primal-dual gap is zero, so a practical stopping criterion
is when the gap reaches a certain threshold set in advance.

Lastly, note that the optimisation is performed jointly over
both $u$ and $v$, which introduces a difficulty for the 
term $H_2(L_2 w)$ in Step 3 above, as this
requires the proximal operator of the joint Kullback-Leibler
divergence $D_{KL}(u,v)$. Similarly, the computation of
the primal-dual gap in \eqref{eq:pd gap} requires the convex
conjugate of the joint Kullback-Leibler divergence.
We describe the details of these computations in
Section~\ref{sec:prox} and Section~\ref{sec:conv conj}
respectively. 

\subsection{Computing the proximal operator of 
  the joint Kullback–Leibler divergence} 
\label{sec:prox}

When writing the optimisation problem in the 
form~\eqref{eq:min problem general}, it is common that
the functions $G$ and $H_i$ ($i=1,\ldots,m$) are ``simple'',
meaning that their proximity operators have a closed form
solution or can be easily computed with high precision.
This is certainly true for $G$ and $H_1$, but not obvious
for the joint Kullback-Leibler divergence.

First, for discrete images
$u=[u_1,\ldots,u_N]^T,[v_1,\ldots,v_N]^T$, 
the definition \eqref{eq:dkl continuous} becomes:
\begin{equation}
  D_{KL}(v,u)= \sum_{j=1}^N 
  u_j - v_j + v_j \log\frac{v_j}{u_j}
  \label{eq:discrete dkl}
\end{equation}
and then:
\begin{align}
  \prox_{\gamma D_{KL}}(u^*,v^*) 
  &= \argmin_{u,v} \left\{
    D_{KL}(u,v) + \frac{1}{2\gamma}
    \left\|
      \begin{bmatrix} u\\v \end{bmatrix} -
      \begin{bmatrix} u^*\\v^* \end{bmatrix}
    \right\|^2_2
  \right\}
  \nonumber\\
  &= \argmin_{u,v} \left\{
    \sum_{j=1}^N
    u_j - v_j + v_j \log\frac{v_j}{u_j}
    + \frac{1}{2\gamma}
    [(u_j-u_j^*)^2 + (v_j-v_j^*)^2]
  \right\}
  \nonumber\\
  &=
  \sum_{j=1}^N \argmin_{u_j,v_j}
  \Phi(u_j,v_j),
\end{align}
where we define the function $\Phi$ as:
\begin{equation}
  \Phi(u_j,v_j) :=
  u_j - v_j + v_j \log\frac{v_j}{u_j}
  + \frac{1}{2\gamma}
  [(u_j-u_j^*)^2 + (v_j-v_j^*)^2].
\end{equation}
To find the minimiser of $\Phi(u_j,v_j)$, we let its gradient
be equal to zero:
\begin{equation}
  \begin{cases}
    \partial_{u_j} \Phi(u_j,v_j) = 0 \\
    \partial_{v_j} \Phi(u_j,v_j) = 0
  \end{cases}
  \iff
  \begin{cases}
    1 - \frac{v_j}{u_j} + \frac{1}{\gamma}(u_j-u_j^*) = 0\\
    \log v_j - \log u_j + \frac{1}{\gamma}(v_j-v_j^*) = 0
  \end{cases}
\end{equation}
In the second equation, we write $u_j$ as a function 
of $v_j$, which we substitute in the first equation 
to obtain:
\begin{equation}
  \begin{cases}
    1 - e^{-\frac{1}{\gamma}(v_j-v_j^*)} + \frac{1}{\gamma}
    \left(v_j e^{\frac{1}{\gamma}(v_j-v_j^*)} - u_j^*\right)
    = 0
    \\
    u_j = v_j e^{\frac{1}{\gamma}(v_j-v_j^*)}
  \end{cases}
\end{equation}
The first equation is then solved using Newton's method,
where the iteration is given by:
\begin{equation}
  v_j^{(k+1)} = v_j^{(k)} - \frac{
    \gamma - \gamma e^{-\frac{1}{\gamma}(v_j^{(k)}-v_j^*)}
    + v_j^{(k)} e^{\frac{1}{\gamma} (v_j^{(k)} - v_j^*)} 
    - u_j^*
  }{
    e^{-\frac{1}{\gamma}(v_j^{(k)}-v_j^*)} 
    + (1 + \frac{1}{\gamma} v_j^{(k)})
    e^{\frac{1}{\gamma}(v_j^{(k)}-v_j^*)} 
  }.
  \label{eq:newton v}
\end{equation}

\subsection{Computing the convex conjugate of 
  the joint Kullback–Leibler divergence}
\label{sec:conv conj}

We compute the convex conjugate of the discrete joint
Kullback-Leibler divergence $D_{KL}(v,u)$ 
in \eqref{eq:discrete dkl} for $u,v \in [l_1,l_2]^N$:
\begin{align}
  D^*_{KL}(v^*,u^*) &= \sup_{v,u \in [l_1,l_2]^N} \left\{
    \left\langle 
      \begin{bmatrix} u\\v \end{bmatrix},
      \begin{bmatrix} u^*\\v^* \end{bmatrix}
    \right\rangle
    - D_{KL}(v,u)
  \right\}
  \nonumber \\
  &= \sup_{v,u \in [l_1,l_2]^N } \left\{
    \sum_{j=1}^N u_j u_j^* + v_j v_j^*
    - u_j + v_j - v_j \log \frac{v_j}{u_j}
  \right\}
  \nonumber \\
  &= \sum_{j=1}^N \sup_{v_j,u_j \in [l_1,l_2]} \Psi(v_j,u_j),
  \label{eq:dkl psi}
\end{align}
where $\Psi$ is defined as:
\begin{equation}
  \Psi(v_j,u_j) := 
    u_j u_j^* + v_j v_j^*
    - u_j + v_j - v_j \log \frac{v_j}{u_j}.
\end{equation}
To solve the optimisation problem on the last line
in \eqref{eq:dkl psi}, we write the KKT conditions
(where we use $u,v$ instead of $u_j, v_j$ to simplify
the notation:
\begin{align}
  -\nabla \Psi(v,u) 
  &+ \sum_{i=1}^4 \mu_i \nabla g_i(v,u) = 0,
  \label{eq:psi mu g}
  \\
  g_i(v,u) &\leq 0, \quad \forall i=1,\ldots,4,
  \\
  \mu_i &\geq 0, \quad \forall i=1,\ldots,4,
  \\
  \mu_i g_i(v,u) &= 0, \quad \forall i=1,\ldots,4.
  \label{eq:compl}
\end{align}
where the functions $g_i$ correspond to the bound constraints:
\begin{align}
  g_1(v,u) &= u-l_2; \\
  g_2(v,u) &= v-l_2; \\
  g_3(v,u) &= -u+l_1; \\
  g_4(v,u) &= -v+l_1; \\
\end{align}
Noting that \eqref{eq:psi mu g} is equivalent to:
\begin{align}
  &-u^* + 1 - \frac{v}{u}  + \mu_1 - \mu_3 = 0, 
  \\
  &-v^* + \log v - \log u + \mu_2 - \mu_4 = 0,
\end{align}
we solve the last two equations by using the 
complementarity conditions \eqref{eq:compl} for 
cases when the Lagrange multipliers $\mu_i$ are 
zero or non-zero.

%% file: 5_results.tex
\section{Numerical results}
\label{sec:numerics}

In this section, we describe a number of numerical experiments
that illustrate the performance of our deconvolution method.
We start with four examples of simulated data, where we are 
able to quantify the reconstructed image in relation to the
known ground truth image. Then, we show how our method performs
on microscopy data, where we reconstruct an image of spherical
beads and a sample of a Marchantia thallus.

  \subsection{Simulated data}

We consider four images of size $128\times125\times64$: 
a $5 \times 5 \times 5$ grid of beads where the 
effect of the light-sheet in the $z$ coordinate and the
shape of the objective PSF are noticeable, 
a piecewise constant image of ``steps''
where the Poisson noise affects each step differently 
based on intensity, and an image that replicates a 
realistic biological samples of tissue.
These are shown in the top row of Figure~\ref{fig:simulated data}.

To obtain the measured data, we proceed as follows. Given the
ground truth image $u_0$, the forward operator described in
Section~\ref{sec:ifm} is applied to obtain the blurred image
$Lu_0$. The parameters for the forward model are taken to be
those of the microscope used in the experimental setup, and
are given in Table~\ref{table:microscope params}.
Then, the image corrupted with a mixture 
of Poisson and Gaussian noise. For the vectorised image
$Lu_0$, at each pixel $i=1,\ldots,N$, the Poisson noise
component follows the Poisson distribution with 
parameter $(Lu_0)_i$ and the additive Gaussian component
has zero mean and standard deviation $\sigma_G=10$.
The original image, which has intensity in $[0,1]$
is scaled so that the intensity of $Lu_0$ is 
in $[0,2000]$, to replicate a realistic scenario
for the Poisson noise intensity. The resulting simulated
measured data is shown in the bottom row of 
Figure~\ref{fig:simulated data}.

\begin{table}[h]
  \centering
   \begin{tabular}{| l | l | l |} 
     \hline
     Parameter & Value & Description, units \\
     \hline
     $n$ & $1.35$ & refractive index \\
     $NA_h$ & $1$ & numerical aperture (objective lens) \\
     $NA_l$ & $0.25$ & numerical aperture (light-sheet) \\
     $\lambda_h$ & $0.525$ & wave length (objective lens), $\mu m$ \\
     $\lambda_l$ & $0.488$ & wave length (light-sheet), $\mu m$ \\
     $px_x$ & $0.3250$ & pixel size ($x$), $\mu m$ \\
     $px_y$ & $0.3250$ & pixel size ($y$), $\mu m$ \\
     $step_z$ & $1$ & light-sheet step size ($z$), $\mu m$ \\
     \hline
   \end{tabular}
   \caption{Forward model parameters used 
     in Section~\ref{sec:numerics}.}
   \label{table:microscope params}
\end{table}

\begin{figure}[h!]
  \centering
  \includegraphics[width=0.3\textwidth]{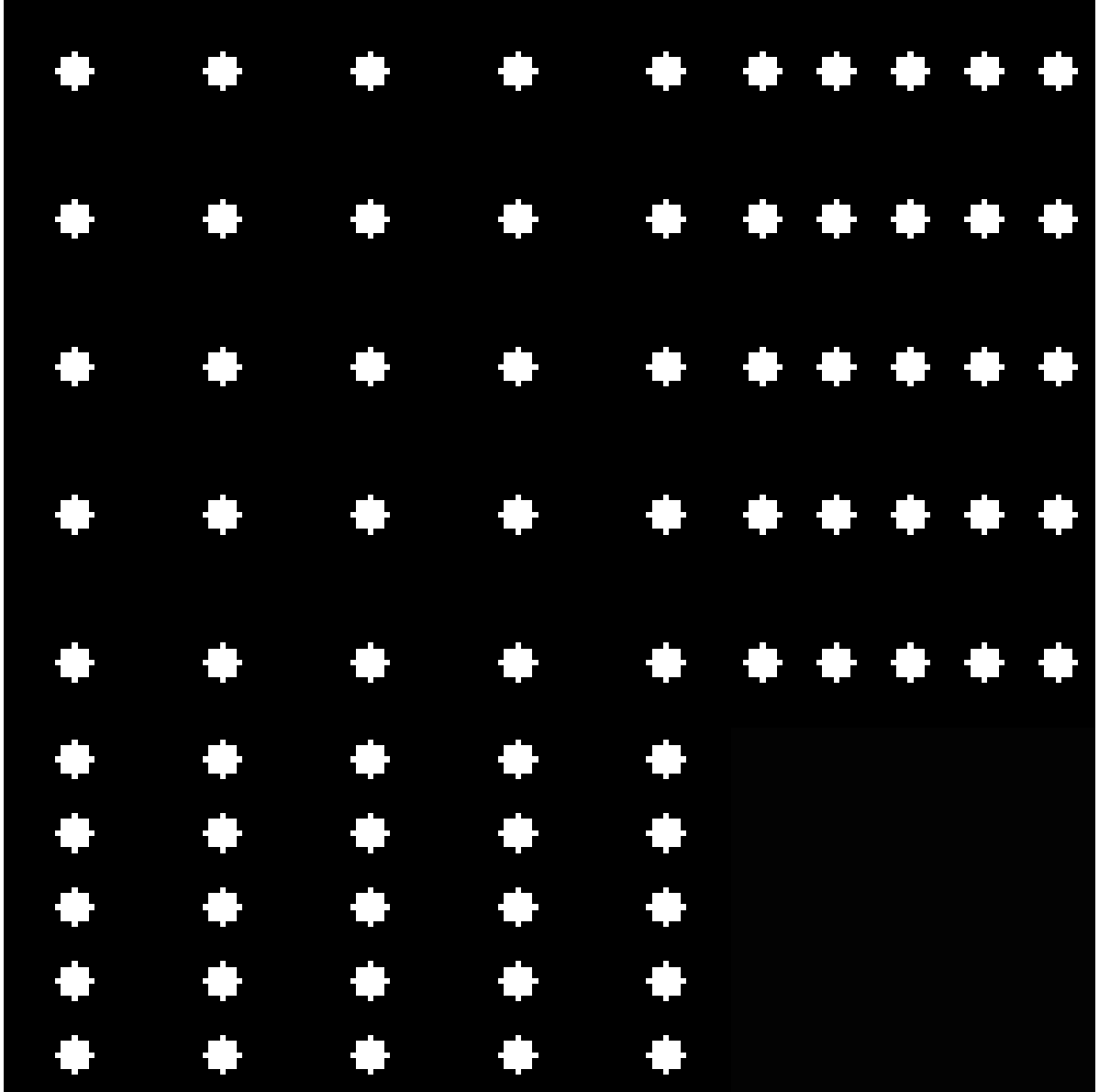}
  \includegraphics[width=0.3\textwidth]{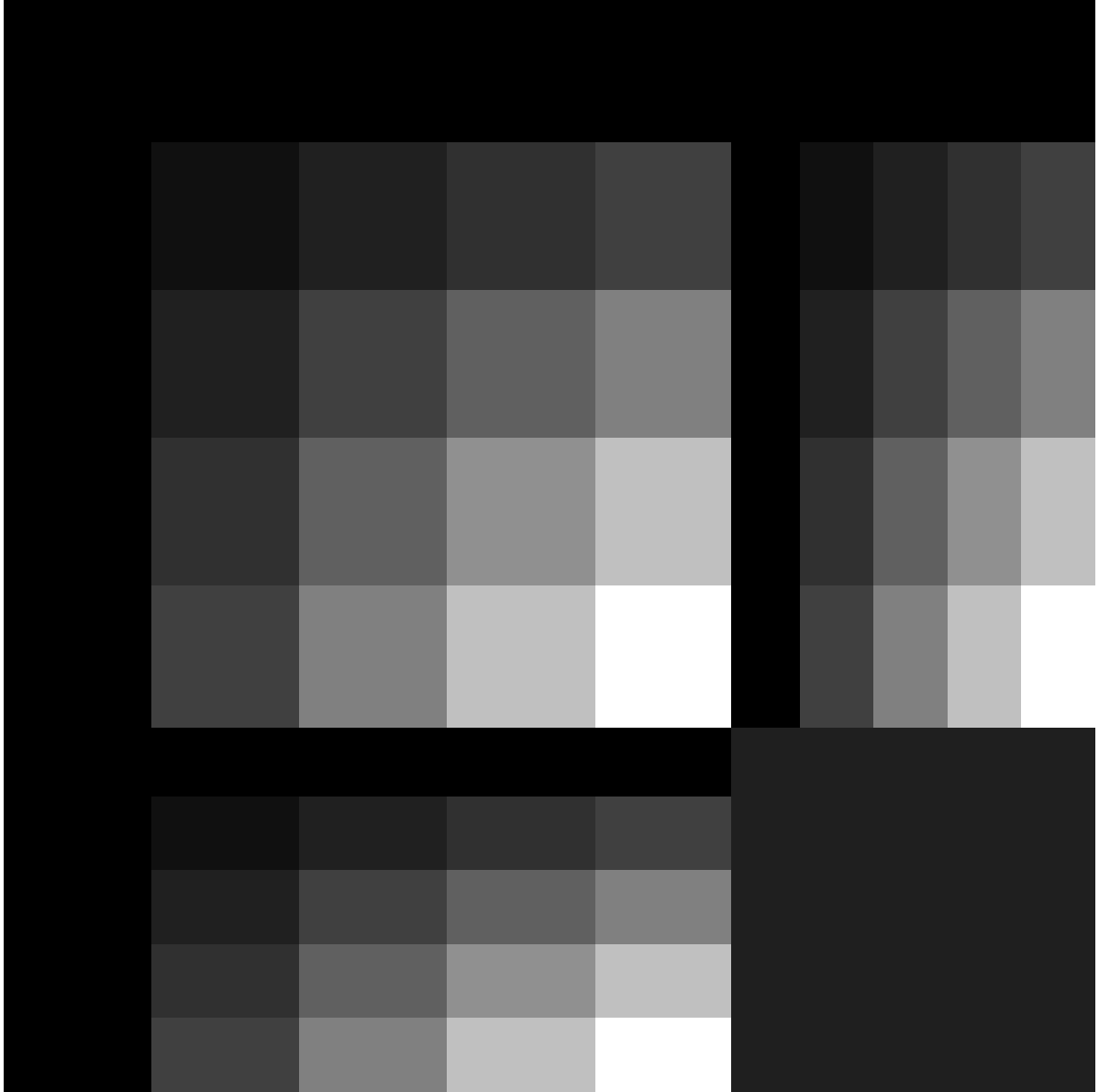}
  \includegraphics[width=0.3\textwidth]{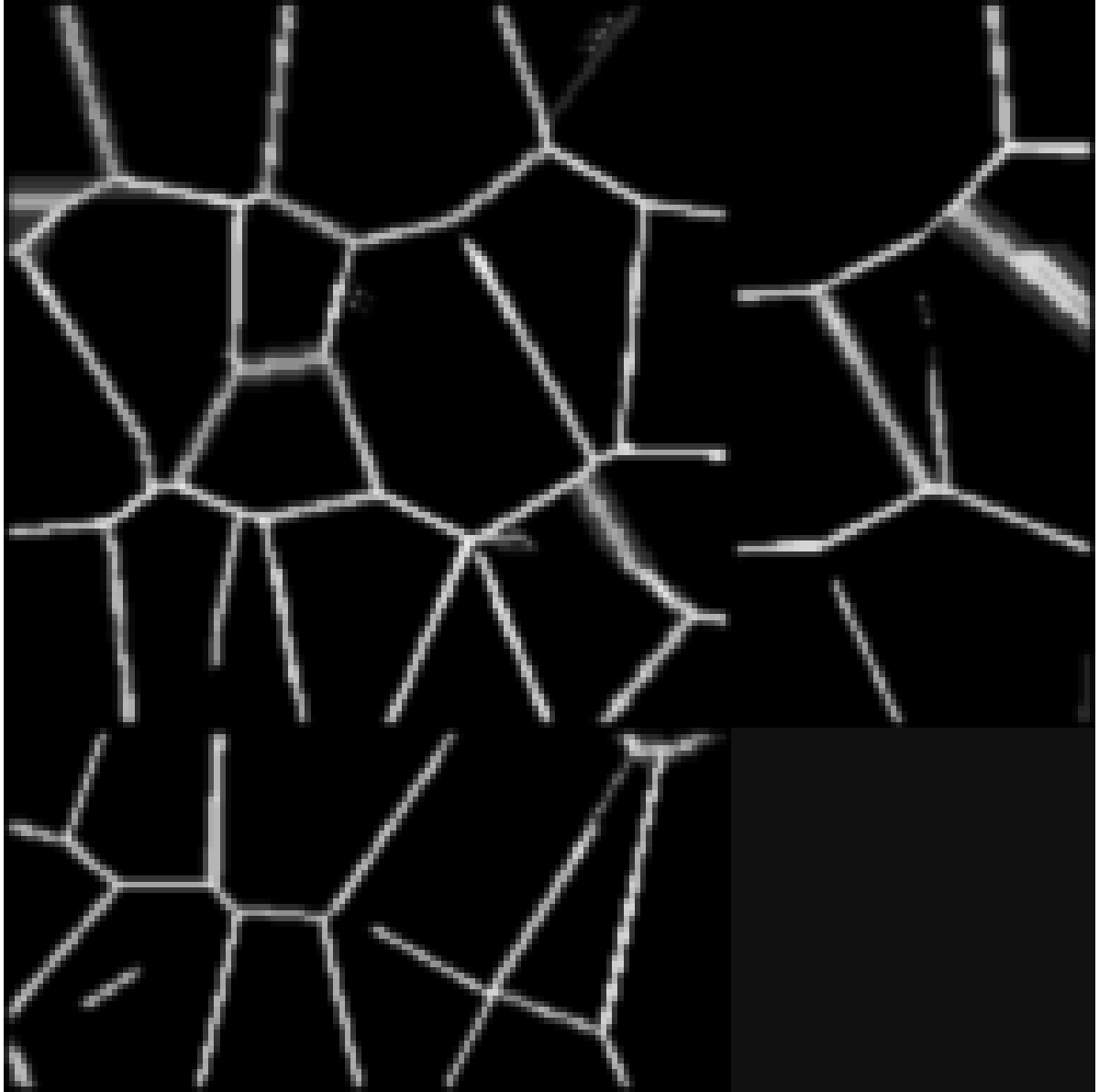}
  \includegraphics[width=0.3\textwidth]{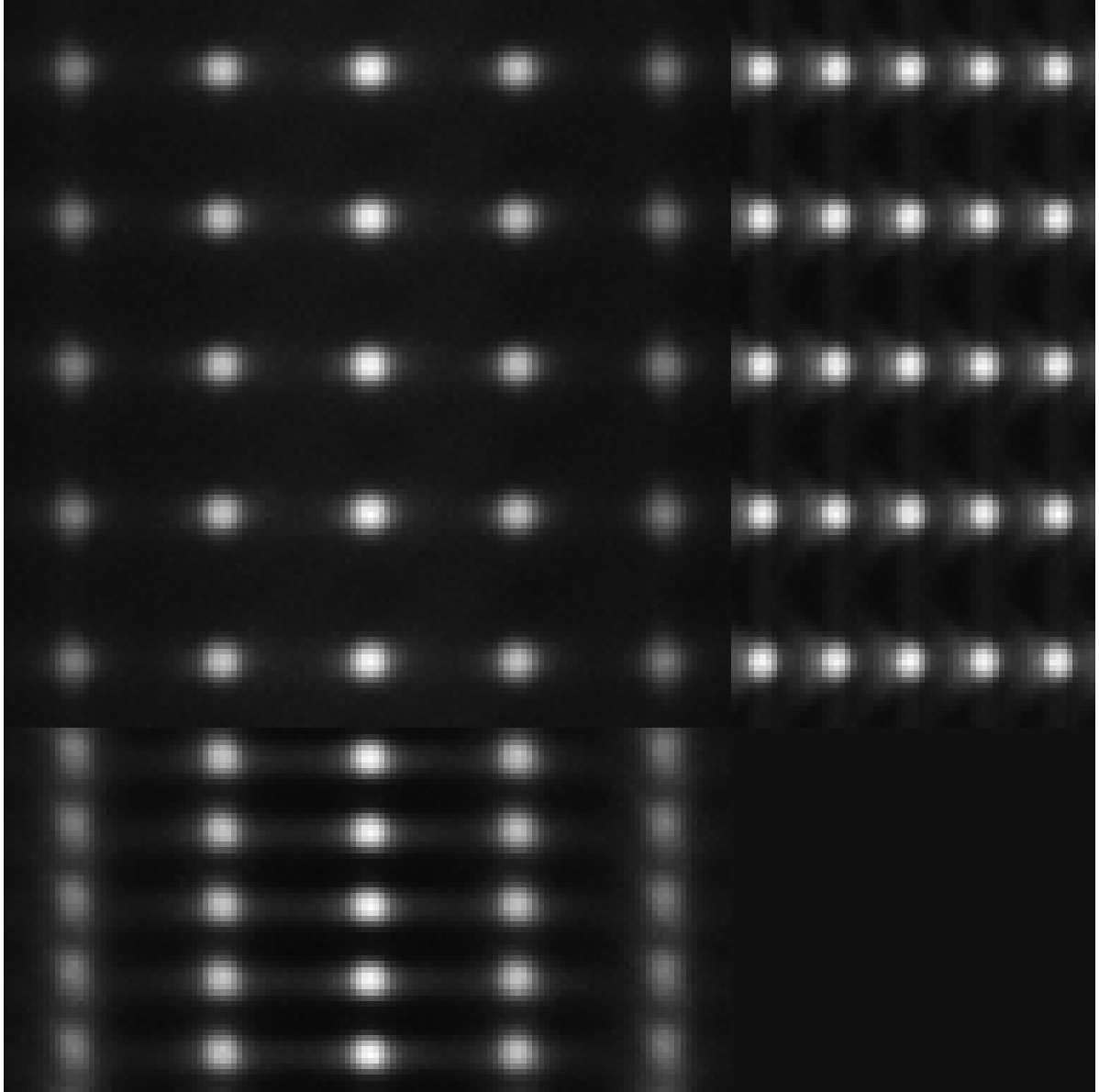}
  \includegraphics[width=0.3\textwidth]{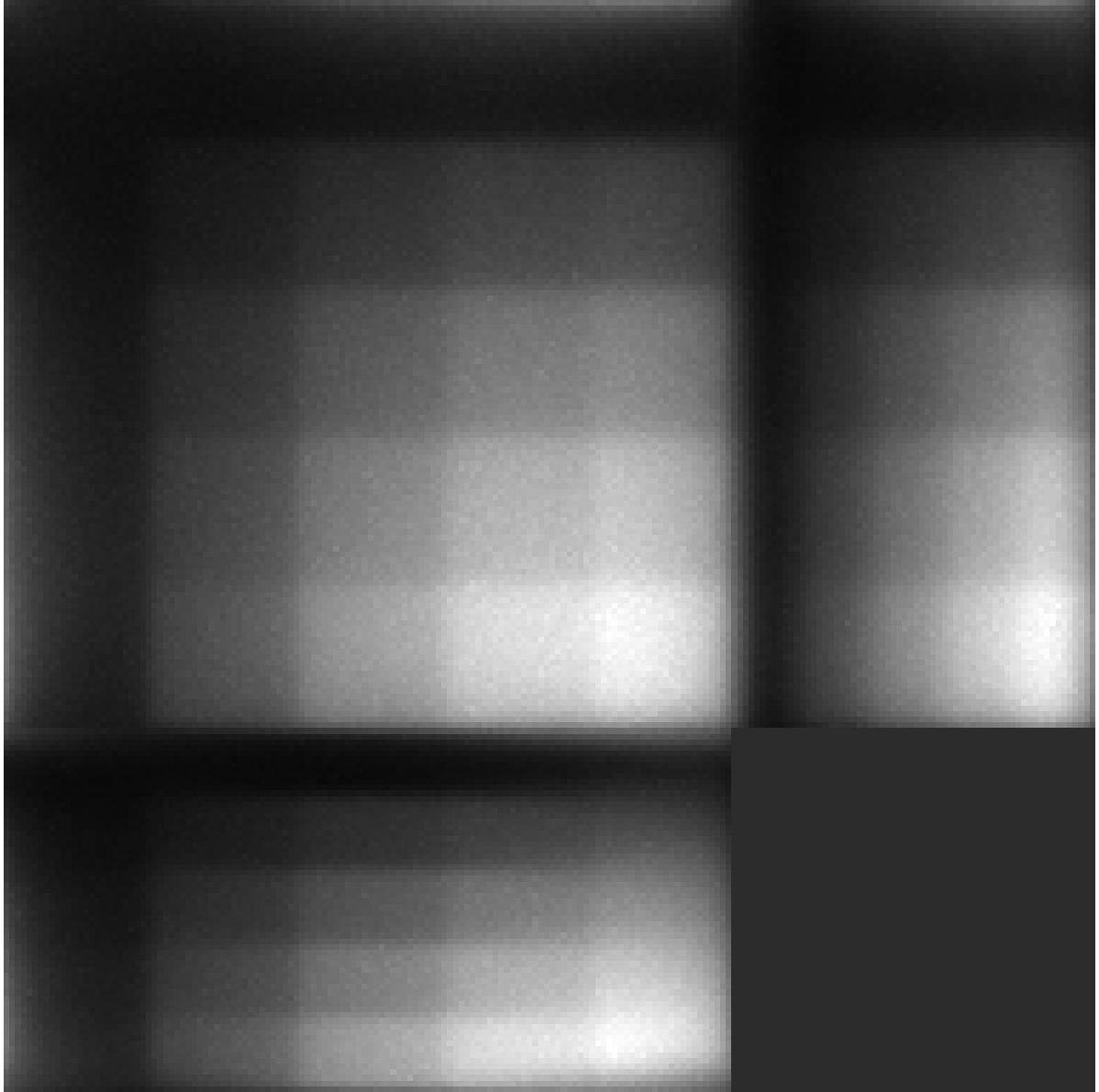}
  \includegraphics[width=0.3\textwidth]{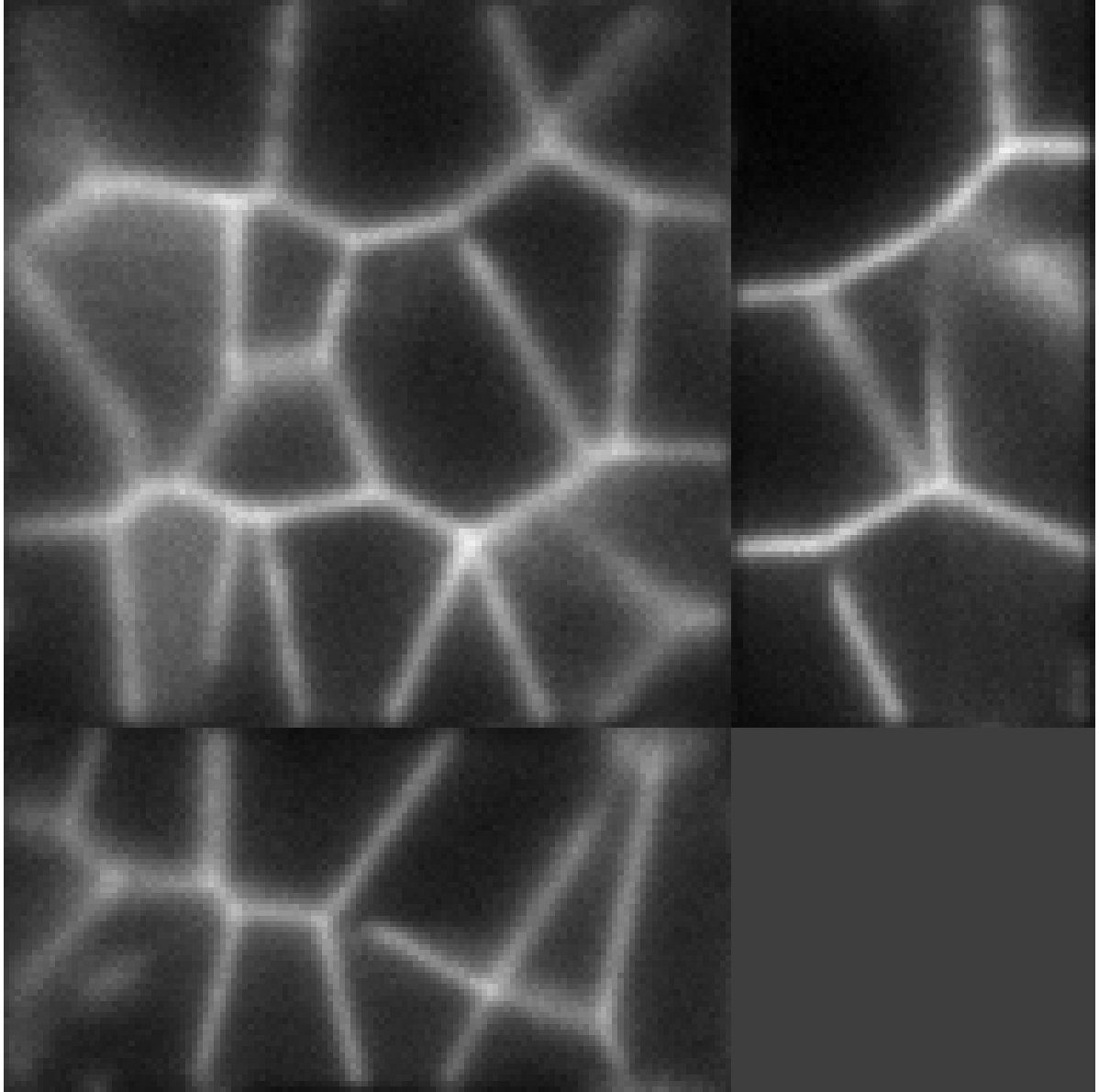}

  \caption{Ground truth (top row) and measured images
    (bottom row), shown using maximum intensity projections,
    except for tissue, where slices in each direction 
    are shown.}
  \label{fig:simulated data}
\end{figure}

We compare the 
reconstruction obtained using the proposed approach,
which we will refer to as LS-IC (light-sheet - infimal convolution), 
with the
reconstructions obtained by using an $L^2$ data fidelity 
term instead of the infimal convolution term, or using 
a convolution operator corresponding to the objective PSF
instead of the light-sheet forward model from 
Section~\ref{sec:ifm}. Specifically, we compare the solution
of \eqref{eq:min problem impl} with the solutions to the
following problems, all solved using PDHG as
described in Section~\ref{sec:solving min problem}:
\begin{align}
  &\min_{u} \left\{
    \alpha \TV(u) +
    \frac{1}{2\sigma_G^2} \|f-Hu\|_2^2 
    + \chi_{[0,B]^{2N}}([u, v]^T)
  \right\},
  \tag{PSF-L2}
  \label{eq:psf l2}\\
  &\min_{u,v} \left\{
    \alpha \TV(u) + 
    \frac{1}{2\sigma_G^2} \|f-v\|_2^2 + D_{KL}(v,Hu) 
    + \chi_{[0,B]^{2N}}([u, v]^T)
  \right\},
  \tag{PSF-IC} 
  \label{eq:psf ic}\\
  &\min_{u} \left\{
    \alpha \TV(u) + 
    \frac{1}{2\sigma_G^2} \|f-Lu\|_2^2 
    + \chi_{[0,B]^{2N}}([u, v]^T)
  \right\},
  \tag{LS-L2}
  \label{eq:ls l2}
\end{align}
where $H$ is the convolution operator with the 
detection objective PSF $h_z$ as given
in \eqref{eq:psf zernike}.

For each test image and each method above, the PDHG
parameters $\rho$ and $\sigma$ used are given 
in Table~\ref{table:pdhg params sim} and $\tau$ is set
to $\tau = 1/\sigma \| \sum_{i=1}^m L_i^* L_i \|$ 
to ensure convergence according to Theorem 5.3
in \cite{condat2013primal}. 
As a stopping criterion, we used the primal-dual gap~\eqref{eq:pd gap},
normalised by the number of pixels $N$ and the dynamic range
of the measured image $f$:
\begin{equation}
  \tilde{D}_{pd} = \frac{D_{pd}}{N \cdot \max_{j=1,\ldots,N} f_j},
\end{equation}
with a threshold of $10^{-6}$ and a maximum number 
of $10000$ iterations.

\begin{table}[h]
  \tabcolsep=0.09cm
  \centering
   \begin{tabular}{| c | c c c | c c c | c c c | c c c |} 
     \hline
      method & & LS-IC & & & \ref{eq:ls l2} & & & \ref{eq:psf ic} & & & \ref{eq:psf l2} & \\
      image & beads & steps & tissue & beads & steps & tissue & beads & steps & tissue & 
      beads & steps & tissue \\
     \hline
     $\rho$ & 0.9 & 0.9 & 0.9 & 0.9 & 0.9 & 0.9 & 0.9 & 0.9 & 0.8 & 0.9 & 0.9 & 0.9 \\
     $\sigma$ & 0.0001 & 0.0001 & 0.00001 & 0.0001 & 0.001 & 0.0001 & 0.0001 & 0.0001 & 0.0001 & 0.0001 & 0.001 & 0.0001 \\
     \hline
   \end{tabular}
   \caption{Values of the PDHG parameters $\rho$ and $\sigma$
     used in the numerical experiments with simulated data.}
   \label{table:pdhg params sim}
\end{table}

The results of the four methods applied to the test 
images are given in Figure~\ref{fig:simulated results} and 
quantitative results are given in Table~\ref{table:simulated results}.
For each test image and each method, the regularisation parameter
has been chosen to optimise the normalised $l^2$ error and the 
structural similarity index (SSIM) respectively.

We note that \ref{eq:psf l2} and \ref{eq:psf ic} perform
particularly poorly, highlighting the importance of an
accurate representation of the image formation model
instead of simply using the detection objective PSF
as the forward operator. Comparing LS-IC and \ref{eq:ls l2},
we see better results when using the infimal convolution
data fidelity for the beads and the steps image,
both visually and quantitatively.
The deblurring is performed better on the beads image, 
while on the steps image we see a better denoising
effect, especially along the edges in the image.
For the tissue image, both fidelities give 
comparable results, but as we see in Figure~\ref{fig:simulated discrepancy}, 
when the ground truth is not known, choosing $\alpha$
using the discrepancy principle gives a better
result for the infimal convolution model.

The reconstructions shown in Figure~\ref{fig:simulated
  discrepancy} are obtained by applying the discrepancy
principle corresponding to each method. For LS-IC, we choose
a value of $\alpha$ such that it satisfies a variation of
the discrepancy principle given in \eqref{eq:discr_pr1},
where we enforce that the single noise fidelities are
bounded by their respective noise bounds, rather than the sum
of the fidelities being bounded by the sum of the noise
bounds, as stated in \eqref{eq:discr_pr1}.
While both versions give good results, we found the former
to give more accurate reconstructions. Here, the bound on
the Poisson noise is set to $\frac12$, motivated by 
the following lemma from~\cite{zanella2009}, which gives the
expected value of the Kullback-Leibler divergence:
\begin{lem}
  Let $Y_{\beta}$ be a Poisson random variable with 
  expected value $\beta$ and consider the function:
  \begin{equation*}
   F(Y_{\beta}) = 2\left\{ 
     Y_{\beta} \log\left(\frac{Y_{\beta}}{\beta}\right)
     +\beta - Y_{\beta}
   \right\}.
  \end{equation*}
  Then, for large $\beta$,the following estimate of the 
  expected value of $F(Y_{\beta})$ holds:
  \begin{equation*}
   \mathbb{E}[F(Y_{\beta})] 
   = 1 + \mathcal{O}\left(\frac{1}{\beta}\right).
  \end{equation*}  
  \label{lem:kld pois}
\end{lem}
  
The experiments were run using 
Matlab version R2020b Update 2 (9.9.0.1524771) 64-bit
in Scientific Linux 7.9
on a machine with
Intel Xeon E5-2680 v4 2.40 GHz CPU,
256 GB memory and 
Nvidia P100 16 GB GPU.
The running times, averaged over 5
runs for each method and each image, are given
in Table~\ref{table:times}.

\begin{table}[h!]
  \centering
   \begin{tabular}{| l | c c | c c | c c |} 
     \hline
     image  & beads & & steps & & tissue & \\
     error metric  & $l_2$ & SSIM & $l_2$ & SSIM & $l_2$ & SSIM \\
     \hline
     \ref{eq:psf l2} & 1.74 & 0.845 & 0.499 & 0.561 & 1.57 & 0.592 \\
     \ref{eq:psf ic} & 1.54 & 0.844 & 0.324 & 0.659 & 1.65 & 0.582 \\
     \ref{eq:ls l2} & 0.282 & 0.982 & 0.055 & 0.971 & 0.301 & 0.951 \\
     LS-IC & 0.258 & 0.983 & 0.012 & 0.998 & 0.349 & 0.931 \\
     \hline
   \end{tabular}
   \caption{Results of the numerical experiments 
     on simulated data, with the regularisation 
     parameter $\alpha$ chosen to optimise the normalised
     $l_2$ error and the SSIM respectively.}
   \label{table:simulated results}
\end{table}

\begin{table}[h!]
  \centering
 
  \begin{tabular}{| c | r r r|}
    \hline
    image & beads & steps & tissue \\
    \hline
    \ref{eq:psf l2} & 233  & 1793 & 903 \\
    \ref{eq:psf ic} & 689 & 1077 & 1805 \\
    \ref{eq:ls l2} & 2913 & 2194 & 2273 \\
    LS-IC & 972  & 601 & 850 \\
    \hline
  \end{tabular}
  \caption{Running times for each method and each 
    simulated test image, averaged over 5 runs, in seconds. 
    The minimisation is stopped when the primal-dual gap is 
    lower than $10^{-6}$ or the the maximum number of
    $10000$ iterations is reached.}
  \label{table:times}
\end{table}

\begin{figure}[H]
  \centering
  \includegraphics[width=0.3\textwidth]{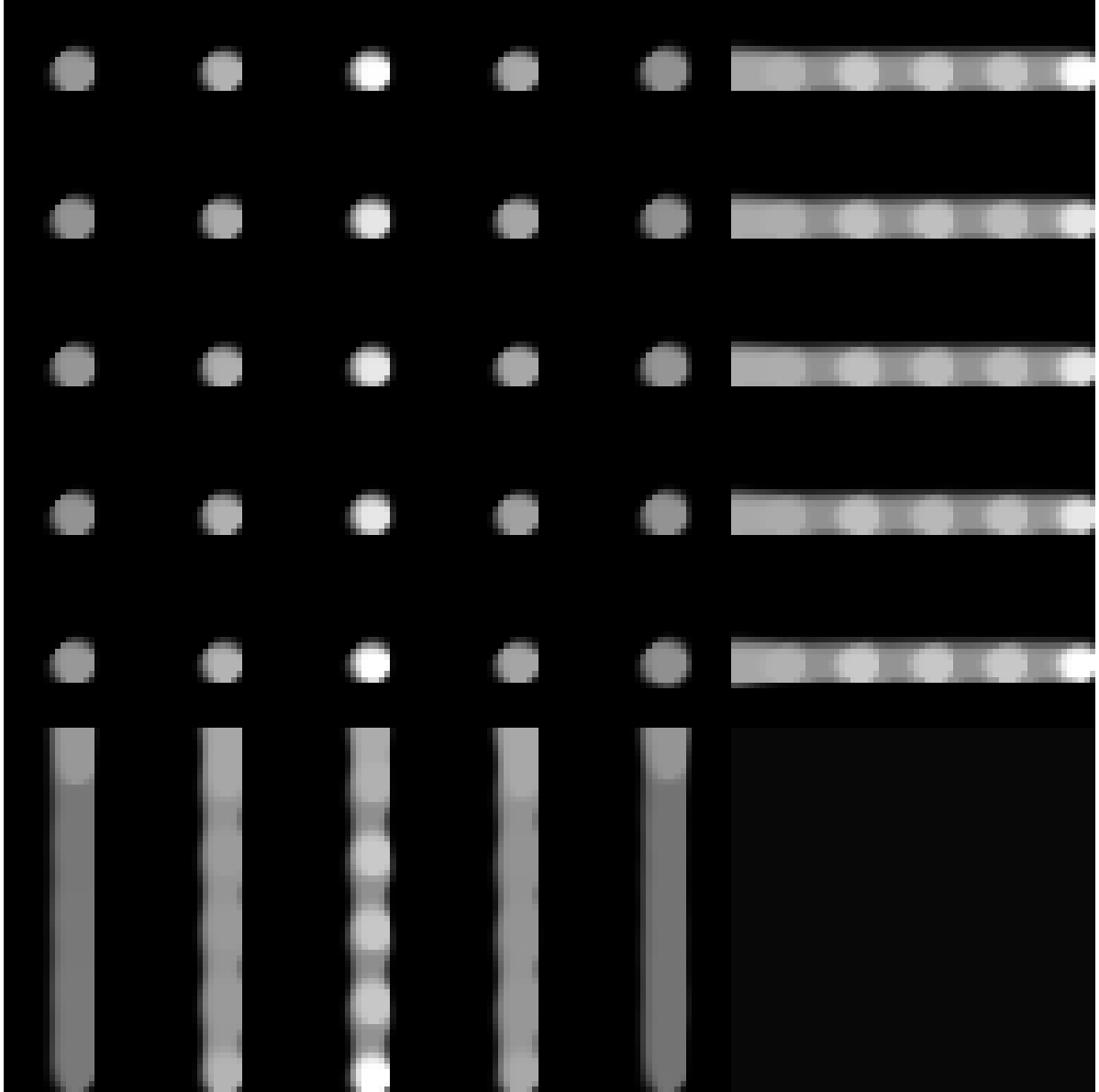}
  \includegraphics[width=0.3\textwidth]{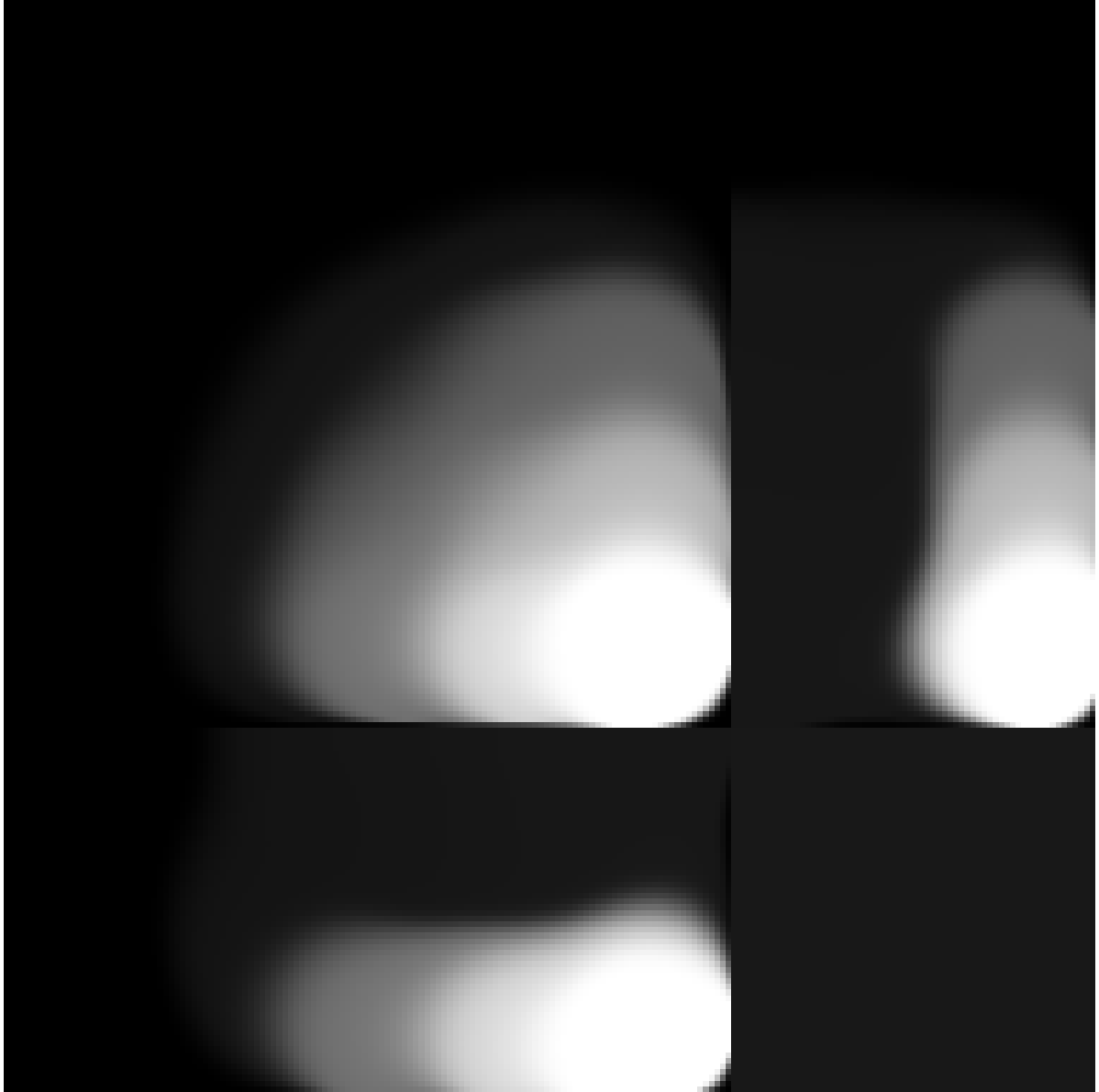}
  \includegraphics[width=0.3\textwidth]{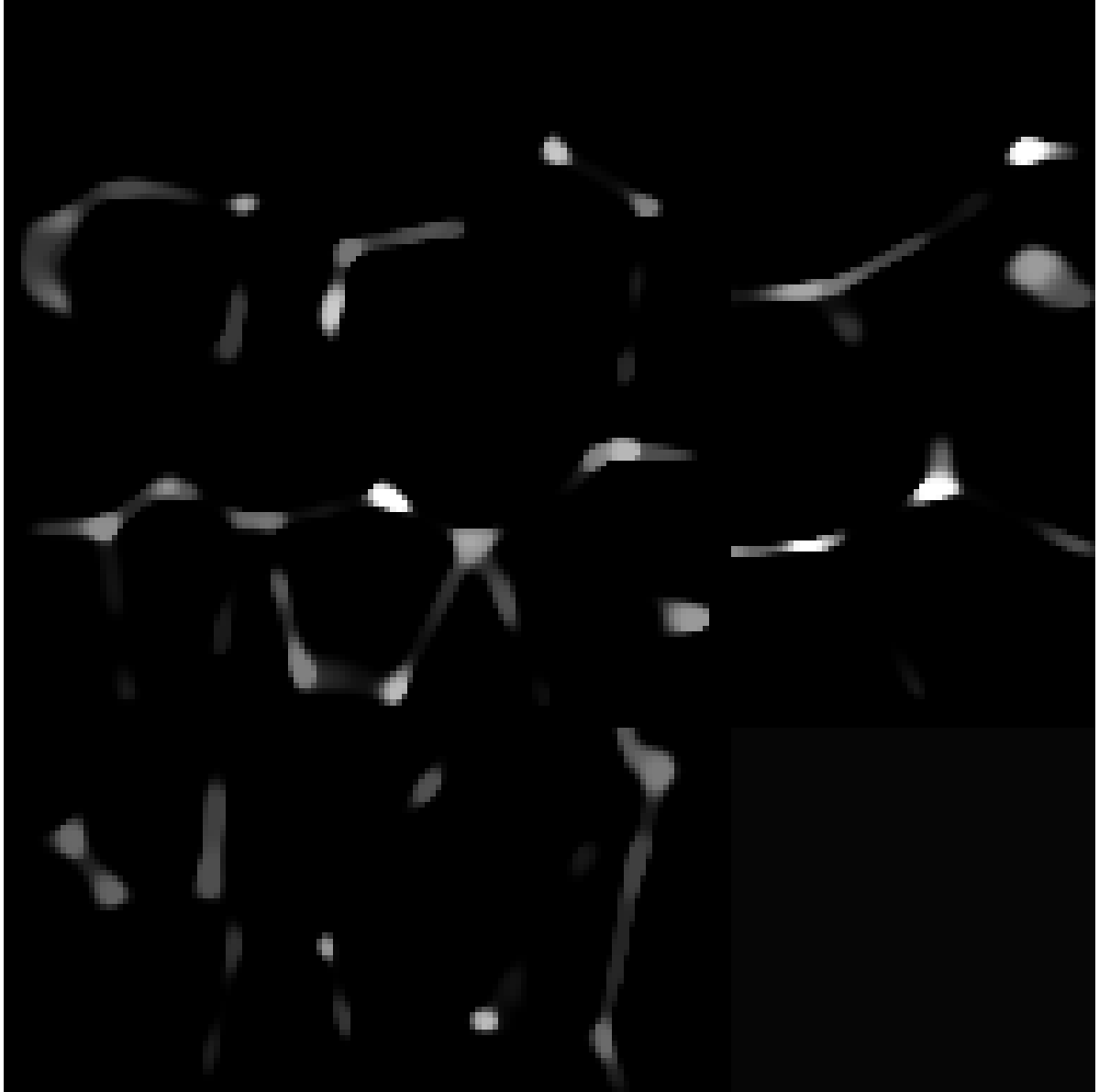}
  \includegraphics[width=0.3\textwidth]{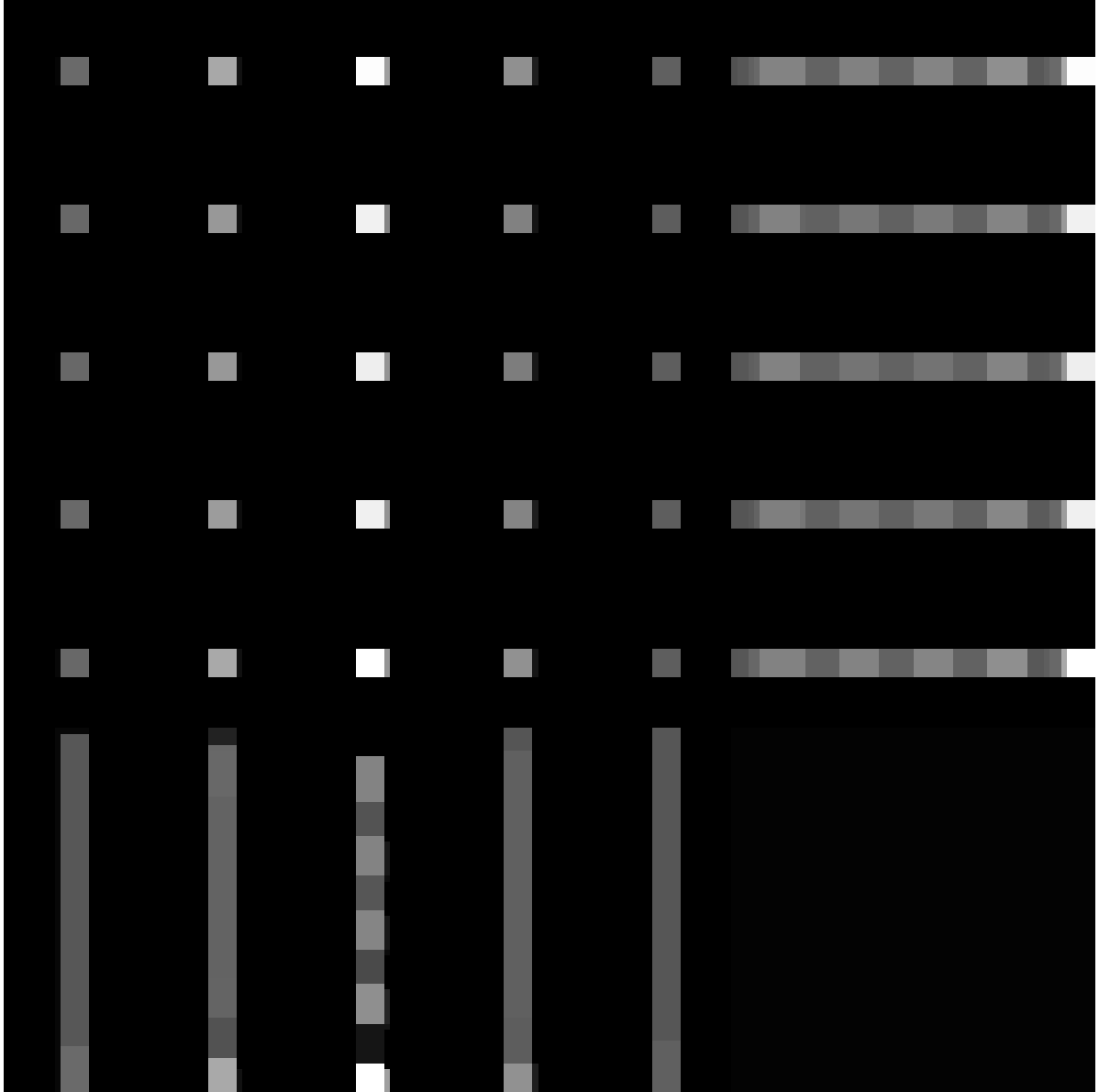}
  \includegraphics[width=0.3\textwidth]{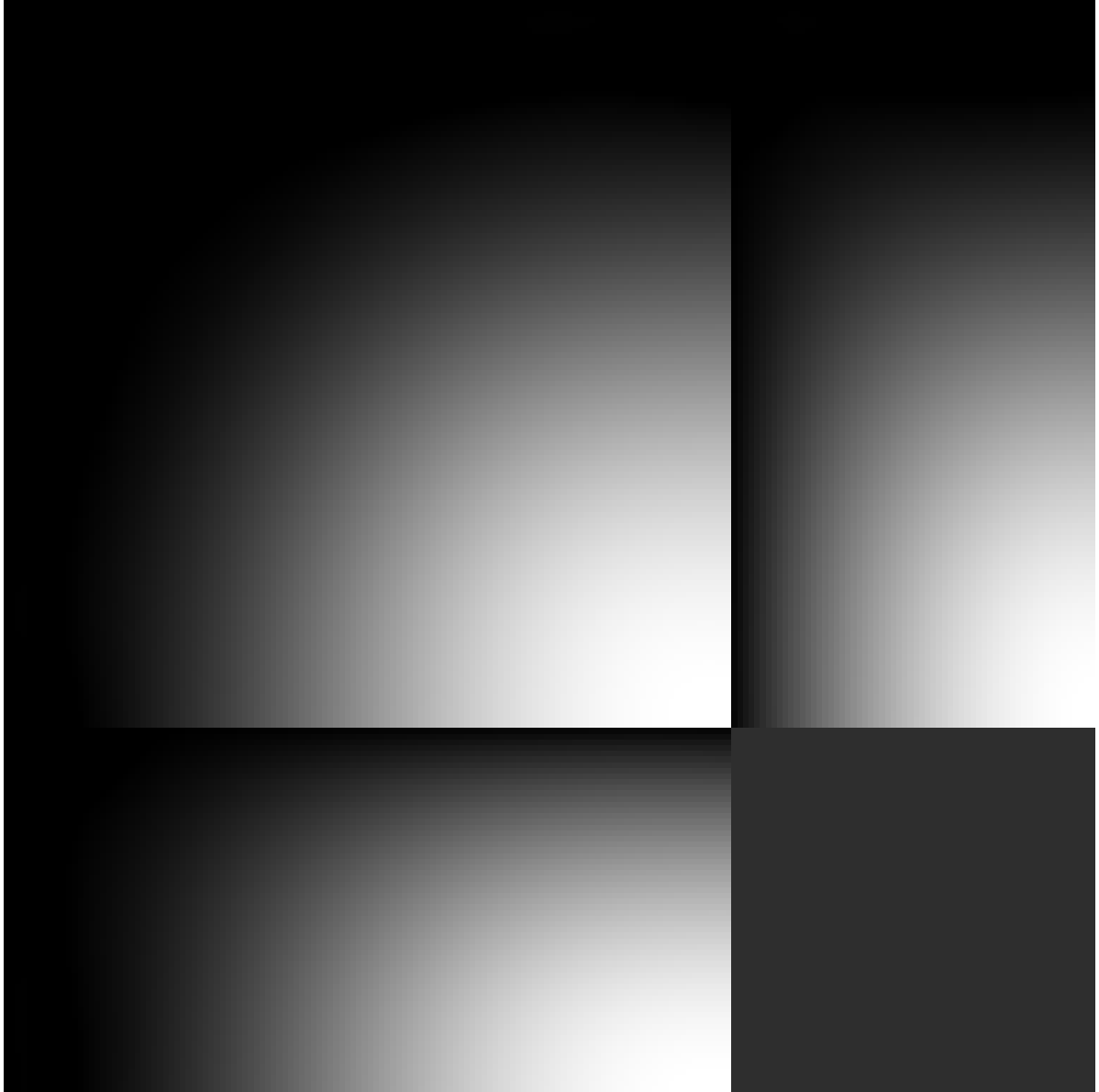}
  \includegraphics[width=0.3\textwidth]{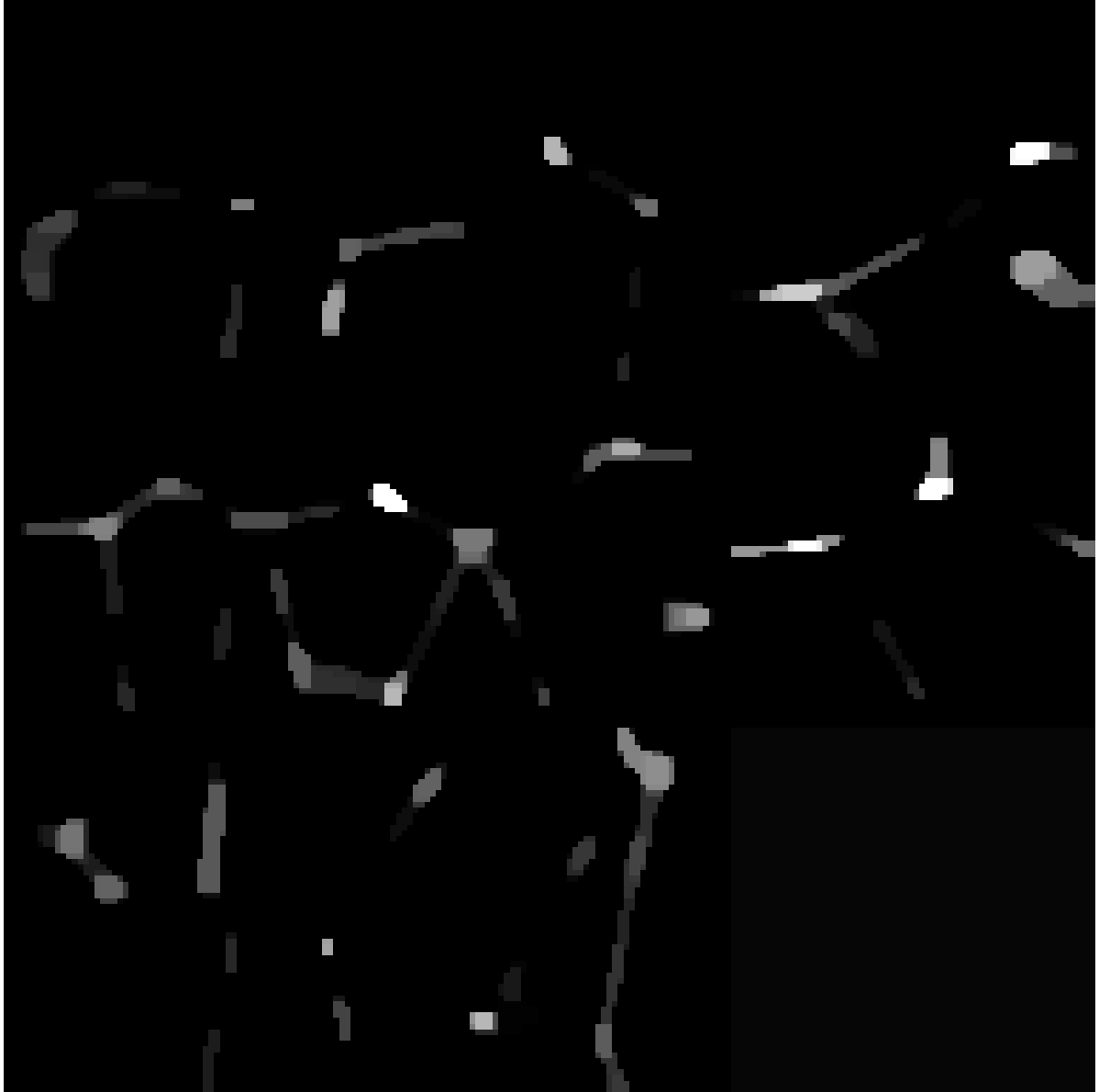}
  \includegraphics[width=0.3\textwidth]{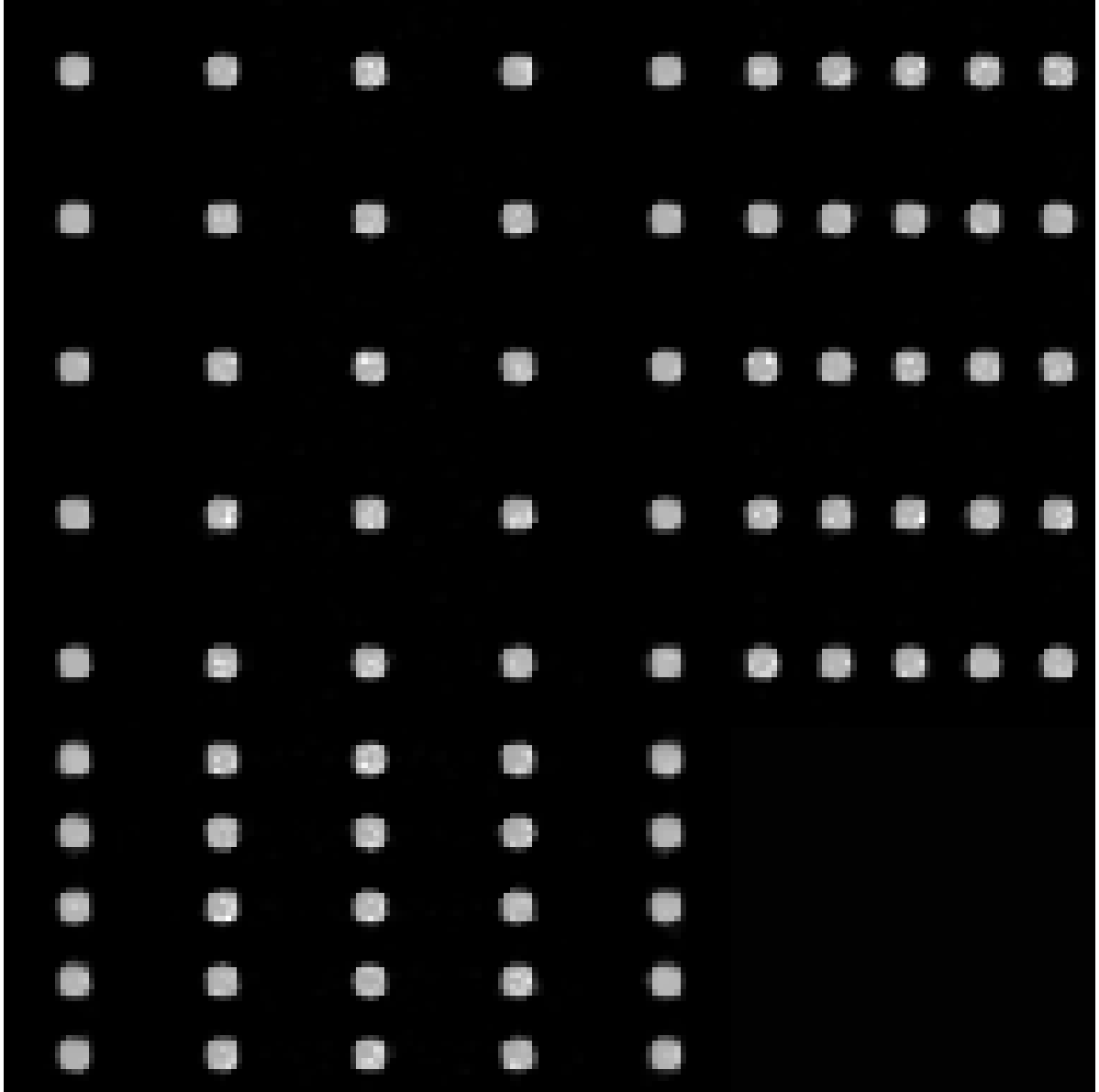}
  \includegraphics[width=0.3\textwidth]{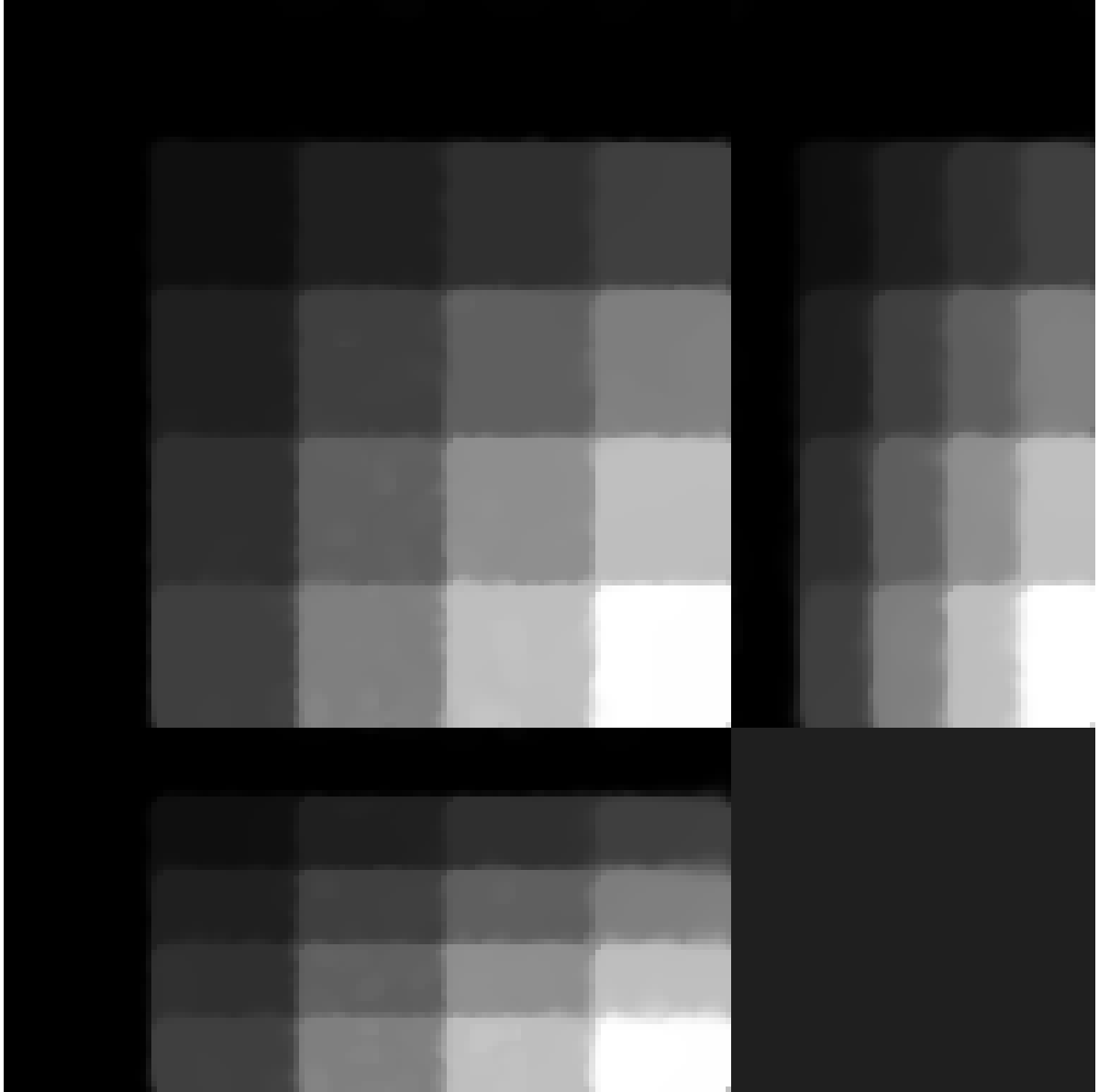}
  \includegraphics[width=0.3\textwidth]{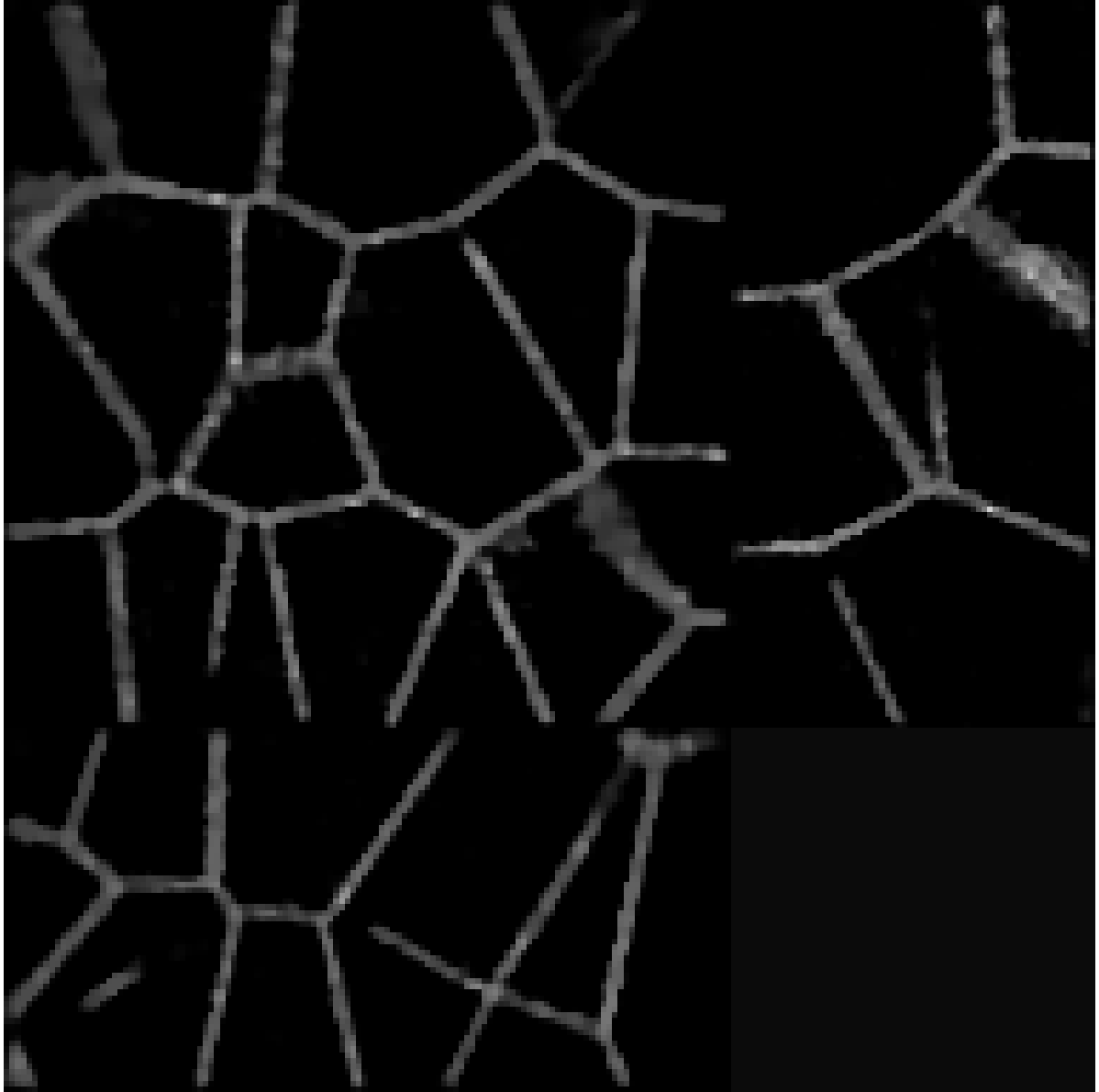}
  \includegraphics[width=0.3\textwidth]{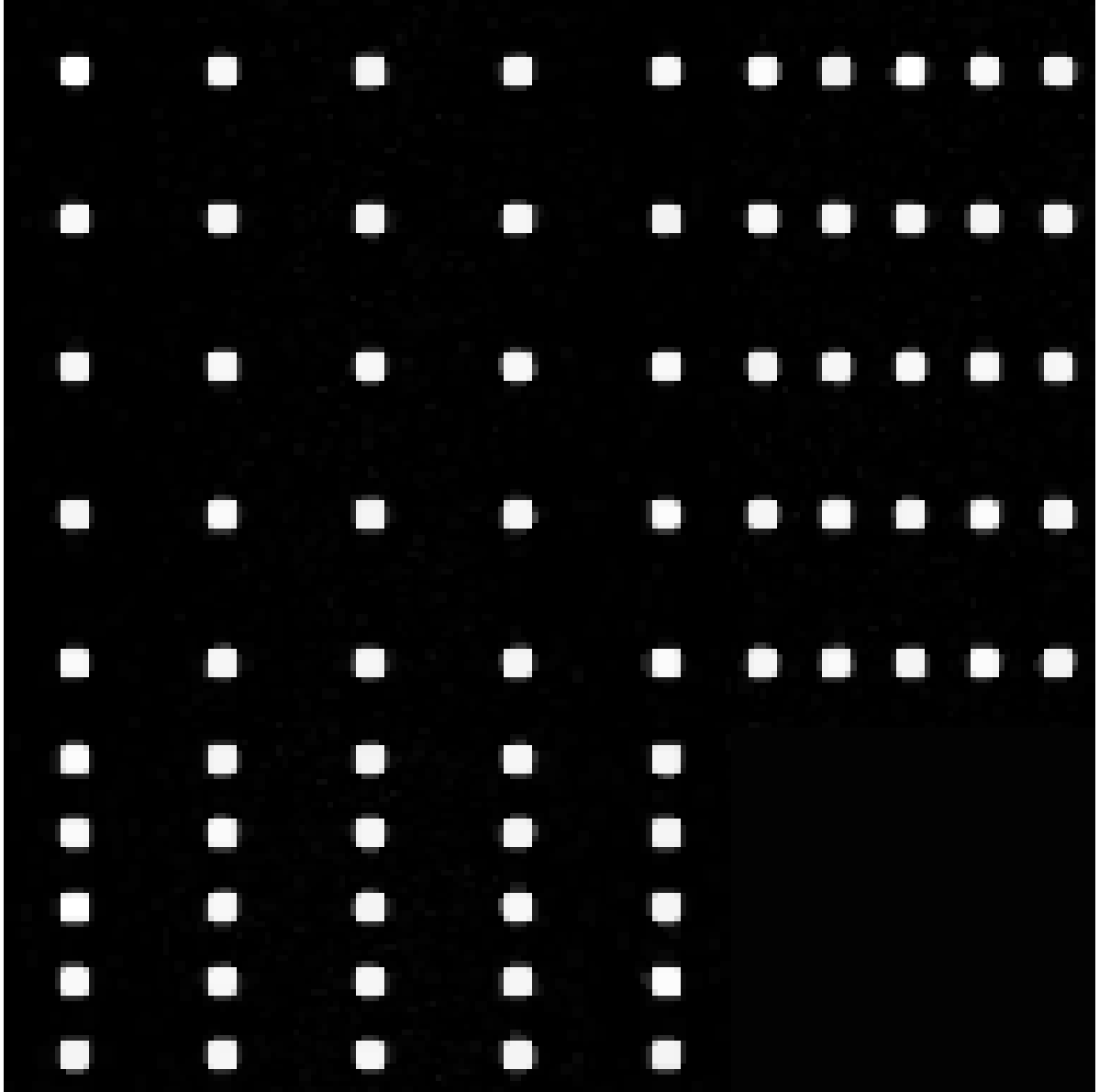}
  \includegraphics[width=0.3\textwidth]{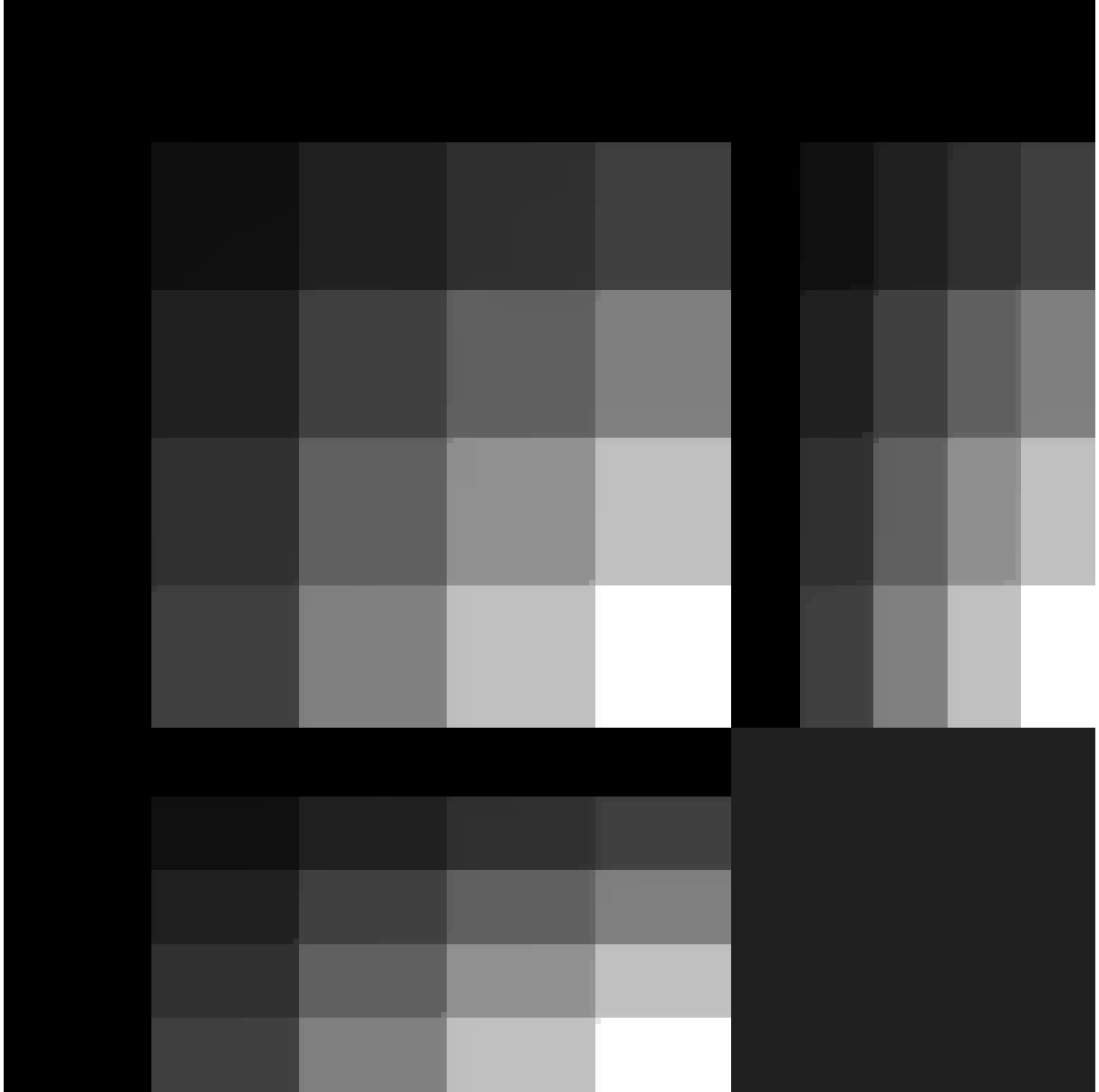}
  \includegraphics[width=0.3\textwidth]{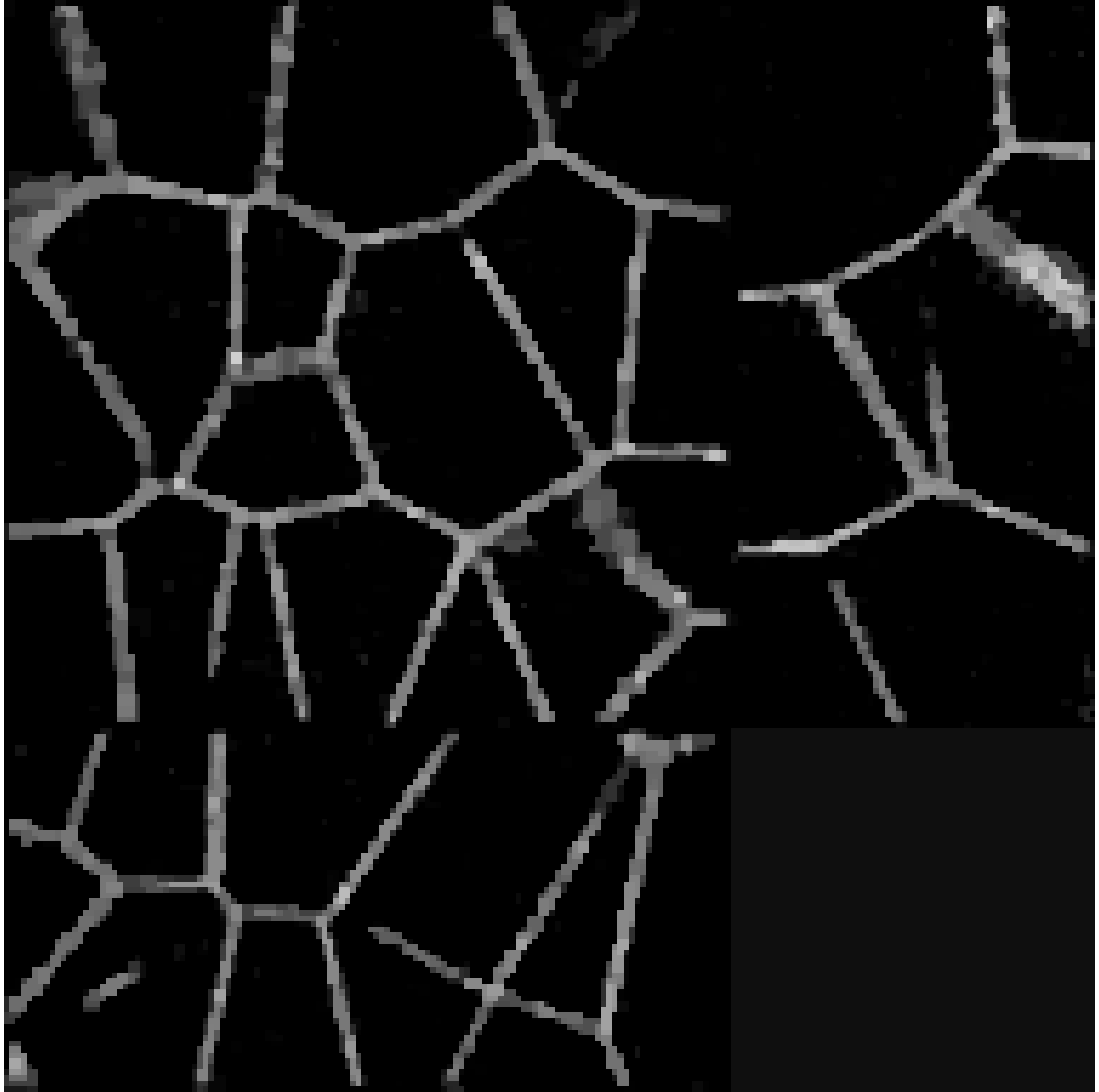}

  \caption{Reconstruction on simulated data with regularisation 
    parameter $\alpha$ such that best MSE 
    is achieved for each method and each image. Shown as 
    maximum intensity projections, except for tissue, 
    where slices in each direction in the centre of the
    sample are shown.
    \textbf{First row:} \ref{eq:psf l2}.
    \textbf{Second row:} \ref{eq:psf ic}.
    \textbf{Third row:} \ref{eq:ls l2}.
    \textbf{Fourth row:} LS-IC.}
  \label{fig:simulated results}
\end{figure}

\begin{figure}[H]
  \centering
  \includegraphics[width=0.3\textwidth]{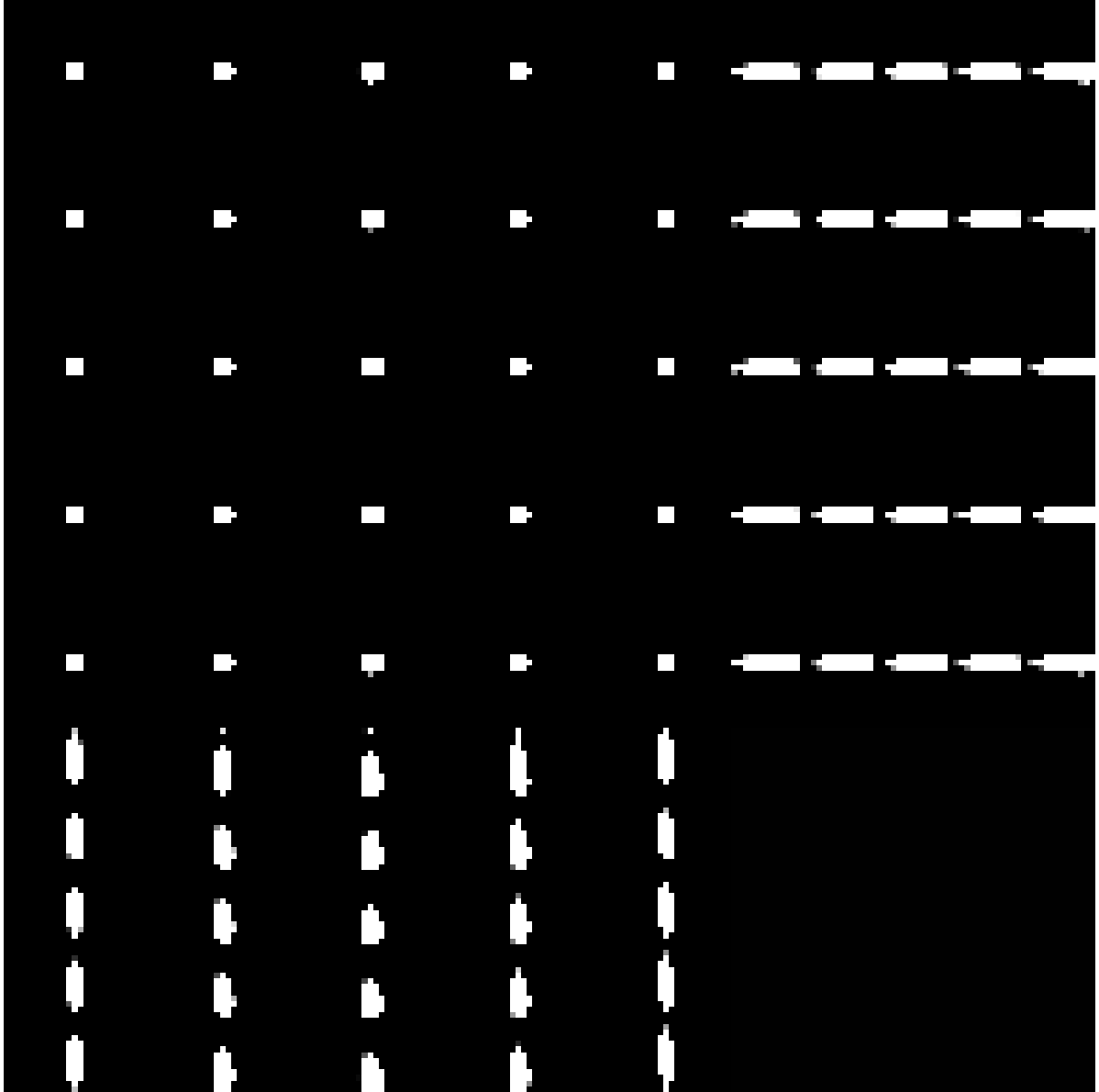}
  \includegraphics[width=0.3\textwidth]{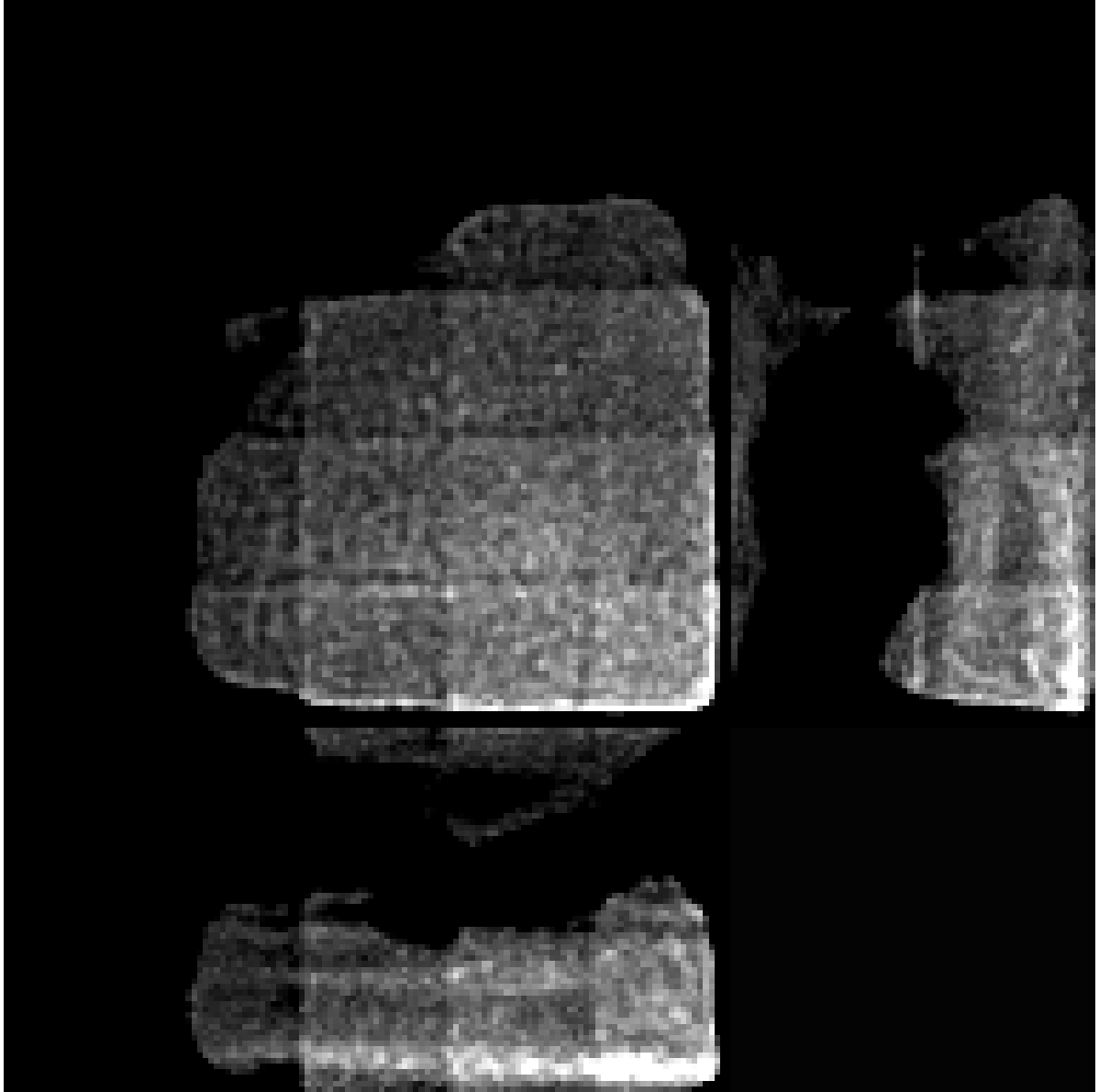}
  \includegraphics[width=0.3\textwidth]{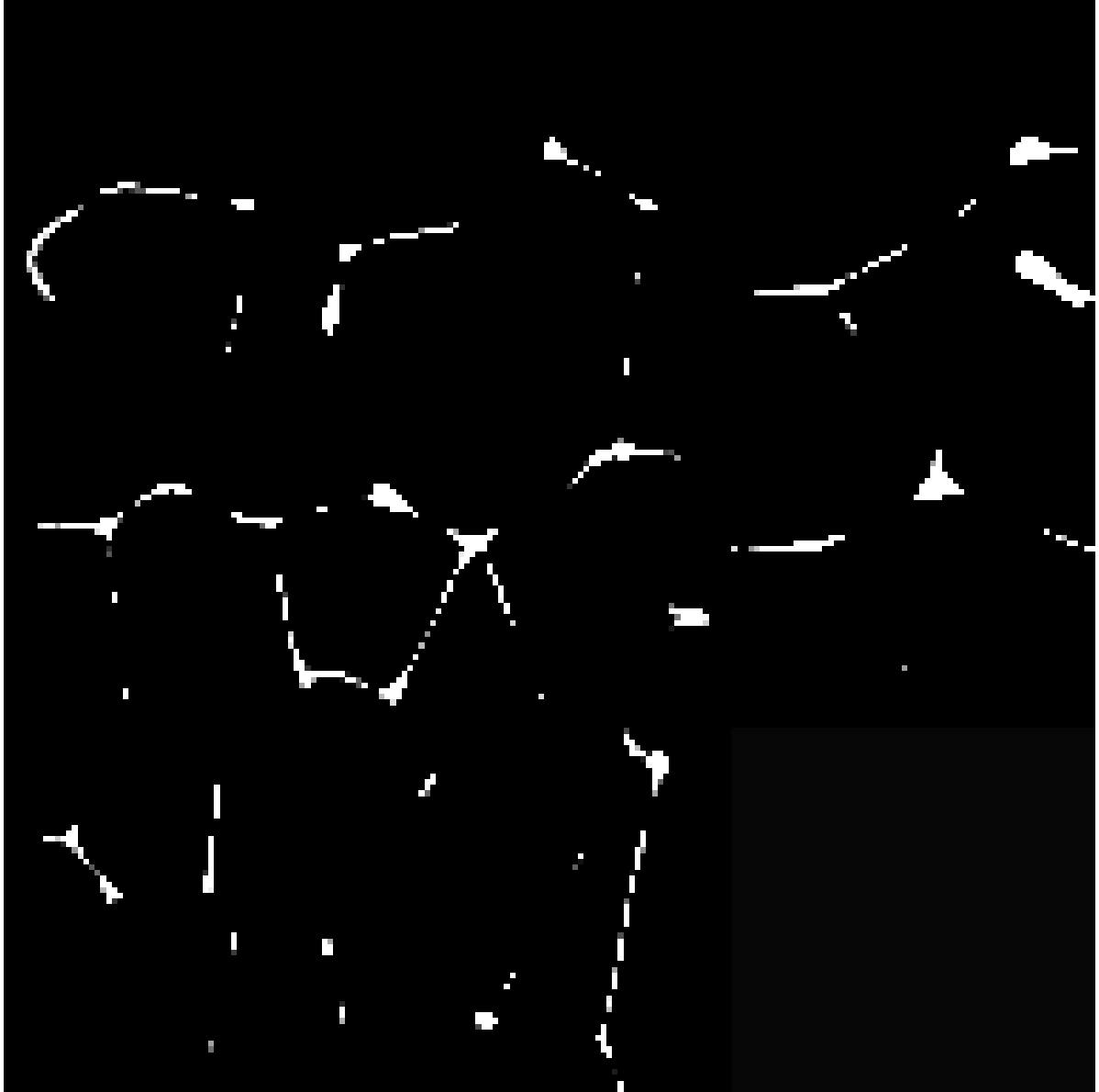}
  \includegraphics[width=0.3\textwidth]{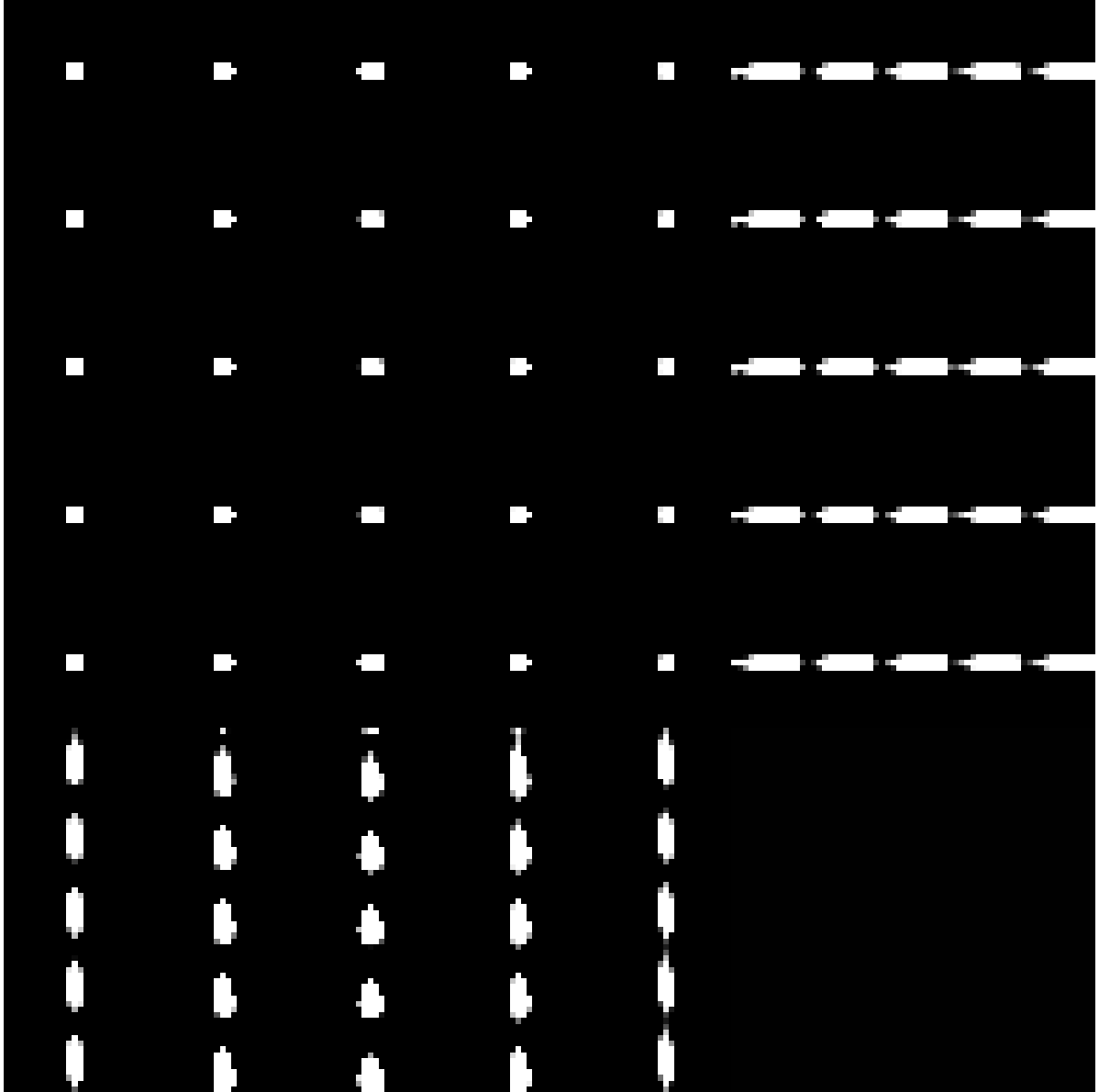}
  \includegraphics[width=0.3\textwidth]{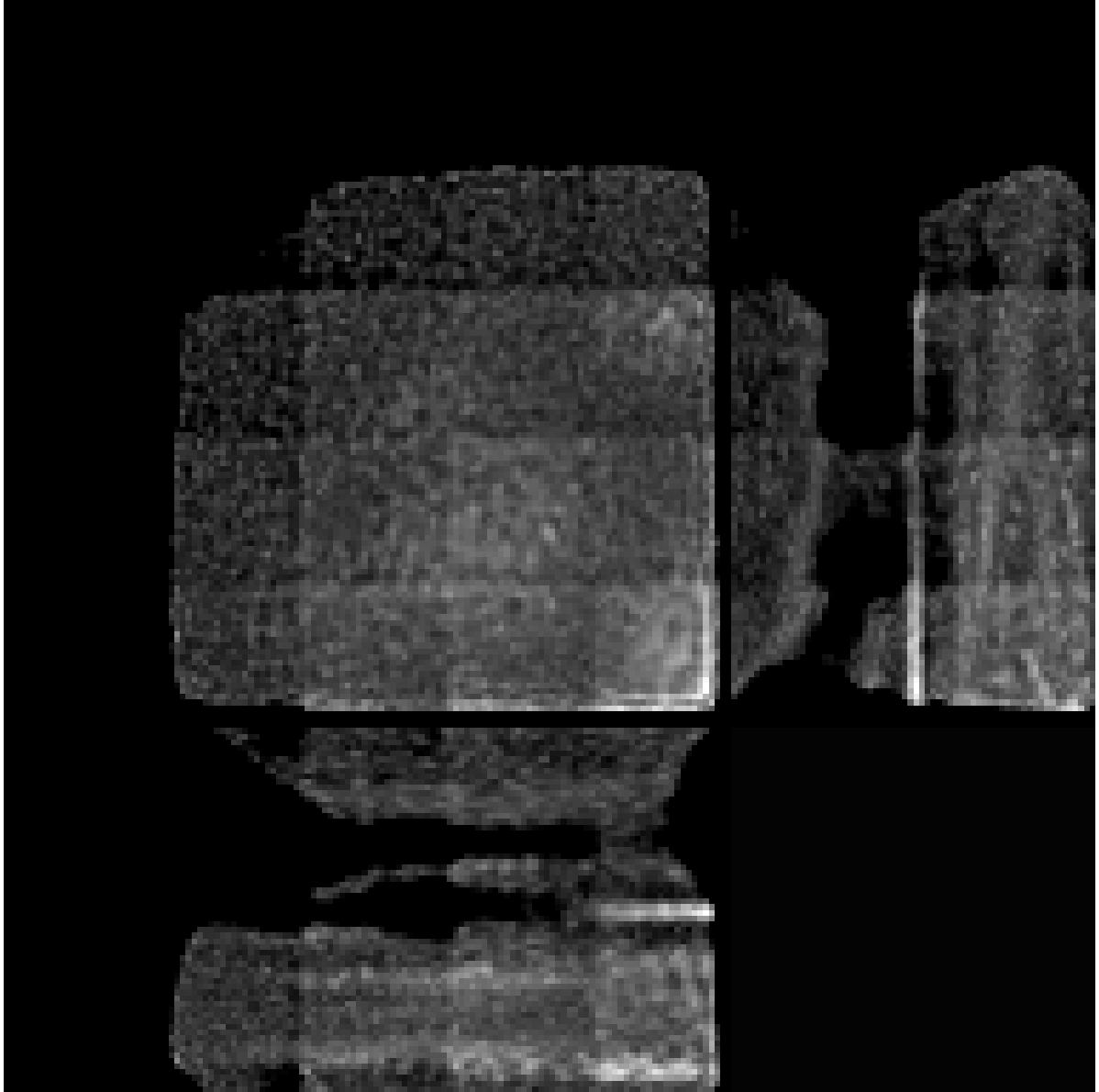}
  \includegraphics[width=0.3\textwidth]{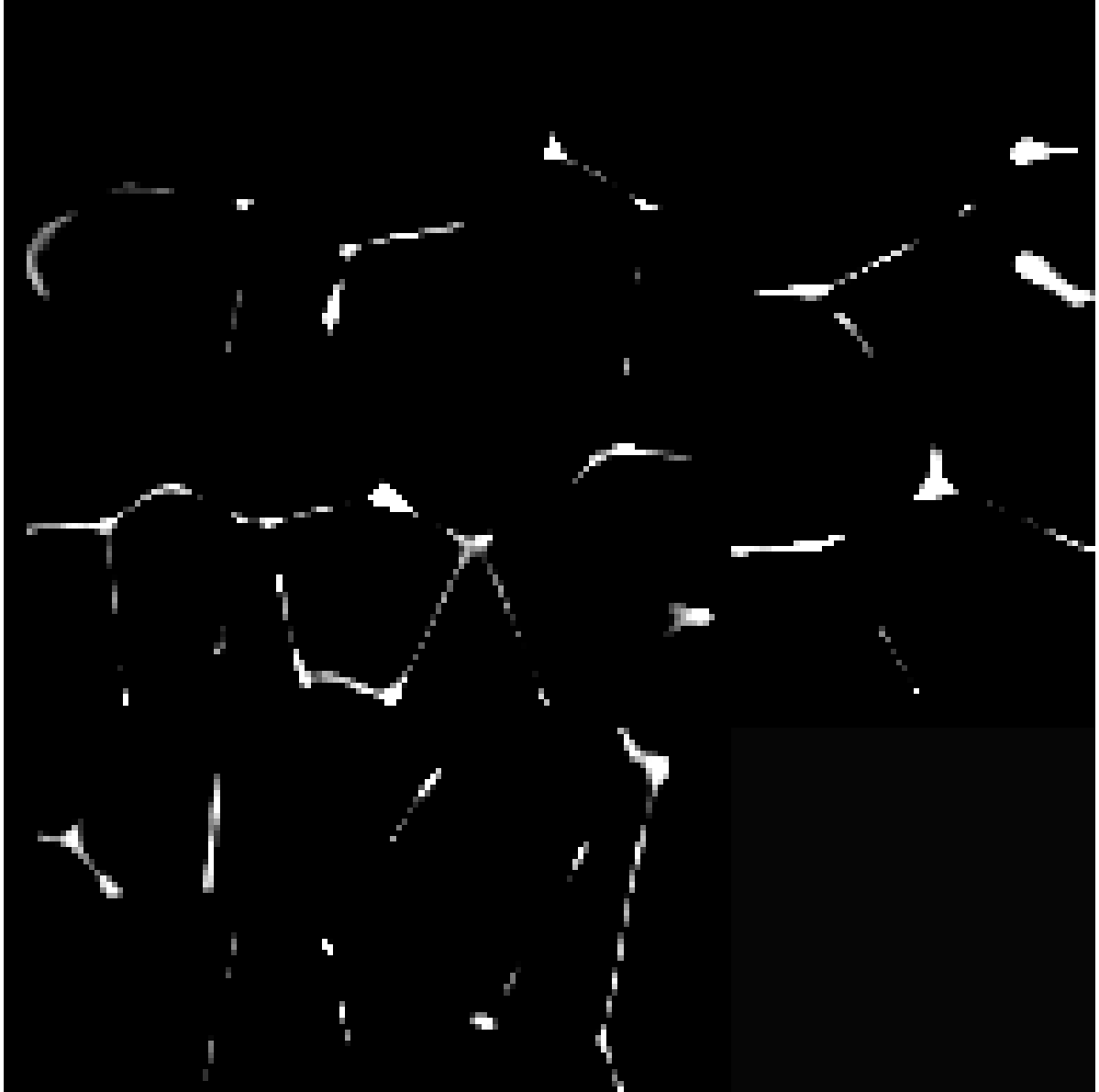}
  \includegraphics[width=0.3\textwidth]{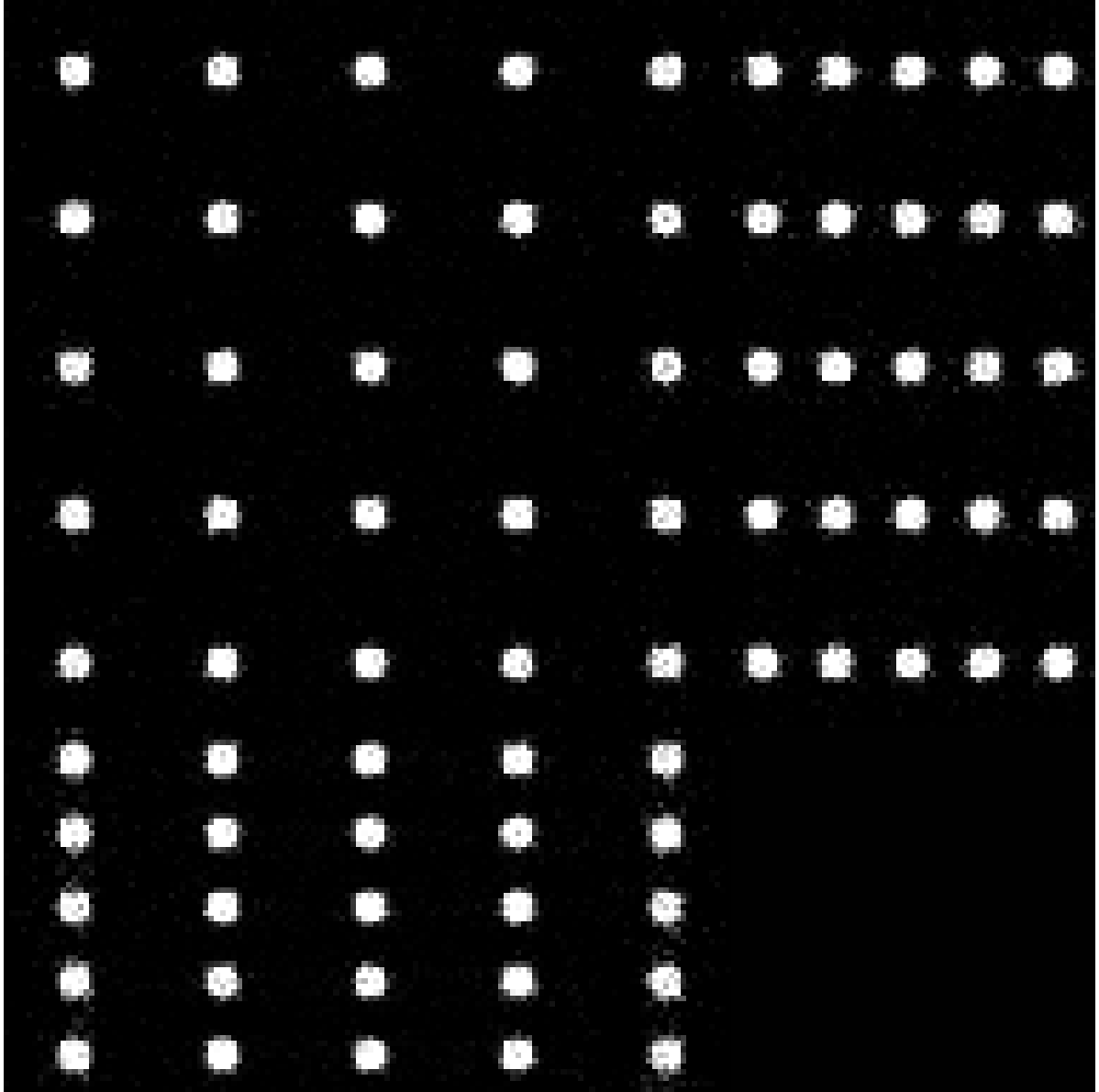}
  \includegraphics[width=0.3\textwidth]{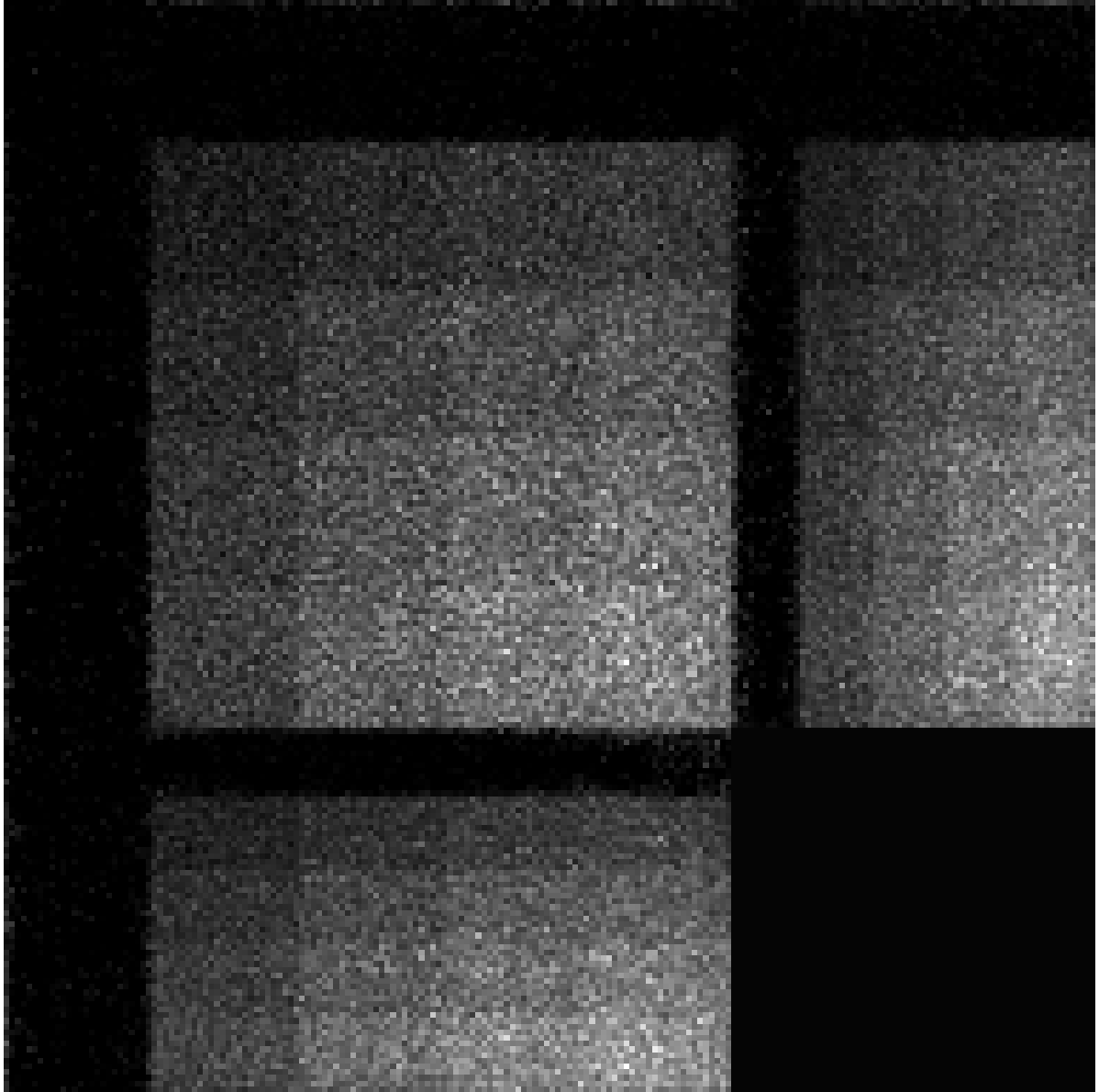}
  \includegraphics[width=0.3\textwidth]{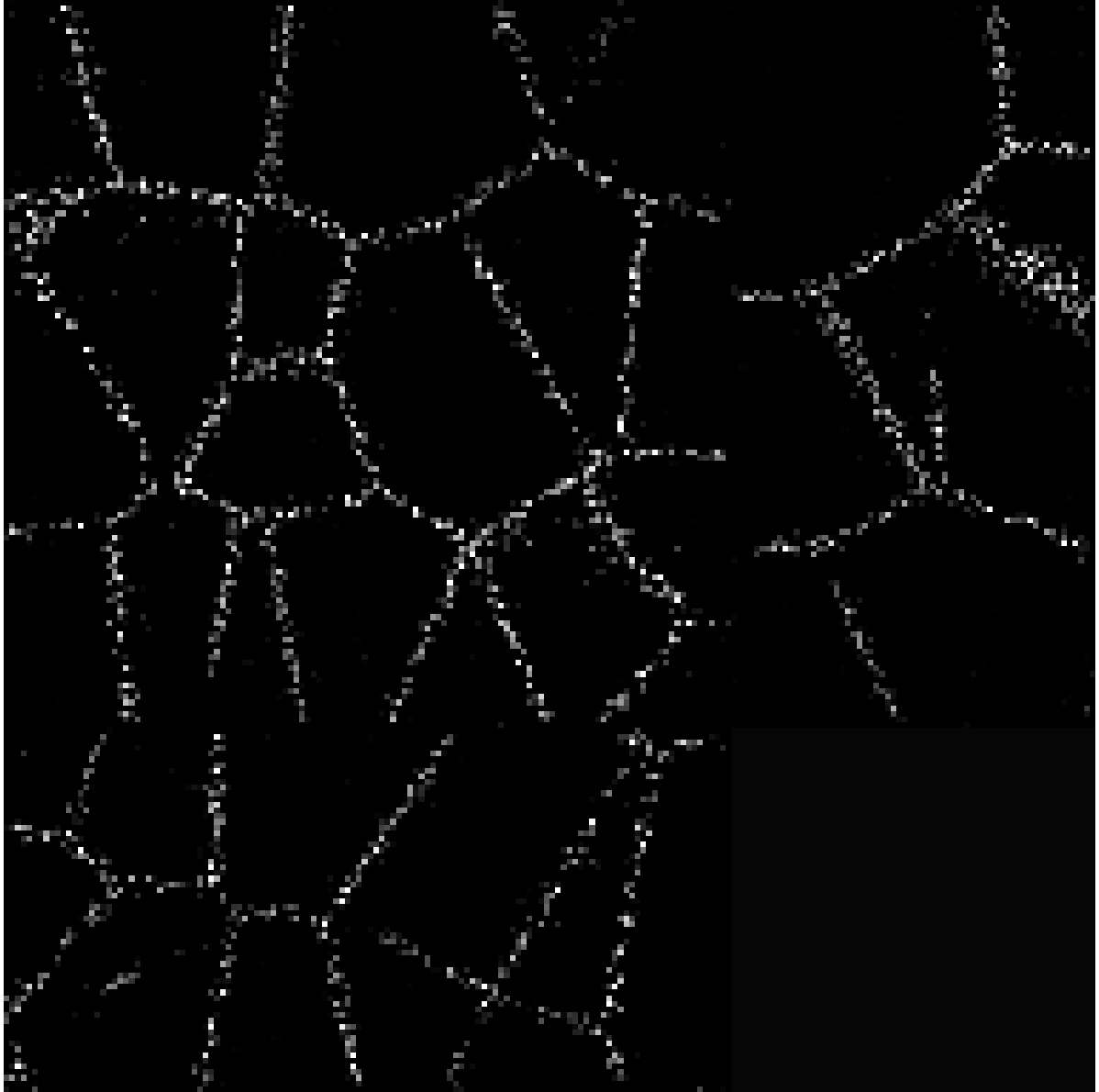}
  \includegraphics[width=0.3\textwidth]{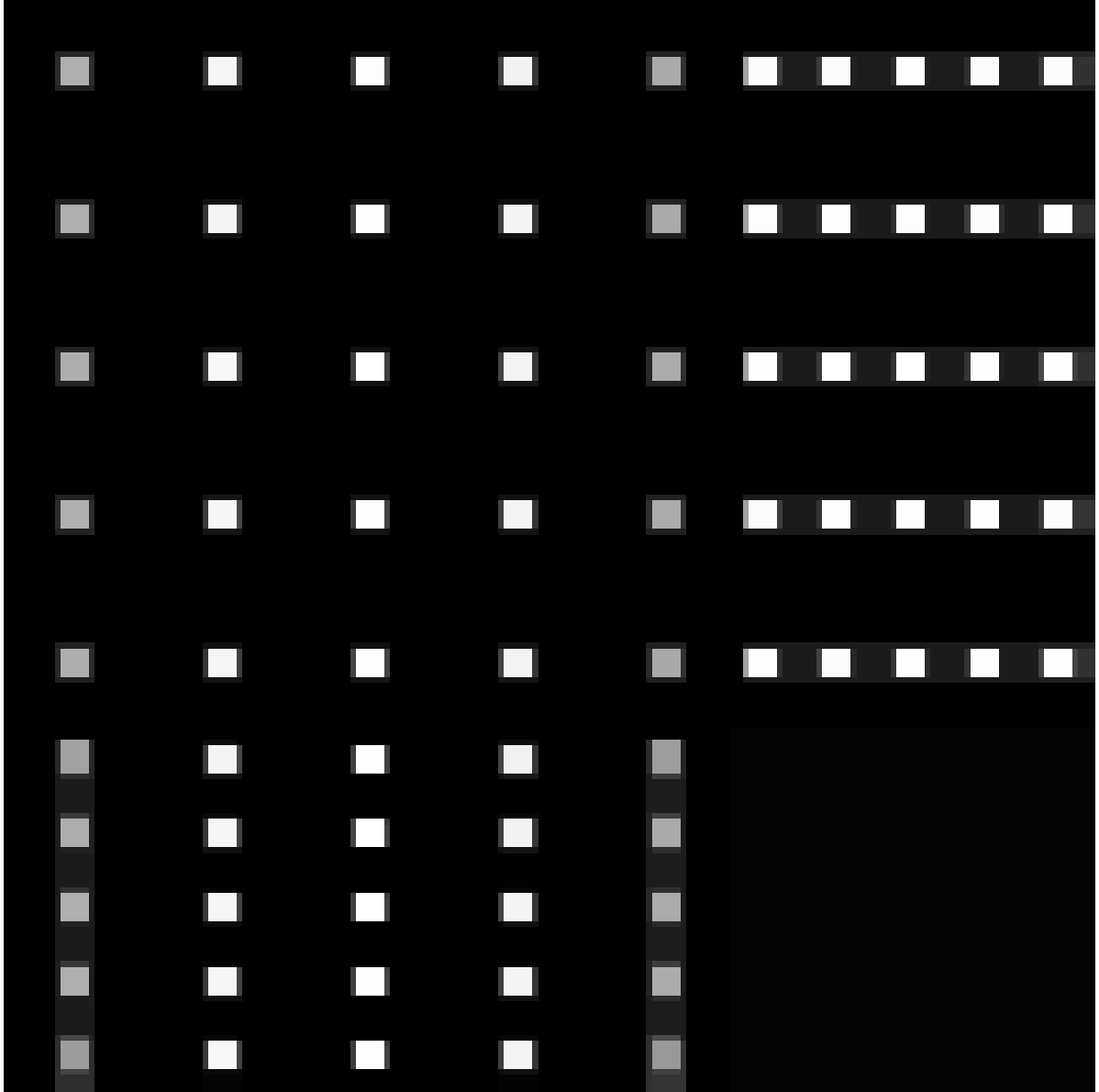}
  \includegraphics[width=0.3\textwidth]{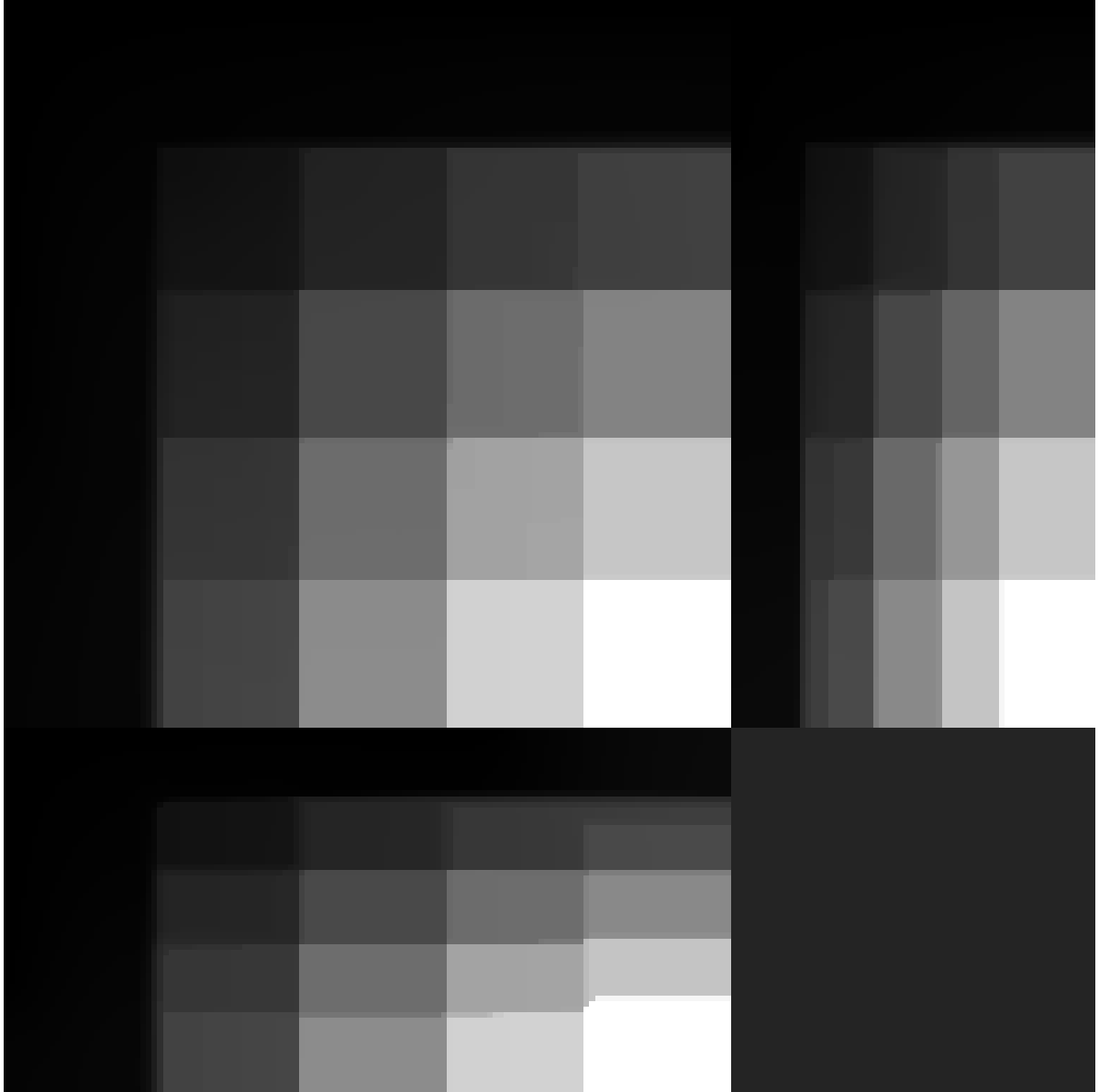}
  \includegraphics[width=0.3\textwidth]{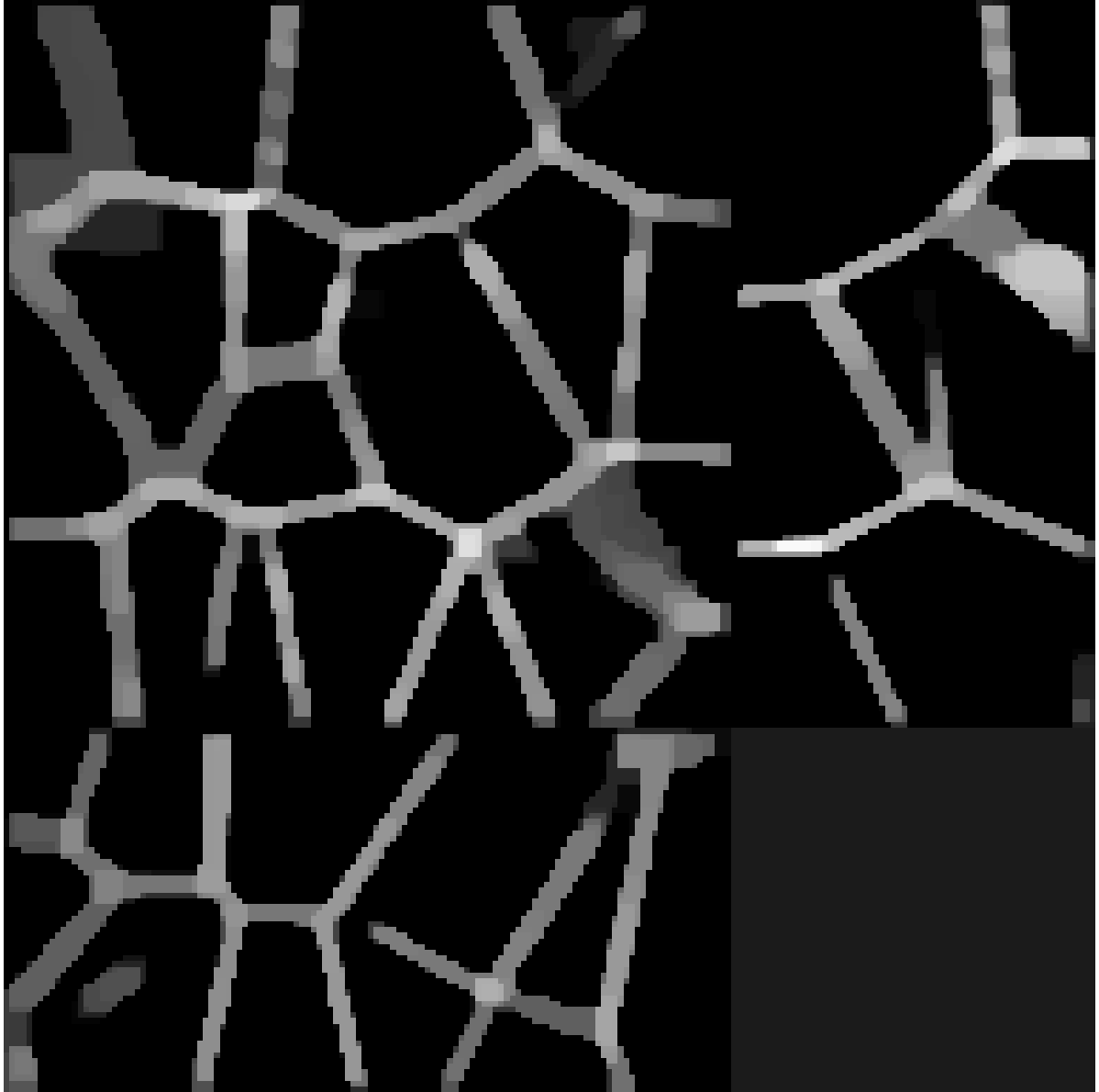}
  \caption{Reconstruction on simulated data with 
    regularisation parameter $\alpha$ chosen to satisfy 
    the discrepancy principle \eqref{eq:discr_pr1}.
    Shown as maximum intensity projections, except for tissue,
    where slices in each direction in the centre 
    of the sample are shown.
    \textbf{First row:} \ref{eq:psf l2}.
    \textbf{Second row:} \ref{eq:psf ic}.
    \textbf{Third row:} \ref{eq:ls l2}.
    \textbf{Fourth row:} LS-IC.}
  \label{fig:simulated discrepancy}
\end{figure}

  \subsection{Light-sheet data}
  \label{sec:results real data}

In this section, we show the results of applying 
LS-IC to a cropped portion of the
full resolution images in Figure \ref{fig:data example}.
Specifically, we select a cropped beads
image of $1127 \times 111 \times 100$ voxels
and a cropped Marchantia image 
of $1127 \times 156 \times 100$ voxels.

For comparison, we also run \ref{eq:psf l2} on the same
images. We run both methods on both images for up
to $6000$ iterations, with a normalised primal-dual gap
of $10^{-6}$ as a stopping criterion. The parameters for the image formation 
model used are the same as in Table~\ref{table:microscope params}
and the PDHG parameters are given 
in Table~\ref{table:pdhg params real}. 

The results of the deconvolution are shown in Figure~\ref{fig:real data beads}
and Figure~\ref{fig:real data marchantia} for the beads image 
and the Marchantia image respectively. In both figures,
we first show the position of the light-sheet in the first
row (due to the cropping, this is no longer centred), 
the measured data in the second row, followed by the \ref{eq:psf l2}  
reconstruction and the LS-IC reconstruction on
the third and fourth row respectively. The regularisation
parameter $\alpha$ was chosen in all four cases visually
such that
a balance is achieved between the amount of regularisation 
and the noise in the reconstruction.

In the beads image in Figure~\ref{fig:real data beads}, we 
note that the LS-IC performs better than \ref{eq:psf l2} 
at reversing the effect of the light-sheet. This is most
obvious in the $yz$ plane on the right-hand side of the
image, where the length of the beads in the $z$ direction
has been reduced to a greater extent than in 
the \ref{eq:psf l2} reconstruction.
In addition, the beads
appear less blurry in the LS-IC reconstruction
in the right-hand side of the $xy$ plane.
We show the bead images in 3D in Figure~\ref{fig:real data beads 3D},
where the effect of the deconvolution in the $z$ direction
is more significant in the LS-IC reconstruction than
in the \ref{eq:psf l2} reconstruction, namely the beads
are shorter in $z$.
In the Marchantia reconstruction in Figure~\ref{fig:real data marchantia},
we see a similar effect of better reconstruction in 
the $z$ direction, most easily seen in the right-hand side
and bottom projections. 
Moreover, the 3D rendering of the 
Marchantia sample in Figure~\ref{fig:real data marchantia 3d}
shows smoother cell edges in the LS-IC reconstruction,
while the \ref{eq:psf l2} reconstruction contains
reconstruction artefacts that are non-existent in the 
LS-IC reconstruction, indicated by the yellow arrows. 

\begin{table}[h]
  \centering
   \begin{tabular}{| c | c c | c c |}
     \hline
     method & LS-IC & & PSF-L2 & \\ 
     image & beads & Marchantia & beads & Marchantia \\
     \hline
     $\rho$ & 0.5 & 0.7 & 0.9 & 0.9 \\
     $\sigma$ & 0.0001 & 0.0001 & 0.01 & 0.001 \\
     \hline
   \end{tabular}
   \caption{Values of the PDHG parameters $\rho$ and $\sigma$
      used in the numerical experiments with real data.} 
   \label{table:pdhg params real}
\end{table}

\begin{figure}[H]
  \centering
  \includegraphics[width=\textwidth]{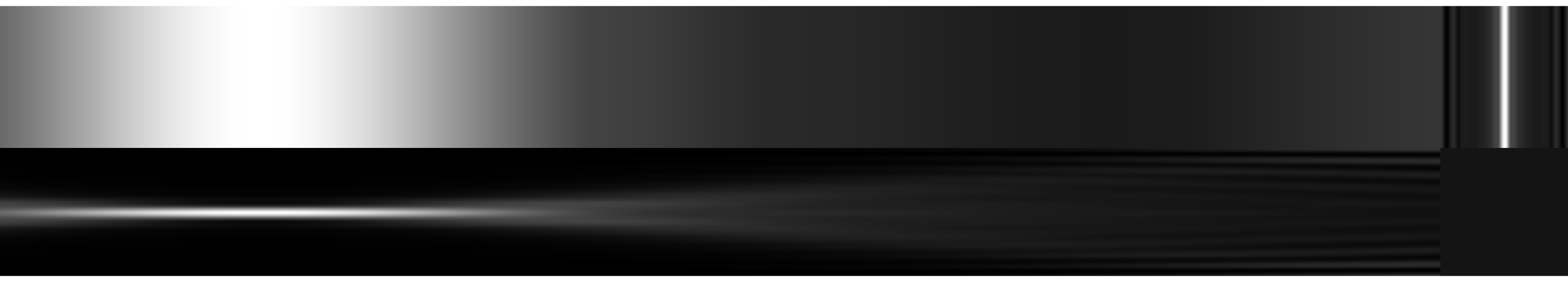}
  \includegraphics[width=\textwidth]{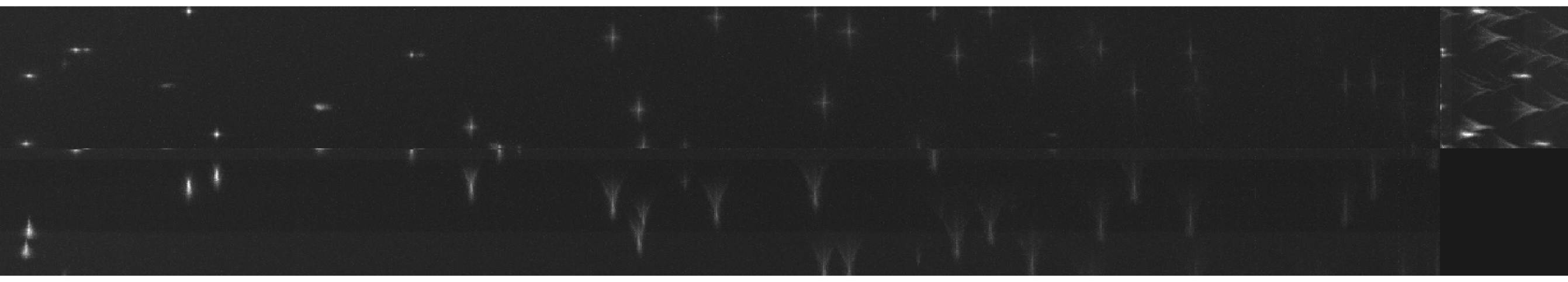}
  \includegraphics[width=\textwidth]{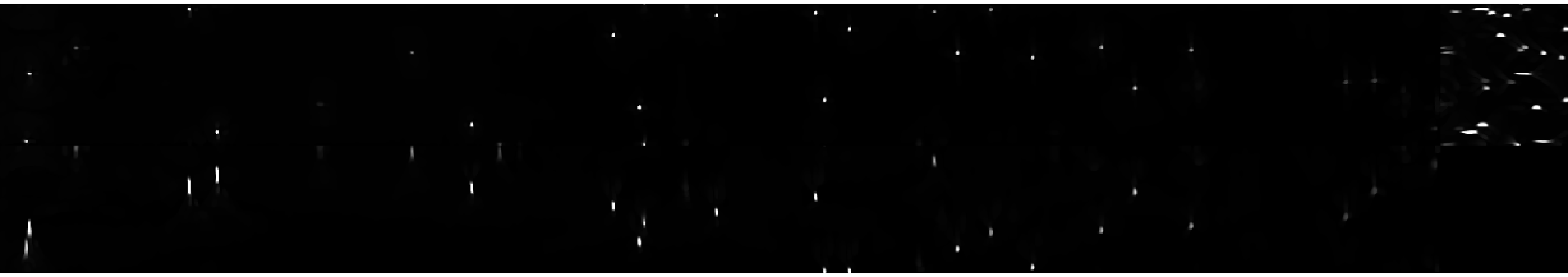}
  \includegraphics[width=\textwidth]{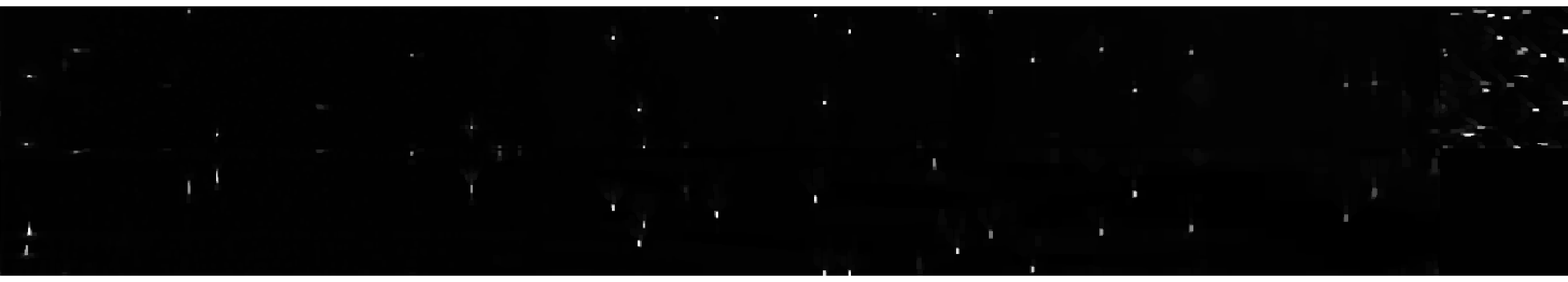}

  \caption{Reconstruction results for the light-sheet bead
    image, shown as maximum intensity projections.
    \textbf{First row:} The fitted light-sheet profile. 
    \textbf{Second row:} The data.
    \textbf{Third row:} PSF-L2 with $\alpha=0.1$.
    \textbf{Fourth row:} LS-IC with $\alpha=0.0046$.
  }
  \label{fig:real data beads}  
\end{figure}

\begin{figure}[H]
  \centering
  \includegraphics[width=\textwidth]{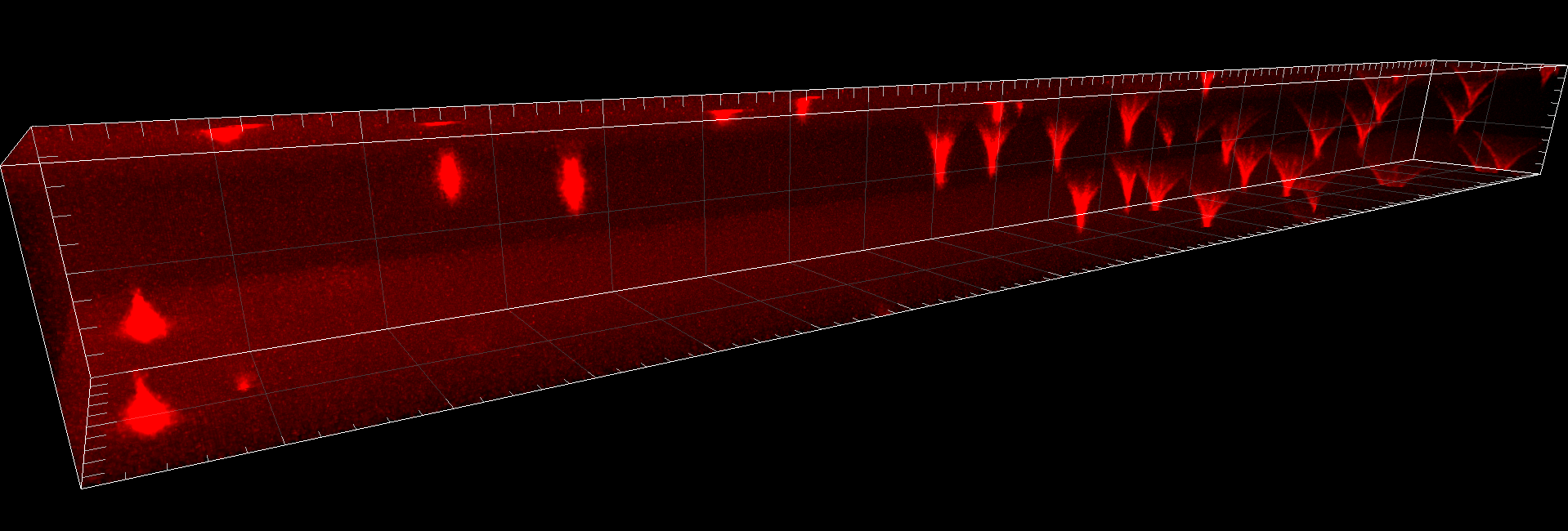}
 \includegraphics[width=\textwidth]{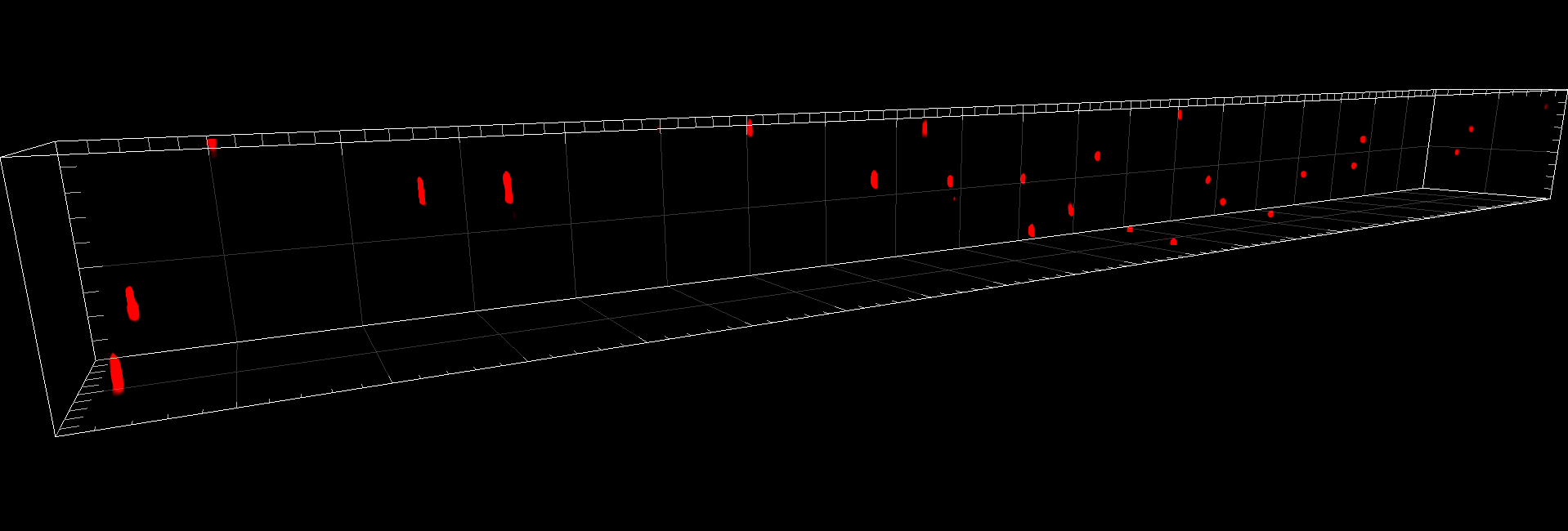}
 \includegraphics[width=\textwidth]{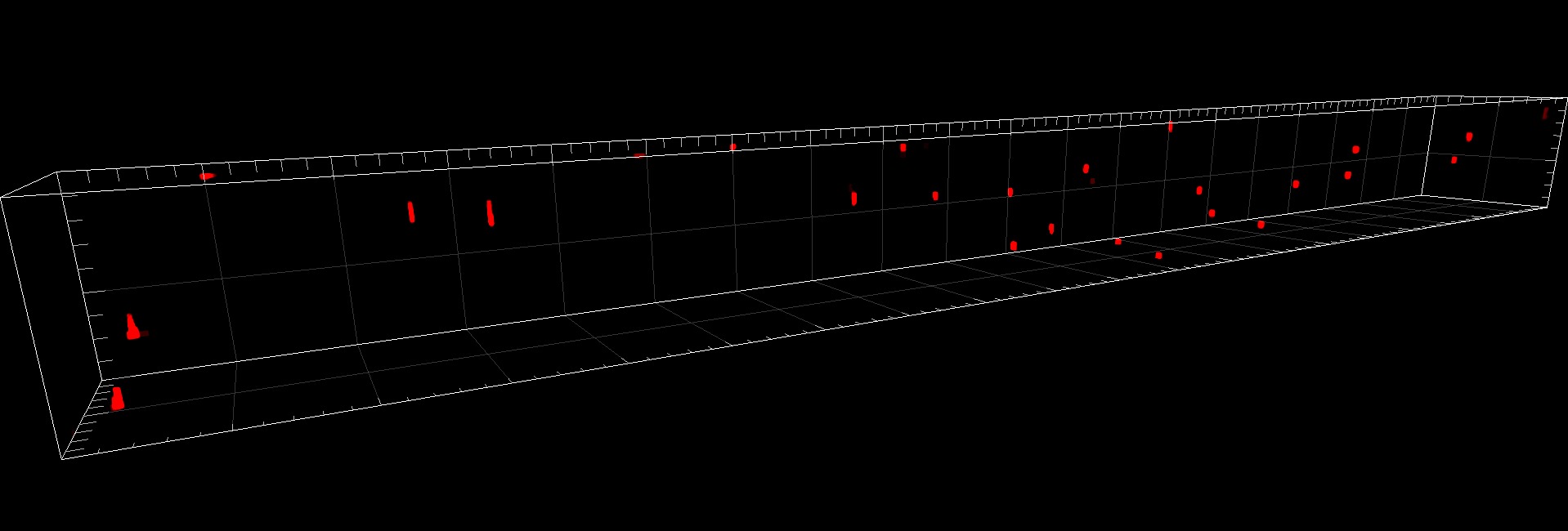}

  \caption{3D rendering of the beads data and reconstruction
    images using ImarisViewer 9.7.2.
    \textbf{First row:} The data.
    \textbf{Second row:} PSF-L2 with $\alpha=0.1$.
    \textbf{Third row:} LS-IC with $\alpha=0.0046$.
  }
  \label{fig:real data beads 3D}  
\end{figure}

\begin{figure}[H]
  \centering
  \includegraphics[width=\textwidth]{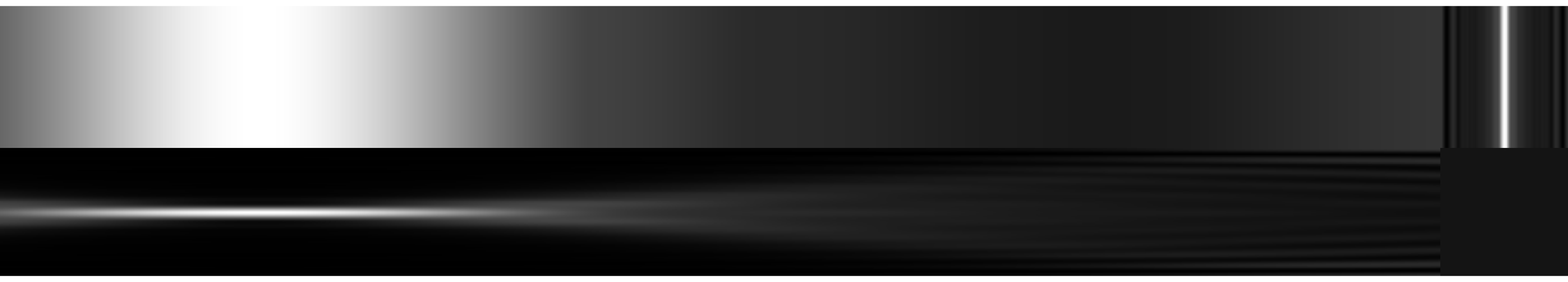}
  \includegraphics[width=\textwidth]{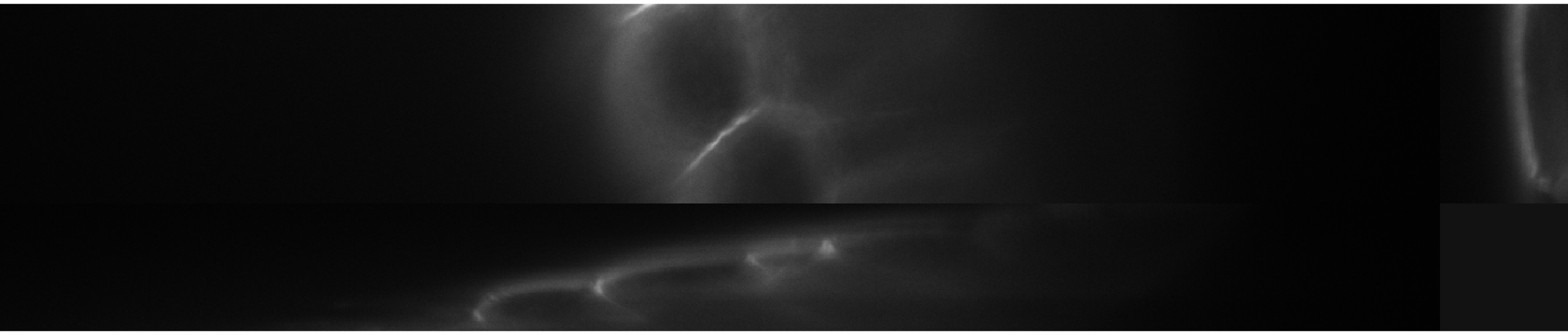}
  \includegraphics[width=\textwidth]{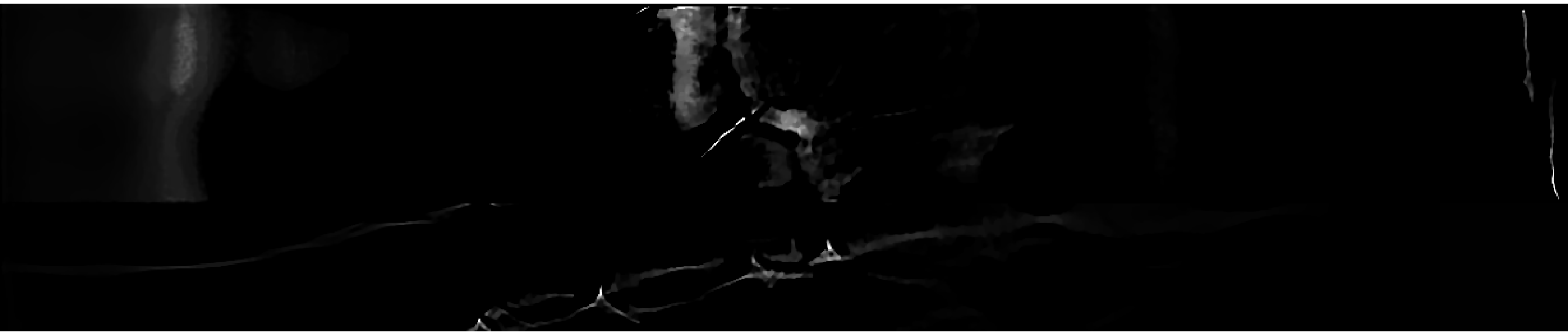}
  \includegraphics[width=\textwidth]{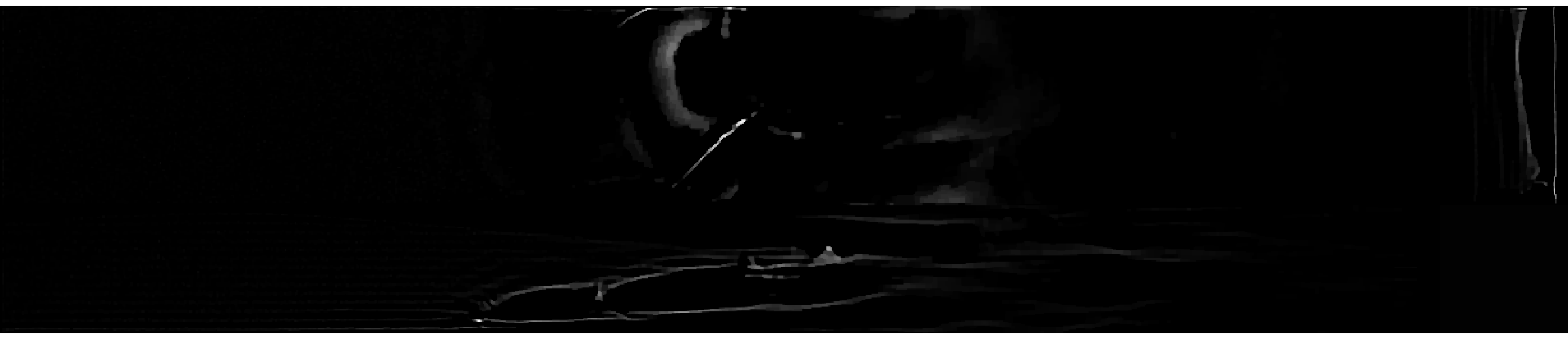}

  \caption{Reconstruction results for the Marchantia sample,
    shown as slices in each direction in the centre of the
    sample.
    \textbf{First row:} The fitted light-sheet profile. 
    \textbf{Second row:} The data.
    \textbf{Third row:} PSF-L2 with $\alpha=0.1$.
    \textbf{Fourth row:} LS-IC with $\alpha=0.0005$.
  }
  \label{fig:real data marchantia}  
\end{figure}

\begin{figure}[H]
  \centering
  \includegraphics[width=0.85\textwidth]{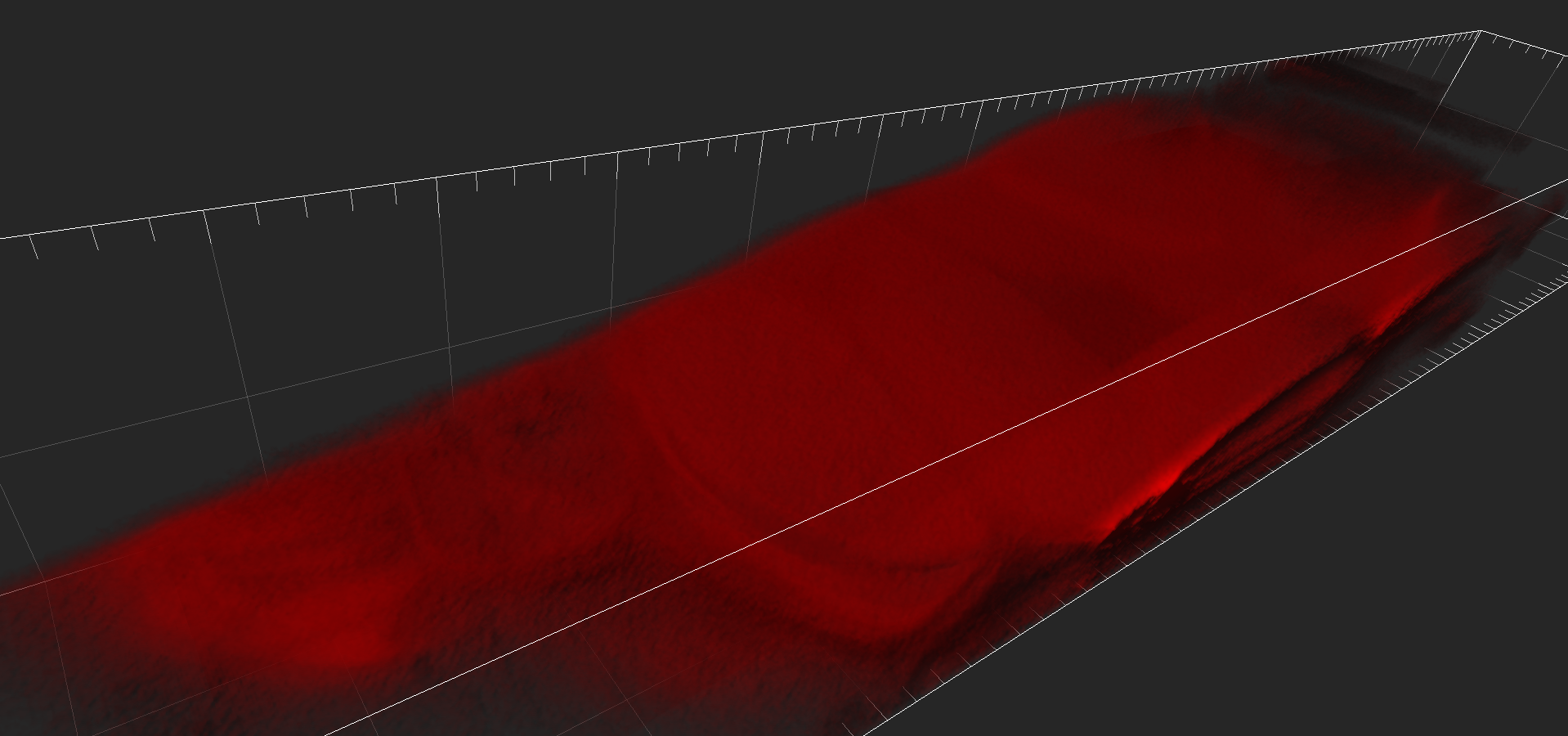}
  \includegraphics[width=0.85\textwidth]{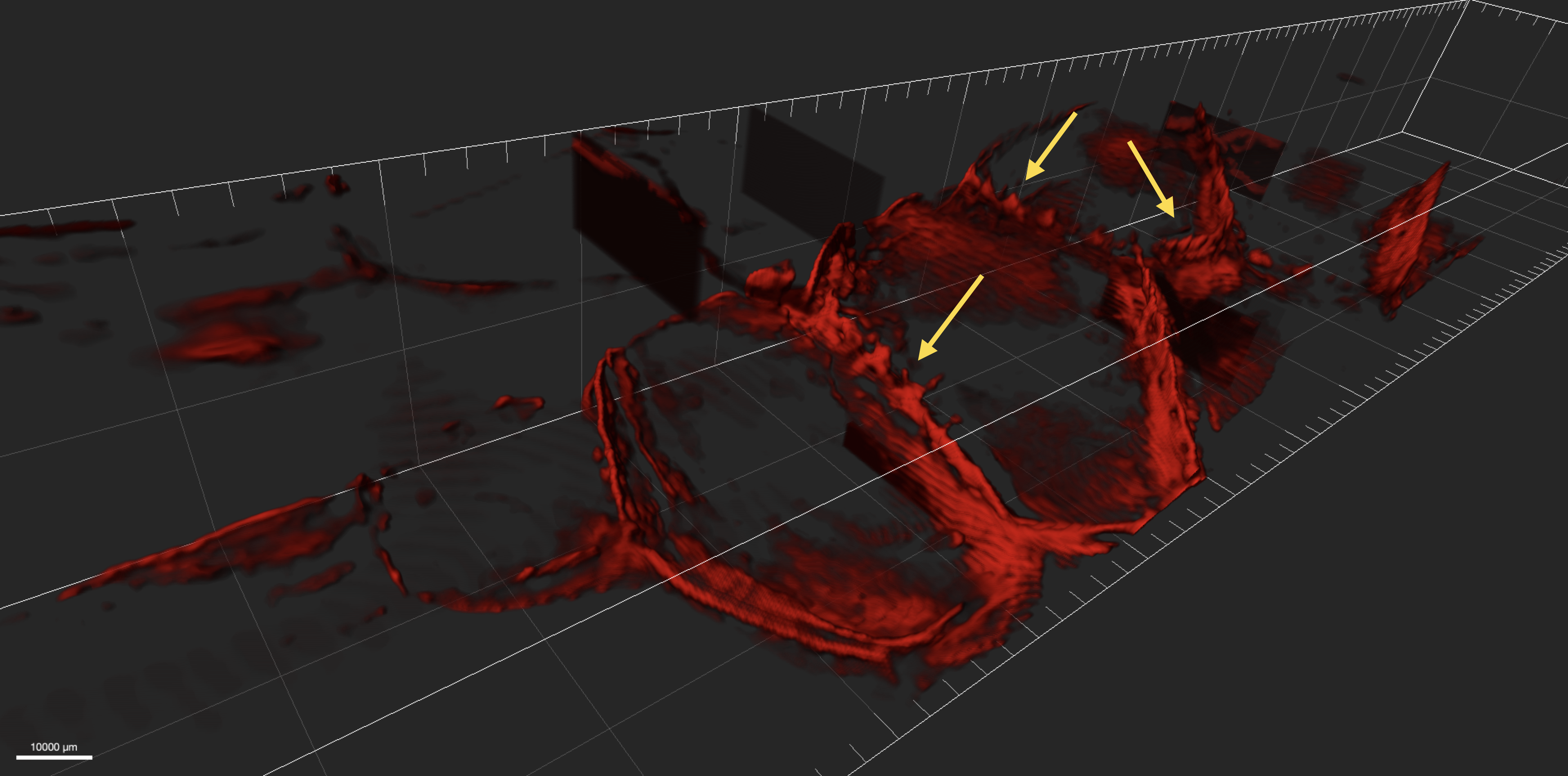}
  \includegraphics[width=0.85\textwidth]{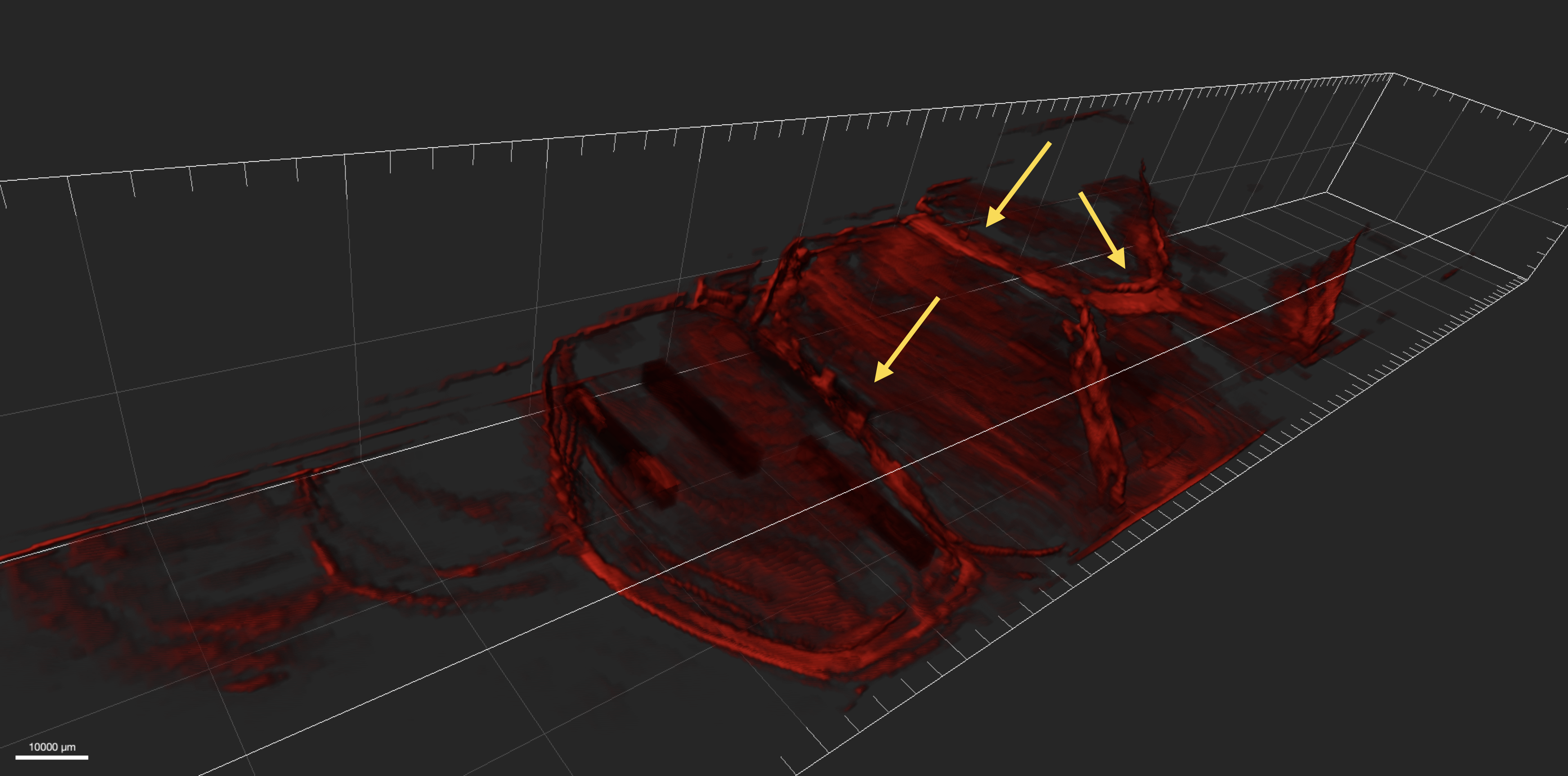}

  \caption{3D rendering of the Marchantia data and 
    reconstruction images using ImarisViewer 9.7.2.
    \textbf{First row:} The data.
    \textbf{Second row:} PSF-L2 with $\alpha=0.1$.
    \textbf{Third row:} LS-IC with $\alpha=0.0005$.
  }
  \label{fig:real data marchantia 3d}  
\end{figure}

%% file: 6_conclusion.tex
\section{Conclusion}
\label{sec:conclusion}

In this paper we introduced a novel method for performing
deconvolution for light-sheet microscopy. 
We start by modelling the image formation process in a way 
that replicates the physics of a light-sheet microscope, 
which is achieved by explicitly modelling the interaction 
of the illumination light-sheet and the detection 
objective PSF. Moreover, the optical aberrations in the 
system are modelled using a linear combination of
Zernike polynomials in the pupil function of the detection
PSF, fitted to bead data using a least squares procedure. 
We then formulate a variational
model taking into account the image formation model as the
forward operator and a combination of Poisson and Gaussian
noise in the data. The model combines a total variation
regularisation term and a fidelity term that is an infimal
convolution between an $\L^2$ term and the Kullback-Leibler
divergence, introduced in~\cite{calatroni2017infimal}.
In addition, we establish convergence rates with respect to 
the noise and we introduce a discrepancy principle for
selecting the regularisation parameter $\alpha$ in the mixed noise
setting. We solve the resulting inverse problem by applying the
PDHG algorithm in a non-trivial way.

The results in the numerical experiments section show that our
method, LS-IC, outperforms simpler approaches to deconvolution of
light-sheet microscopy data, where one does not take into
account the variability of the overall PSF introduced by the
light-sheet excitation, or the combination of Gaussian and Poisson
noise. In particular, numerical experiments with simulated
data show superior reconstruction quality in terms of the
normalised $l^2$ error and the structural similarity index, not
only by optimising over the regularisation parameter
$\alpha$ given the ground truth, but also with an a
posteriori choice of $\alpha$ using the stated discrepancy
principle. On bead data, the reconstruction obtained using 
LS-IC shows a more significant reduction of the blur in the
$z$ direction compared to \ref{eq:psf l2}, where the
light-sheet variations and the Poisson noise are not taken into
account. Moreover, reconstruction of a Marchantia
sample with LS-IC shows fewer artefacts than the
\ref{eq:psf l2} reconstruction, as well as sharper cell 
edges and smoother cell membranes.

Future work includes applying this technique to a broader
range of samples and using it to answer questions of
biological interest. To do so, we see a number of potential 
future directions that this work can take:
\begin{itemize}
  \item Adapting the discrepancy principle given in
    \eqref{eq:discr_pr1} for choosing the regularisation
    parameter $\alpha$ to real data sets, like the ones in
    Section~\ref{sec:results real data}.

  \item Improving the running time of the method potentially
    by means of randomised approaches.
  
  \item Investigating other regularisation terms.
  
  \item Making the technique available to other users as a
    more user-friendly tool.

\end{itemize}